\documentclass[12pt,draftcls,onecolumn]{IEEEtran}
\usepackage{cite}
\usepackage{url}
\usepackage{amsmath,amssymb,amsfonts}
\usepackage{multirow}
\usepackage{amsthm}
\usepackage{graphicx}
\usepackage{textcomp}
\usepackage{epsfig}
\usepackage{algorithm}
\usepackage{algorithmic}
\usepackage{epstopdf}
\usepackage{enumerate}
\usepackage{float}
\usepackage{graphics,graphicx}
\usepackage{subfigure}
\usepackage{array}
\usepackage{mathtools}
\usepackage{color}
\usepackage{mathtools}
\DeclareGraphicsExtensions{.png}
\graphicspath{{./Pics/}}
\usepackage{xspace}
\usepackage{tcolorbox}
\usepackage{dashbox}
\usepackage{bm} 
\usepackage{caption}
\usepackage{cite}
\usepackage{amsmath,amssymb,amsfonts}
\usepackage{algorithmic}
\usepackage{graphicx}
\usepackage{textcomp}

\usepackage{cite}
\usepackage{url}
\usepackage{amsmath,amssymb,amsfonts}
\usepackage{multirow}
\usepackage{algorithmic}
\usepackage{cleveref}
\usepackage{textcomp}
\usepackage{epsfig}
\usepackage{algorithmic}
\usepackage{epstopdf}
\usepackage{float}
\usepackage{caption}  
\usepackage{booktabs} 
\usepackage{pifont}
\usepackage{array}
\usepackage{mathtools}
\usepackage[justification=centering]{caption}

\renewcommand{\Re}{{\mathbb{R}}}
\newcommand{\bx}{{\mathbf{x}}}
\newcommand{\bB}{{\mathbf{B}}}

\newcommand{\bG}{{\mathbf{G}}}
\newcommand{\bL}{{\mathbf{L}}}

\newcommand{\bq}{{\mathbf{q}}}
\newcommand{\bI}{{\mathbf{I}}}
\newcommand{\bK}{{\mathbf{K}}}

\newcommand{\bg}{{\mathbf{g}}}
\newcommand{\bs}{{\mathbf{s}}}
\newcommand{\bH}{{\mathbf{H}}}
\newcommand{\bD}{{\mathbf{D}}}
\newcommand{\bW}{{\mathbf{W}}}

\newtheorem{lemma}{Lemma}
\newtheorem{theorem}{Theorem}

\newtheorem{remark}{Remark}
\newtheorem{assumption}{Assumption}
\newtheorem{proposition}{Proposition}

\usepackage[justification=centering]{caption}
\allowdisplaybreaks[4]
\begin{document}
	
	\title{\LARGE \bf A Unifying Primal-Dual Proximal Framework for Distributed Nonconvex Optimization}
\author{Zichong Ou and Jie Lu
\thanks{Zichong Ou and Jie Lu are with the School of Information Science and Technology, Shanghaitech University, 201210 Shanghai, China. J. Lu is also with the Shanghai Engineering Research Center of Energy Efficient and Custon AI IC, 201210 Shanghai, China. Email: {\tt\small ouzch, lujie@shanghaitech.edu.cn}.}
}
	\maketitle
	
	\begin{abstract}
We consider distributed nonconvex optimization over an undirected network, where each node privately possesses its local objective and communicates exclusively with its neighboring nodes, striving to collectively achieve a common optimal solution. To handle the nonconvexity of the objective, we linearize the augmented Lagrangian function and introduce a time-varying proximal term. This approach leads to a Unifying Primal-Dual Proximal (UPP) framework that unifies a variety of existing first-order and second-order methods. Building on this framework, we further derive two specialized realizations with different communication strategies, namely UPP-MC and UPP-SC. We prove that both UPP-MC and UPP-SC achieve stationary solutions for nonconvex smooth problems at a sublinear rate. Furthermore, under the additional Polyak-{\L}ojasiewics (P-{\L}) condition, UPP-MC is linearly convergent to the global optimum. These convergence results provide new or improved guarantees for many existing methods that can be viewed as specializations of UPP-MC or UPP-SC. To further optimize the mixing process, we incorporate Chebyshev acceleration into UPP-SC, resulting in UPP-SC-OPT, which attains an optimal communication complexity bound. Extensive experiments across diverse network topologies demonstrate that our proposed algorithms outperform state-of-the-art methods in both convergence speed and communication efficiency.
\end{abstract}


	
\section{Introduction}
\label{section: introduction}

This article delves into minimizing a sum-utility function $f$, formulated as:
\begin{equation}
\min _{x \in \mathbb{R}^{d}} f(x)= \sum_{i=1}^{N} f_{i}(x), \label{p1}
\end{equation}
where each $f_i:\mathbb{R}^d \rightarrow \mathbb{R}$ is continuously differentiable and the global objective $f(x)$ is $\bar{M}$-smooth for some $\bar{M}>0$. In the realm of distributed optimization, we consider solving this problem over an $N$-node, connected multi-agent system, where each node $i$ retains its local objective $f_i$ and all the nodes cooperatively solve the minimization problem~\eqref{p1}. 

When the local objectives $f_i$ possess convexity, a multitude of distributed first-order algorithms have been devised, including the gradient-based methods \cite{shi_extra_2015,AugDGM2015,nedic_achieving_2017} and primal-dual methods \cite{xu_bregman_2018,ADMMconvex2014,ADMMconvex2017}. Specifically, EXTRA \cite{shi_extra_2015} ensures sublinear convergence for convex objectives and linear convergence under strong convexity. Aug-DGM \cite{AugDGM2015} and DIGing \cite{nedic_achieving_2017}—both leveraging gradient tracking techniques—also achieve analogous convergence properties. Moreover, ID-FPPS \cite{xu_bregman_2018} relaxes the consensus constraint by addressing an Augmented Lagrangian (AL) dual problem. Concurrently, methods such as \cite{ADMMconvex2014,ADMMconvex2017} extend the Alternating Direction Method of Multipliers (ADMM) to the distributed consensus setting. To systematically analyze these algorithms, unifying primal-dual frameworks such as PUDA \cite{alghunaim_decentralized_2021} and DAMM \cite{Wu2020AUA} are emerged, which generalize the aforementioned first-order methods including \cite{shi_extra_2015,AugDGM2015,nedic_achieving_2017,xu_bregman_2018}. In particular, DAMM employs time-varying surrogate functions and broadens its coverage by additionally encompassing second-order methods including \cite{mokhtari_dqm_2016,wu_second-order_2021}. 

While convex optimization provides theoretical guarantees, modeling real-world problems such as large-scale machine learning \cite{NIPS2015_452bf208}, empirical risk minimization \cite{doi:10.1137/16M1080173}, and multi-robot surveillance \cite{CARNEVALE2024111767} using convex formulations risks oversimplifying inherent nonconvex structures. To address this limitation, a number of distributed nonconvex optimization algorithms \cite{hong2016convergence,hong_distributed_2018,mancino2023decentralized,yi2022sublinear,yi2021linear,alghunaim_unified_2022,hong_prox-pda_2017,sun2018distributed,sun_distributed_2019} have been developed to tackle nonconvex objectives. In particular, \cite{hong2016convergence} demonstrates that the classical ADMM converges at a sublinear rate of $\mathcal{O}(1/T)$ (where $T$ denotes the total number of iterations) to a stationary solution under a master-slave network architecture. Subsequent advancements extend ADMM to asynchronous implementations \cite{hong_distributed_2018} and generalize it to arbitrary network topologies via inexact local updates \cite{mancino2023decentralized} or primal linearization of local objectives \cite{yi2022sublinear}. On the other hand, AL-type primal-dual methods iteratively minimize AL functions through strongly convex surrogate approximations \cite{alghunaim_unified_2022,hong_prox-pda_2017,sun2018distributed,sun_distributed_2019}, all achieving a canonical sublinear convergence rate $\mathcal{O}(1/T)$ to the stationarity. Notably, under the Polyak-{\L}ojasiewicz (P-{\L}) condition—a relaxation of strong convexity—several algorithms \cite{yi2022sublinear,yi2021linear,alghunaim_unified_2022} attain linear convergence to global optima.

Most of the aforementioned algorithms \cite{hong2016convergence,hong_distributed_2018,yi2022sublinear,yi2021linear,hong_prox-pda_2017} adopt one-communication-one-computation update paradigm, where each local computation is followed by a communication round. This sequential interaction pattern inherently creates computational bottlenecks, particularly in large-scale or sparse networks, where communication efficiency dominates execution time. To characterize the influence of network topology on algorithmic performance, we define $\gamma$ as the condition number of the graph Laplacian matrix—the ratio between its largest and smallest non-zero eigenvalues. Notably, $\gamma$ increases with sparser network topologies, thereby further exacerbating communication inefficiencies. To address this critical limitation and improve communication efficiency, Chebyshev acceleration has emerged as a powerful technique for reducing communication overhead \cite{arioli2014chebyshev}. By leveraging Chebyshev polynomials, \cite{scaman2017optimal} and \cite{xu2020accelerated} develop communication-optimal algorithms that achieve the lower communication complexity bounds for strongly convex smooth and convex smooth problems, respectively. For nonconvex smooth objectives, recent advances such as ADAPD-OG-MC \cite{mancino2023decentralized} and xFILTER \cite{sun_distributed_2019} are shown to reach an $\epsilon$-stationary solution within $\mathcal{O}(\bar{M}\sqrt{\gamma}/\epsilon)$ communication rounds—a complexity bound proven to be optimal for first-order algorithms that merely transmit local decisions over the network.

In this paper, we propose a \emph{\underline{U}nifying \underline{P}rimal-dual \underline{P}roximal} (UPP) framework for distributed nonconvex optimization. The framework innovatively combines three key components: ($\romannumeral1$) a first-order approximation of the AL function with consensus penalties in primal updates; ($\romannumeral2$) a time-varying proximal term strategically designed to enable mixing acceleration; and ($\romannumeral3$) flexible dual ascent mechanisms with tunable parameters to ensure algorithmic generality. Through distinct parameter configurations, UPP systematically generates two distributed realizations: UPP-MC (\underline{M}ulti-inner-loop \underline{C}ommunication) and UPP-SC (\underline{S}ingle-inner-loop \underline{C}ommunication). The main contributions of this work are summarized as follows:
\begin{enumerate}
    \item[1)] We construct UPP as a unifying framework that encompasses diverse distributed optimization algorithms. By appropriately selecting parameters, UPP-MC subsumes a collection of first-order algorithms, including distributed convex optimization methods (e.g., \cite{shi_extra_2015,AugDGM2015,nedic_achieving_2017,xu_bregman_2018}) and advanced nonconvex optimization algorithms (e.g., \cite{yi2022sublinear,yi2021linear,alghunaim_unified_2022,hong_prox-pda_2017,sun2018distributed}). Moreover, UPP-SC is able to generalize a number of second-order algorithms (e.g., \cite{mokhtari_dqm_2016,wu_second-order_2021}).
    \item[2)] We show that both UPP-MC and UPP-SC achieve $\mathcal{O}(1/T)$ sublinear convergence rates toward stationarity for nonconvex smooth problems. This extends the convergence guarantees for convex optimization methods in \cite{shi_extra_2015,AugDGM2015,xu_bregman_2018,mokhtari_dqm_2016,wu_second-order_2021} to handle \textit{nonconvex} problems. Additionally, under the P-{\L} condition, UPP-MC achieves linear convergence to the global optimum, which establishes new theoretical results for the existing nonconvex methods in \cite{alghunaim_unified_2022,hong_prox-pda_2017,sun2018distributed}.
    \item[3)] We incorporate Chebyshev acceleration into UPP-SC, yielding its accelerated version UPP-SC-OPT. It improves the convergence guarantees with tighter dependency on both the smoothness parameter $\bar{M}$ and the graph Laplacian condition number $\gamma$. Specifically, UPP-SC-OPT attains an $\epsilon$-stationary solution within $\mathcal{O}(\bar{M}\sqrt{\gamma}/\epsilon)$ communication rounds—a significant improvement over the existing methods such as \cite{yi2022sublinear, yi2021linear,alghunaim_unified_2022,hong_prox-pda_2017}. This aligns with the optimal communication complexity bound established in \cite{sun_distributed_2019, mancino2023decentralized} for nonconvex, smooth problems. A comprehensive comparison is provided in Table~\ref{tab:convergence rate}. Compared to the communication-optimal algorithms in \cite{sun_distributed_2019, mancino2023decentralized}, UPP-SC-OPT possesses fewer parameters and variables, as well as a more concise update form.
    \item[4)] We evaluate the practical performance of our proposed algorithms through numerical experiments. Our simulations are conducted on graphs with varying sparsity levels, showcasing the superior performance of our algorithms in terms of both convergence speed and communication cost against the state-of-the-art methods.
\end{enumerate}

The rest of this article is organized as follows: Section~\ref{section: problem formulation} reformulates the nonconvex optimization problem, Section~\ref{section: algorithm development} describes the development of UPP, and its specializations are exhibited in Section~\ref{section: specializations}. Section~\ref{section: convergence analysis} provides the convergence rates and communication complexity bounds for our proposed methods. Section~\ref{section: numerical results} presents the numerical experiments. Finally, Section~\ref{section: conclusion} concludes this article.

\begin{table}[t]
    \centering
    \begin{tabular}{ccc}
        \toprule
        \multirow{2}{*}{Method} & Convergence & Communication \\
        & rate & complexity \\
        \midrule
        L-ADMM \cite{yi2022sublinear} & $\mathcal{O}(\bar{M}^2\gamma^3/T)$ & $\mathcal{O}(\bar{M}^2\gamma^3/\epsilon)$ \\
        PFOPDA \cite{yi2021linear} & $\mathcal{O}(\bar{M}^2\gamma^3/ T)$& $\mathcal{O}(\bar{M}^2\gamma^3/ \epsilon)$ \\
        SUDA \cite{alghunaim_unified_2022} & $\mathcal{O}(\bar{M}^4\gamma^3/T)$ & $\mathcal{O}(\bar{M}^4\gamma^3/\epsilon)$ \\
        Prox-PDA \cite{hong_prox-pda_2017} & $\mathcal{O}(\bar{M}^2\gamma^2/T)$ & $\mathcal{O}(\bar{M}^2\gamma^2/\epsilon)$\\
        ADAPD-OG \cite{mancino2023decentralized} & $\mathcal{O}(\bar{M}\gamma^2/T)$ & $\mathcal{O}(\bar{M}\gamma^2/\epsilon)$ \\
        ADAPD-OG-MC \cite{mancino2023decentralized} & $\mathcal{O}(\bar{M}/T)$ & $\mathcal{O}(\bar{M}\sqrt{\gamma}/\epsilon)$ (\textbf{optimal}) \\
        xFILTER \cite{sun_distributed_2019} & $\mathcal{O}(\bar{M}/T)$ & $\mathcal{O}(\bar{M}\sqrt{\gamma}/\epsilon)$ (\textbf{optimal}) \\
        \textbf{UPP-SC-OPT} & $\mathcal{O}(\bar{M}/T)$ & $\mathcal{O}(\bar{M}\sqrt{\gamma}/\epsilon)$ (\textbf{optimal})\\
        \bottomrule
    \end{tabular}
    \caption{Convergence rates and communication complexity bounds of related works, where the communication complexity measures the communication rounds needed to reach $\epsilon$-stationarity. Notably, the graph Laplacian matrix here has a unit spectral radius.}
    \label{tab:convergence rate}
\end{table}


\emph{Notation:} For any differentiable function $f$, $\nabla f$ denotes its gradient. The null space and range space of a matrix argument are represented by $\operatorname{Null}(\cdot)$ and $\operatorname{span}(\cdot)$, respectively. We denote the $n$-dimensional column vector of all ones (zeros) by $\mathbf{1}_n$ ($\mathbf{0}_n$), and the $n$-dimensional identity and zero matrix by $\mathbf{I}_n$ and $\mathbf{O}_n$, respectively. The ceiling of $a\in\mathbb{R}$ is denoted by $\lceil a \rceil$. For $a_i\in \mathbb{R}$ ($i=1,\dots,n$), $\operatorname{diag}(a_1,\cdots,a_n)$ is the $n$-dimensional diagonal matrix with $a_i$ as its $i$-th diagonal entry. We use $\langle \cdot,\cdot \rangle$, $\otimes$ and $\|\cdot\|$ to denote the Euclidean inner product, Kronecker product and $\ell_2$ norm, respectively. For $\mathbf{A},\mathbf{B}\in \mathbb{R}^{d\times d}$, $\mathbf{A}\succ \mathbf{B}$ means $\mathbf{A}-\mathbf{B}$ is positive definite, and $\mathbf{A} \succeq \mathbf{B}$ means $\mathbf{A}-\mathbf{B}$ is positive semi-definite. The Moore-Penrose inverse of a matrix $\mathbf{A}$ is denoted by $\mathbf{A}^{\dagger}$, and $\lambda_i^\mathbf{A}$ is the $i$-th largest eigenvalue of $\mathbf{A}$. For symmetric $\mathbf{A} \succeq \mathbf{O}_d$ and $\mathbf{x}\in \mathbb{R}^d$, the weighted norm is defined as $\|\mathbf{x}\|^2_\mathbf{A}:=\mathbf{x}^{\mathsf{T}}\mathbf{A}\mathbf{x}$.

A preliminary conference version of this paper can be found in \cite{mappro2024Ou}, which involves a special case of UPP-MC.

\section{Problem Formulation}\label{section: problem formulation}
We solve problem~\eqref{p1} in a distributed way over a connected, undirected graph $\mathcal{G} = (\mathcal{V},\mathcal{E})$, where the vertex set $\mathcal{V} = \{1, \dots, N\}$ is the set of $N$ nodes and the edge set $\mathcal{E} \subseteq \{\{i,j\}|i,j \in \mathcal{V},i \neq j\}$ captures the underlying interactions among the nodes. Specifically, each node $i$ communicates only with the neighboring nodes in $\mathcal{N}_i=\{j\in\mathcal{V}:\{i,j\}\in\mathcal{E}\}$.
To solve problem~\eqref{p1} over the graph $\mathcal{G}$, we let each node $i \in \mathcal{V}$ maintain a local estimate $x_i \in \mathbb{R}^d$ of the global decision $x \in \mathbb{R}^d$ in problem~\eqref{p1}. We define 
\begin{equation}
\tilde{f}(\mathbf{x}) := \sum _{i \in \mathcal{V}} f_i(x_i),\quad \mathbf{x} = (x_1^{\mathsf{T}}, \dots, x_N^{\mathsf{T}})^{\mathsf{T}} \in \Re ^{Nd}.
\end{equation}
It has been shown in \cite{Mokhtari2016ADS} that problem~\eqref{p1} can be equivalently transformed  into 
\begin{equation}
\begin{aligned}
\underset{\mathbf{x} \in \mathbb{R}^{N d}}{\operatorname{minimize}}\quad & \tilde{f}(\mathbf{x}) \\
\operatorname{subject\;to}\quad & \bH^{\frac{1}{2}} \mathbf{x}=0,
\end{aligned} \label{p2}
\end{equation}
where the matrix ${\bH}$ is symmetric and positive semi-definite, i.e., $\bH = \bH^{\mathsf{T}}\succeq\mathbf{O}_{Nd}$. Besides, we require $\operatorname{Null}(\bH) = \mathcal{S} := \{\mathbf{x} \in \Re ^{Nd}|x_1 = \cdots = x_N\}$ and the corresponding optimal value coincides with that of problem~\eqref{p1}. Next, we introduce the following assumptions for the problem.
\begin{assumption}\label{assumption smooth}
The global objective $\tilde{f}:\Re^{Nd}\rightarrow\Re$ is $\bar{M}$-smooth for some $\bar{M}> 0$, i.e.,
\begin{equation}
\|\nabla \tilde{f}(\mathbf{x}) - \nabla \tilde{f}(\mathbf{y})\|\leq \bar{M}\|\mathbf{x}-\mathbf{y}\|,\quad\forall \mathbf{x},\mathbf{y}\in \Re^{Nd}.\label{smooth}
\end{equation}
\end{assumption}

\begin{assumption}\label{assumption: optimality}
The optimal set $\mathbb{X}^*$ of problem~\eqref{p1} is nonempty and the optimal value $f^* >-\infty$.
\end{assumption}

Assumptions~\ref{assumption smooth} and~\ref{assumption: optimality} are commonly adopted in distributed nonconvex optimization  (e.g., \cite{sun2018distributed, sun_distributed_2019, yi2021linear, yi2022sublinear}). Also, compared with the smoothness of each $f_i$ assumed in \cite{yi2022sublinear,yi2021linear}, Assumption~\ref{assumption smooth} only requires the smoothness of the global cost function ${f}$ and thus is slightly less restrictive.

\section{Algorithm Development}\label{section: algorithm development}
This section develops a unifying primal-dual proximal framework as well as its two distributed realizations.

\subsection{Unifying primal-dual proximal framework}\label{subsection:algorithm development UPP}
We first consider solving \eqref{p2} in the following primal-dual approach. Let $\mathbf{x}^k\in \mathbb{R}^{Nd}$ and $\mathbf{v}^{k}=((v_1^k)^{\mathsf{T}}, \dots, (v_N^k)^{\mathsf{T}})^{\mathsf{T}}\in \mathbb{R}^{Nd}$ denote the global primal and dual variables at iteration $k \geq 0$, respectively. Then, given $\rho, \theta>0$ and arbitrary initial points $\mathbf{x}^0,\mathbf{v}^0 \in \mathbb{R}^{Nd}$, 
\begin{align}
	\label{argmin approximation primal} \mathbf{x}^{k+1}=&\underset{\mathbf{x} \in \mathbb{R}^{N d}}{\arg \min}\;\tilde{f}(\mathbf{x}^k) + \langle \nabla \tilde{f}(\mathbf{x}^k), \mathbf{x}-\mathbf{x}^k \rangle +\theta(\mathbf{v}^{k})^{\mathsf{T}} \bH^{\frac{1}{2}} \mathbf{x}+\frac{\rho}{2}\|\mathbf{x}\|_{\mathbf{H}}^{2}+ \frac{1}{2} \|\mathbf{x}-\mathbf{x}^k\|^2_{\mathbf{B}^k}, \\ 
	\label{argmin approximation dual} \mathbf{v}^{k+1}=&\mathbf{v}^{k}+\rho \bH^{\frac{1}{2}} \mathbf{x}^{k+1}, \quad \forall k \geq 0.
\end{align}

The primal update \eqref{argmin approximation primal} minimizes an augmented-Lagrangian-like function, where we substitute the objective function by its first-order approximation at $\bx^k$ and insert an additional parameter $\theta>0$ into the Lagrangian multiplier term for enhancing generality. The AL-like function is also incorporated with a proximal term governed by a time-varying symmetric matrix $\mathbf{B}^k\in \mathbb{R}^{Nd \times Nd}$. We require that $\mathbf{B}^k + \rho \mathbf{H} \succ \mathbf{O}_{Nd}$, so that $\mathbf{x}^{k+1}$ in \eqref{argmin approximation primal} is well-defined and uniquely exists. For the dual update \eqref{argmin approximation dual}, we execute a dual ascent step on the AL-like function, with the dual gradient derived from the constraint residual evaluated at the newly computed primal variable $\bx^{k+1}$.

It can be shown that any primal-dual optimum pair $(\mathbf{x}^*,\mathbf{v}^*)$ of problem~\eqref{p2} constitutes a fixed point of \eqref{argmin approximation primal}--\eqref{argmin approximation dual}. To elucidate this, observe from \eqref{argmin approximation primal} that $\mathbf{x}^{k+1}$ uniquely exists and fulfills the first-order optimality condition
\begin{equation}
\label{first-order opt} \nabla \tilde{f}(\mathbf{x}^k) + \theta\bH^{\frac{1}{2}} \mathbf{v}^k + \rho \mathbf{H} \mathbf{x}^{k+1}+ \mathbf{B}^k(\mathbf{x}^{k+1} - \mathbf{x}^k)=0.
\end{equation}
Also note that $(\mathbf{x}^*,\mathbf{v}^*)$ satisfies
\begin{equation}
\nabla \tilde{f}(\mathbf{x}^*)+ \theta\bH^{\frac{1}{2}} \mathbf{v}^*=0. \label{fixed point}
\end{equation}
Once $(\mathbf{x}^k,\mathbf{v}^k) = (\mathbf{x}^*,\mathbf{v}^*)$, $(\mathbf{x}^{k+1},\mathbf{v}^{k+1})$ has to be $(\mathbf{x}^*,\mathbf{v}^*)$ due to \eqref{fixed point}, $\mathbf{H}\mathbf{x}^*=\bH^{\frac{1}{2}}\mathbf{x}^*=0$, and the uniqueness of $\mathbf{x}^{k+1}$.

Next, we discuss our algorithm design based on \eqref{argmin approximation primal}--\eqref{argmin approximation dual}. Using \eqref{first-order opt}, we rewrite \eqref{argmin approximation primal} as
\begin{equation}
	\mathbf{x}^{k+1} = \mathbf{x}^k - (\mathbf{B}^k + \rho \mathbf{H})^{-1}(\nabla \tilde{f}(\mathbf{x}^k) + \theta\bH^{\frac{1}{2}} \mathbf{v}^k + \rho \mathbf{H} \mathbf{x}^k). \label{x^k+1 vk}
\end{equation}
Then, we apply the following change of variable
\begin{equation}\label{qk=tildeH1/2vk}
\mathbf{q}^k = \bH^{\frac{1}{2}} \mathbf{v}^k. 
\end{equation}
This requires $\mathbf{q}^k \in \mathcal{S}^{\perp}\, \forall k \geq 0$, where $\mathcal{S}^{\perp} := \{\mathbf{x}\in \Re ^{Nd}| \, x_1 + \cdots + x_N = \mathbf{0}\}$ is the orthogonal complement of $\mathcal{S}$, which can be simply guaranteed by
\begin{equation}
\mathbf{q}^0 \in \mathcal{S}^{\perp}.\label{q0}
\end{equation}
Furthermore, for convenience, we define $\mathbf{G}^k:= (\mathbf{B}^k + \rho \mathbf{H})^{-1}$. With \eqref{qk=tildeH1/2vk}, we have the following equivalent form of \eqref{argmin approximation primal}--\eqref{argmin approximation dual}:
\begin{align} 
\label{xk+1 original} &\mathbf{x}^{k+1} = \mathbf{x}^k - \mathbf{G}^k (\nabla \tilde{f}(\mathbf{x}^k) + \theta\mathbf{q}^k + \rho \mathbf{H} \mathbf{x}^k),\\ 
\label{qk+1 original} &\mathbf{q}^{k+1}=\mathbf{q}^{k}+\rho \bH \mathbf{x}^{k+1}.
\end{align}

Notably, we will directly design the structure of $\bG^k$ later as there must exist some $\bB^k=(\bG^{k})^{-1}-\rho\bL$ for any feasible $\bG^k$. To further enhance the generality, we replace the weight matrix $\mathbf{H}$ in \eqref{xk+1 original} and \eqref{qk+1 original} with two possibly time-varying matrices $\mathbf{D}^k$ and $\tilde{\bD}^k$, respectively. Then, we reformulate \eqref{xk+1 original} and \eqref{qk+1 original} with an auxiliary variable $\mathbf{z}^{k}$ as follows: 
\begin{align}
    \label{zk final} \mathbf{z}^{k} &= \nabla \tilde{f}(\mathbf{x}^k) + \theta\mathbf{q}^k + \rho \mathbf{D}^k \mathbf{x}^k,\\
	\label{xk+1 final} \mathbf{x}^{k+1} &= \mathbf{x}^k - \mathbf{G}^k \mathbf{z}^k,\\ 
	\label{qk+1 final} \mathbf{q}^{k+1} &= \mathbf{q}^{k}+\rho \tilde{\mathbf{D}}^k\bx^{k+1}.
\end{align}

It is noteworthy that $\mathbf{D}^k$ and $\tilde{\bD}^k$ inherit all the conditions of $\bH$ below \eqref{p2} and the matrix $\bH$ will subsequently serve as a design foundation for constructing $\mathbf{D}^k$ and $\tilde{\bD}^k$. To formalize the design principles for $\mathbf{D}^k$, $\mathbf{\tilde{D}}^k$ and $\bG^k$, we introduce the following conditions.

\begin{assumption}\label{assumption: Dk}
For any $k\geq 0$, the matrices $\bH$, $\mathbf{D}^k$, $\tilde{\mathbf{D}}^k$ and $\bG^k$ satisfy the following conditions:
\begin{enumerate}
    \item[($\romannumeral1$)] $\bH$, $\mathbf{D}^k$ and $\tilde{\mathbf{D}}^k$ are symmetric and positive semi-definite.
    \item[($\romannumeral2$)] $\operatorname{Null}(\bH)=\operatorname{Null}(\bD^k)=\operatorname{Null}(\tilde{\bD}^k) = \mathcal{S}$.
    \item[($\romannumeral3$)] $\bH$, $\mathbf{D}^k$ and $\tilde{\mathbf{D}}^k$ are commutative in matrix multiplication.
    \item[($\romannumeral4$)] $\bG^k$ is symmetric and positive definite.
\end{enumerate}
\end{assumption}

Assumption~\ref{assumption: Dk}($\romannumeral1$)($\romannumeral2$) are widely adopted in distributed optimization \cite{shi_extra_2015,AugDGM2015,nedic_achieving_2017,xu_bregman_2018,mokhtari_dqm_2016,wu_second-order_2021,alghunaim_decentralized_2021,Wu2020AUA,hong2016convergence,hong_distributed_2018,mancino2023decentralized,yi2022sublinear,yi2021linear,alghunaim_unified_2022}, and Assumption~\ref{assumption: Dk}($\romannumeral3$) is also employed in unifying frameworks \cite{alghunaim_unified_2022,xu_distributed_2021} for generalizing a variety of existing works. Such $\bD^k$ and $\tilde{\bD}^k$ can be chosen as the polynomials of the weight matrix $\bH$, as we will elaborate in the next subsection. Assumption~\ref{assumption: Dk} ($\romannumeral4$) is induced by the condition $\bB^k+\rho\bH\succ\mathbf{O}_{Nd}$, which ensures the well-posedness of the minimization step in \eqref{argmin approximation primal}. The flexible parameterization of $\bD^k,\tilde{\bD}^k,\bG^k$ under Assumption~\ref{assumption: Dk} allows the proposed frameworks to generalize multiple existing algorithms, as will be discussed in Section~\ref{section: specializations}. Hence, we denote the framework described by \eqref{zk final}--\eqref{qk+1 final} as a \emph{\underline{U}nifying \underline{P}rimal-dual \underline{P}roximal} (UPP) framework.

\subsection{Distributed implementation and mixing acceleration}

This subsection presents the distributed implementation of UPP based on the primal-dual framework \eqref{zk final}--\eqref{qk+1 final} and introduces a mixing acceleration strategy over the network.

First, to enable the distributed deployment, we partition the variables in \eqref{zk final}--\eqref{qk+1 final} as $\mathbf{z}^k=\big((z_1^k)^{\mathsf{T}},\ldots,(z_N^k)^{\mathsf{T}}\big)^{\mathsf{T}}$, $\mathbf{x}^k=\big((x_1^k)^{\mathsf{T}},\ldots,(x_N^k)^{\mathsf{T}}\big)^{\mathsf{T}}$ and $\mathbf{q}^k=\big((q_1^k)^{\mathsf{T}},\ldots,(q_N^k)^{\mathsf{T}}\big)^{\mathsf{T}}$. Suppose each node $i\in \mathcal{V}$ maintains $z_i^k,x_i^k,q_i^k\in\mathbb{R}^d$. Thus, the initialization in \eqref{q0} can be simply satisfied by letting $q_i^0=\mathbf{0}_d\,\forall i\in\mathcal{V}$. To design $\bD^k$ and $\tilde{\bD}^k$ based on $\bH$, we choose $ \mathbf{H}= \mathbf{P} \otimes \mathbf{I}_d$ and impose the following Assumption on  $\mathbf{P}\in\mathbb{R}^{N\times N}$.
\begin{assumption}\label{assumption: distributed}
    The matrix $\mathbf{P}=(p_{ij})_{N\times N}$ satisfies the following conditions:
    \begin{itemize}
        \item[($\romannumeral1$)] $\mathbf{P}$ is symmetric and positive semi-definite.
        \item[($\romannumeral2$)] $\operatorname{Null}(\mathbf{P})=\operatorname{span}(\mathbf{1}_N)$.
        \item[($\romannumeral3$)] $p_{ij}=p_{ji}>0, \forall \{i,j\}\in \mathcal{E}$.
        \item[($\romannumeral4$)] $p_{ij}=p_{ji}=0, \forall i\in \mathcal{V},\,\forall j\notin \mathcal{N}_i\cup \{i\}$.
    \end{itemize}
\end{assumption}

Under Assumption~\ref{assumption: distributed}, the nodes can jointly determine such neighbor-sparse $\mathbf{P}$ without any centralized coordination \cite{Wu2020AUA}. 

Now we introduce $\bD^k$ and $\tilde{\bD}^k$ as the polynomials of $\bH$, which take the forms of
\vspace{-0.1cm}
\begin{equation*}\label{def Dk}
    \bD^k=P_{\tau_a^k}(\bH)=\sum_{t=1}^{\tau_a^k}a_t^k\mathbf{H}^t, \quad \tilde{\bD}^k=P_{\tau_{b}^k}(\bH)=\sum_{t=1}^{\tau_{b}^k}b_t^k\mathbf{H}^t.
\end{equation*}
where $\tau_a^k,\tau_{b}^k\geq 1$ are the polynomial degrees and $\mathbf{a}^k=(a_1^k,\dots,a_{\tau_a^k}^k)^{\mathsf{T}}\in\Re^{\tau_a^k}$, $\mathbf{b}^k=(b_1^k,\dots,b_{\tau_b^k}^k)^{\mathsf{T}}\in\Re^{\tau_b^k}$ are the polynomial coefficients. The definitions above meet Assumption~\ref{assumption: Dk}($\romannumeral1$)($\romannumeral2$)($\romannumeral3$).

Oracle~\ref{oracle acc} describes the distributed way of computing $P_{\tau_a^k}(\mathbf{H})\mathbf{x}^k$ in \eqref{zk final} via $\tau_a^k$ local communication rounds, which potentially accelerates the information mixing process. The term $\tilde{\bD}^k\bx^{k+1}$ in \eqref{qk+1 final} can be computed in the same way. We will specify the selections of $\tau_a^k,\tau_b^k$, $\mathbf{a}^k,\mathbf{b}^k$ in Section~\ref{subsection: Chebyshev} aiming at achieving more efficient communications. 

\floatname{algorithm}{Oracle}
\begin{algorithm}[h] 
        \renewcommand{\thealgorithm}{$\mathcal{A}$}
        \caption{: Mixing Acceleration} \label{oracle acc}
        \begin{algorithmic}[1]
                \STATE \textbf{Input:} $\mathbf{x}^k=((x_1^k)^{\mathsf{T}},\ldots,(x_N^k)^{\mathsf{T}})^{\mathsf{T}}\in \Re^{Nd}$, $\mathbf{P}\succeq\mathbf{O}_N$, $\tau_a^k\geq 1$, $\mathbf{a}^k=(a_1^k,\dots,a_{\tau_a^k}^k)^{\mathsf{T}}\in\Re^{\tau_a^k}$.\\ 
                \STATE \textbf{Procedure} $\operatorname{MACC}(\mathbf{x}^k,\mathbf{P},\tau_a^k,\mathbf{a}^k)$ \\
				\STATE Each node $i \in \mathcal{V}$ maintains a variable $x_i^{t}$ and sets $x_i^0=x_i^k$.
                \FOR{$t=0:\tau_a^k-1$} 
                \STATE Each node $i \in \mathcal{V}$ sends $x_i^{t}$ to every neighbor $j\in\mathcal{N}_i$ and then computes $x_i^{t+1}= \sum_{j\in {\mathcal{N}}_i \cup \{i\}} p_{ij} x_j^{t}$.\\
                \ENDFOR
                \STATE \textbf{Output:} Each node $i \in \mathcal{V}$ returns $\sum_{t=1}^{\tau_a^k}a_t^kx_i^t$, so that $P_{\tau_a^k}^\bD(\mathbf{H})\mathbf{x}^k=\big((\sum_{t=1}^{\tau_a^k}a_t^kx_1^t)^{\mathsf{T}},\ldots,(\sum_{t=1}^{\tau_a^k}a_t^kx_N^t)^{\mathsf{T}}\big)^{\mathsf{T}}$.\\
                \STATE \textbf{End procedure}\\
        \end{algorithmic}
\end{algorithm}

The matrix $\bG^k$ can also be constructed as a polynomial of $\bH$ (in Section~\ref{subsection: UPP-MC}) or in a block-diagonalizable structure (in Section~\ref{subsection: UPP-SC}), which results in two distinct realizations.

\subsection{UPP with multi-inner-loop communication} \label{subsection: UPP-MC}

To further accelerate the information fusion throughout the network, we construct a specialized realization of UPP, called UPP with \textbf{M}ulti-inner-loop \textbf{C}ommunication (UPP-MC), which results from specifying the matrix $\bG^k$ as 
\begin{equation}
\mathbf{G}^k = \zeta^k \mathbf{I}_{Nd} - \eta^k P_{\tau_d^k}(\mathbf{H}), \quad\forall k\ge0,\label{Gk}
\end{equation}
where $\zeta^k>0,\eta^k\geq 0$ and $P_{\tau_d^k}(\mathbf{H})=\sum_{t=1}^{\tau_d^k}d_t^k\mathbf{H}^t$ with $d_1^k,\dots,d_{\tau_d^k}^k\in\mathbb{R}$, which is a polynomial of $\bH$ with degree $\tau_d^k\geq 1$ and satisfies the following assumption.
\begin{assumption}\label{assumption polynomial}
For each $k\geq 0$, the parameters $\zeta^k,\eta^k$ and the polynomial $P_{\tau_d^k}(\mathbf{H})$ satisfy the following conditions:
\begin{itemize}
    \item[($\romannumeral1$)] $P_{\tau_d^k}(\mathbf{H})$ is positive semi-definite.
    \item[($\romannumeral2$)] $\zeta^k>0$ and $0\leq\eta^k<\zeta^k/\lambda_1^{P_{\tau_d^k}(\mathbf{H})}$.
\end{itemize}
\end{assumption}

Assumption~\ref{assumption polynomial} guarantees that $\mathbf{G}^k$ is positive definite ($\mathbf{G}^k\succ \mathbf{O}_{Nd}$), and thus satisfies Assumption~\ref{assumption: Dk} ($\romannumeral4$). In addition, the matrices $\mathbf{H}$, $\bD^k$, $\tilde{\bD}^k$ and $\bG^k$ are mutually commutative in matrix multiplication. With \eqref{Gk}, the term $P_{\tau_d^k}(\mathbf{H})\mathbf{z}^k$ in \eqref{xk+1 final} serves to propagate the gradient of the AL function, which can be evaluated in a fully distributed manner using Oracle~\ref{oracle acc}. We present the detailed distributed implementation of UPP-MC in Algorithm~\ref{algorithm: UPP-MC}.

\floatname{algorithm}{Algorithm}
\begin{algorithm}[h]
    \renewcommand{\thealgorithm}{1}
	\caption{UPP-MC}    
	\label{algorithm: UPP-MC} 
	\begin{algorithmic}[1]
	  \STATE \textbf{Parameters:} $\rho,\theta>0$, $\mathbf{P}\succeq\mathbf{O}_N$, $\zeta^k>0$, $\eta^k>0$, $\tau_a^k,\tau_b^k,\tau_d^k \geq 1$, $\mathbf{a}^k\in \mathbb{R}^{\tau_a^k}, \mathbf{b}^k\in \mathbb{R}^{\tau_b^k}, \mathbf{d}^k\in \mathbb{R}^{\tau_d^k}$.
		\STATE \textbf{Initialization:} Each node $i \in \mathcal{V}$ sets $q_i^0=0$ and arbitrary $x_i^0 \in \Re ^d$, and sends $x_i^0$ to every neighbor $j \in \mathcal{N}_i$.
		\FOR{$k \geq 0$} 
			\STATE The nodes jointly compute $\mathbf{z}^{k}=\nabla \tilde{f}(\mathbf{x}^{k}) + \theta \mathbf{q}^{k}+\rho \operatorname{MACC}(\mathbf{x}^{k},\mathbf{P},\tau_a^{k},\mathbf{a}^{k})$.
			\STATE Each node $i \in \mathcal{V}$ sends $z_i^{k}$ to every neighbor $j \in \mathcal{N}_i$.
			\STATE The nodes jointly compute $\mathbf{x}^{k+1} = \mathbf{x}^k - \zeta^k\mathbf{z}^k + \eta^k \operatorname{MACC}(\mathbf{z}^k,\mathbf{P},\tau_d^k,\mathbf{d}^k)$. 
			\STATE Each node $i \in \mathcal{V}$ sends $x_i^{k+1}$ to every neighbor $j \in \mathcal{N}_i$. \\
			\STATE The nodes jointly compute $\mathbf{q}^{k+1}=\mathbf{q}^k+\rho\operatorname{MACC}({\mathbf{x}}^{k+1},\mathbf{P},\tau_b^k,\mathbf{b}^k)$.
		\ENDFOR
	\end{algorithmic}
\end{algorithm}

The operator $\operatorname{MACC}$ in Algorithm~\ref{algorithm: UPP-MC} is achieved via Oracle~\ref{oracle acc} and is fully distributed. Observe that UPP-MC requires three communication inner loops at each iteration, which necessitates $(\tau_{a}^k+\tau_b^k+\tau_d^k)$ rounds of communication. 

\subsection{Communication acceleration scheme}\label{subsection: Chebyshev}

This subsection discusses the design of appropriate polynomials to accelerate information mixing.

As is shown in Table~\ref{tab:convergence rate}, the convergence rates of the mentioned algorithms are fundamentally governed by the condition number $\gamma:=\lambda_1^\bH/\lambda_{N-1}^\bH$ of the graph Laplacian matrix $\bH$. Since UPP adopts the polynomials $P_{\tau_l^k}(\bH),\,l=a,b,d$ instead of $\mathbf{H}$, it is reasonable to deduce that the convergence rate of UPP is influenced by the condition number of the polynomials, denoted as $\kappa_{P,l}^k:=\lambda_1^{P_{\tau_l^k}(\mathbf{H})}/\lambda_{N-1}^{P_{\tau_l^k}(\mathbf{H})}$. From a graph-theoretic perspective, $P_{\tau_l^k}(\mathbf{H})$ corresponds to a weighted interaction graph, where smaller $\kappa_{P,l}^k$ indicate denser graph topologies. This motivates us to incorporate Chebyshev acceleration \cite{auzinger2011iterative}, a well-established technique for spectral shaping, into our distributed framework. As a specialized instance of Oracle~\ref{oracle acc}, this method constructs a favorable polynomial $P_{\tau^k}(\mathbf{H})$ for fixed $\tau^k=\tau\geq 1$ that minimize $\kappa_{P,l}^k$. Subsequently, Oracle~\ref{oracle cheb} exhibits how Chebyshev acceleration can be applied in a distributed fashion over the graph $\mathcal{G}$.

\floatname{algorithm}{Oracle}
\begin{algorithm}[h] 
        \renewcommand{\thealgorithm}{$\mathcal{A}'$}
        \caption{: Chebyshev Acceleration} \label{oracle cheb}
        \begin{algorithmic}[1]
                \STATE \textbf{Input:} $\mathbf{y}=(y_1^{\mathsf{T}},\ldots,y_N^{\mathsf{T}})^{\mathsf{T}}\in \Re^{Nd}$, $\mathbf{H}=\mathbf{P}\otimes \mathbf{I}_d\succeq\mathbf{O}_{Nd}$, $\tau\geq 1$, $c_1 = \frac{1+\gamma}{1-\gamma}$, $c_2=\frac{1}{(1+\gamma)\lambda_1^\mathbf{P}}$.\\
                \STATE \textbf{Procedure} $\operatorname{CACC}(\mathbf{y},\mathbf{P},\tau)$ \\
                \STATE $b^0 = 1$, $b^1 = c_1$.\\ 
	        	\STATE $\mathbf{y}_0=\mathbf{y}$, $\mathbf{y}_1=c_1(\mathbf{I}-c_2\mathbf{H})\mathbf{y}_0$.\\
                \FOR{$t=1:\tau -1$} 
                \STATE $b^{t+1}= 2c_1 b^t-b^{t-1}$.\\
                \STATE $\mathbf{y}^{t+1}=2c_1(\mathbf{I}-c_2\mathbf{H})\mathbf{y}^k-\mathbf{y}^{k-1}.$
                \ENDFOR
                \STATE \textbf{Output:} $P_{\tau}(\mathbf{H})\mathbf{y}=\mathbf{y}_0-\mathbf{y}^\tau/b^\tau$.\\
                \STATE \textbf{End procedure}\\
        \end{algorithmic}
\end{algorithm}

Chebyshev acceleration theoretically optimizes the spectral properties of the polynomial $P_{\tau^k}(\mathbf{H})$ by minimizing its condition number, thereby enhancing convergence rates—a rigorous analysis of which will be presented in Section~\ref{section: convergence analysis}. Note that the number of inner-loop iterations is chosen to be proportional to $\sqrt{\gamma}$, a result rigorously proven in \cite{sun_distributed_2019} and further elaborated on in Section~\ref{section: convergence analysis}. Significantly, while the accelerated algorithm in \cite{sun_distributed_2019} also adopts Chebyshev acceleration at fixed polynomial constructions, our proposed mixing acceleration introduces a fundamentally more versatile approach that allows \textit{time-varying} polynomials to facilitate convergence.

Subsequently, we incorporate Oracle~\ref{oracle cheb} into UPP-MC in Algorithm~\ref{algorithm: UPP-MC}. Here, we set $\mathbf{D}^k=\mathbf{\tilde{D}}^k=\mathbf{H}$ (so that $a^k=b^k=1$), fix $\tau_d^k=\tau=\lceil\sqrt{\gamma}\rceil$ and compute $P_{\tau}(\mathbf{H})$ in \eqref{Gk} by Oracle~\ref{oracle cheb}, resulting in its Chebyshev-accelerated specialization of UPP-MC, referred to as UPP-MC-CA. Note that for UPP-MC-CA, Lines 4 and 8 of Algorithm~\ref{algorithm: UPP-MC} perform identical decision transmissions, where the MACC operator is replaced with CACC in Oracle~\ref{oracle cheb}, so that these two steps only require a single communication operation. Consequently, each iteration of UPP-MC-CA involves $\tau+1$ communication rounds. 

\subsection{UPP with single-inner-loop communication} \label{subsection: UPP-SC}
Driven by the goal of exploring generality of UPP and reducing communication overhead, we propose another realization of UPP, called UPP with \textbf{S}ingle-inner-loop \textbf{C}ommunication (UPP-SC). 

Based on \eqref{zk final}--\eqref{qk+1 final}, we introduce the following variable selections: Let $\theta = 1$, $\mathbf{D}^k=\mathbf{\tilde{D}}^k = \mathbf{L}=P_{\tau_e}(\bH)=\sum_{t=1}^{\tau_e} e_t\bH^t$, $\tau_e\geq 1$, $e_1,\dots,e_{\tau_e}\in \mathbb{R}$, and $\mathbf{G}^k\succ \mathbf{O}_{Nd}$. Then, UPP reduces to
\begin{align}
	\label{xk+1 opt} \mathbf{x}^{k+1} & = \mathbf{x}^k - \mathbf{G}^k (\nabla \tilde{f}(\mathbf{x}^k) + \rho \mathbf{L} \mathbf{x}^k + \mathbf{q}^k),\\
	\label{qk+1 opt} \mathbf{q}^{k+1} &= \mathbf{q}^{k}+\rho \mathbf{L} \mathbf{x}^{k+1}.
\end{align}

Since $\mathbf{L}$ is a polynomial of the weight matrix $\mathbf{H}$, $\mathbf{L}\bx^k$ can be computed through $\tau_e$ inner loops for communication per iteration, as is shown in Oracle~\ref{oracle acc}. Note that, by storing $\mathbf{L}\mathbf{x}^{k+1}$ during the dual update \eqref{qk+1 opt} at $k$-th iteration and reusing it  at $(k+1)$-th iteration, the term $\mathbf{L} \mathbf{x}^{k+1}$ in \eqref{xk+1 opt} consumes no additional communication. Then, iterations \eqref{xk+1 opt}--\eqref{qk+1 opt} with the above operations lead to UPP-SC.

In UPP-SC, the matrix $\mathbf{G}^k$ is not necessarily constructed as a polynomial of the weight matrix $\bH$. Instead of \eqref{Gk} under Assumption~\ref{assumption polynomial}, we stipulate that $\bG^k$ can be block-diagonalized in the following form:
\begin{equation}\label{Gk block diagonalization}
    \bG^k=\operatorname{diag}(\bG_1^k,\dots,\bG_N^k), \quad \forall k\geq 0,
\end{equation}
where each $\bG_i^k\in\mathbb{S}^{N}$ is positive definite. We elaborate the distributed implementation of UPP-SC in Algorithm~\ref{algorithm: UPP-SC}.

\floatname{algorithm}{Algorithm}
\begin{algorithm}[h]
	\renewcommand{\thealgorithm}{3}
	\caption{UPP-SC}
	\label{algorithm: UPP-SC}               
	\begin{algorithmic}[1]
		\STATE \textbf{Parameters:} $\rho>0$, $\mathbf{P}\succeq\mathbf{O}_N$, $\tau_e^k\geq 1$, $\mathbf{e}^k=(e_1^k,\dots,e_{\tau_a^k}^k)^{\mathsf{T}}\in\Re^{\tau_e^k}$ and $\bG^k\in\mathbb{S}^{Nd}$.
		\STATE \textbf{Initialization:} Each node $i \in\mathcal{V}$ sets $x_i^{-1}=0$, $q_i^{-1}=0$, $x_i^{0}=-\bG_i^0 \nabla \tilde{f}(0)$. The nodes jointly compute ${\mathbf{y}}^0=\operatorname{MACC}(\mathbf{x}^0,\mathbf{P},\tau_e^k)$ and $\mathbf{q}^{0}=\rho {\mathbf{y}}^0$.
		\FOR{$k \geq 0$} 
			\STATE Each node $i \in \mathcal{V}$ computes $x_i^{k+1} = x_i^k -\bG_i^k (\nabla f_i(x_i^{k}) + q_i^{k} + \rho y_i^{k})$ and sends it to every neighbor $j \in \mathcal{N}_i$.
			\STATE The nodes jointly compute ${\mathbf{y}}^{k+1}=\operatorname{MACC}(\mathbf{x}^{k+1},\mathbf{P},\tau^k)$.
			\STATE Each node $i \in \mathcal{V}$ computes $q_i^{k+1} = q_i^{k} + \rho y_i^{k+1}.$\\
		\ENDFOR
	\end{algorithmic}
\end{algorithm}

Below, we provide the design examples of $\mathbf{G}^k$.

\subsubsection*{Example 1} In UPP-SC, we substitute $\mathbf{G}^k = \mu\mathbf{I}_{Nd}$, $\tau_e=\tau=\lceil \sqrt{\gamma}\rceil$, $\mathbf{L}=P_{\tau}(\mathbf{H})/\lambda_1^{P_{\tau}(\mathbf{H})}$, where $P_{\tau}(\mathbf{H})$ is generated by Chebyshev acceleration. Practically, we replace $\bG_i^0$ (Line 2 in Alg. \ref{algorithm: UPP-SC}) and $\bG_i^k$ (Line 4 in Alg. \ref{algorithm: UPP-SC}) with the $\mu\bI_{Nd}$, and substitute the operator $\operatorname{MACC}$ in Lines 2 and 5 of Alg. \ref{algorithm: UPP-SC} with $\operatorname{CACC}$, given in Oracle~\ref{oracle cheb}. In addition, With proper parameter selections, this specialization of UPP-SC, referred to as UPP-SC-OPT, is able to attain theoretically optimal communication complexity bounds (will be analyzed in Section~\ref{convergence of UPP-SC-OPT}).

\subsubsection*{Example 2}

When $\tilde{f}$ is twice-differentiable, we may let $\mathbf{G}^k$ involve its second-order information so that UPP-SC can be specialized to \textit{second-order} methods. Let the matrix $\mathbf{G}^k = (2\rho \operatorname{diag}(|\mathcal{N}_1|,\dots,|\mathcal{N}_N|)\otimes \bI_d+\nabla^2 \tilde{f}(\mathbf{x}^k))^{-1}$, where $|\mathcal{N}_i|$ denotes the number of node $i$'s neighbors. Observe that $\mathbf{G}^k$ is block-diagonalizable with each diagonal block $\bG_i^k=(2\rho|\mathcal{N}_i|\otimes \bI_d+\nabla^2 f_i(\mathbf{x}^k))^{-1}$. We also fix $\tau=\lceil \sqrt{\gamma}\rceil$ and let $\mathbf{L}=P_{\tau}(\mathbf{H})/\lambda_1^{P_{\tau}(\mathbf{H})}$, where $P_{\tau}(\mathbf{H})$ is generated by Oracle~\ref{oracle cheb}. Correspondingly, we substitute the operator $\operatorname{MACC}$ in Lines 2 and 5 of Alg. \ref{algorithm: UPP-SC} with $\operatorname{CACC}$. At each iteration, nodes update their local variables using second-order information, motivating the name UPP-SC with \textbf{S}econd-\textbf{O}rder information (UPP-SC-SO). In UPP-SC-SO, instead of feeding $\bG^k$ as an input, each node $i$ evaluates its local $\bG_i^k$ by exploiting the curvature of its $f_i$.

\begin{remark}
UPP-MC and UPP-SC are distinct in parameter configurations, as shown in Table~\ref{tab:realizations}.
\begin{table}[h]
    \centering
    \begin{tabular}{c|c|c|c}
        \toprule
         Framework & $\bG^k$ & $\bD^k$ & $\tilde{\bD}^k$ \\
         \midrule
         UPP-MC & $\zeta^k \mathbf{I}_{Nd} - \eta^k P_{\tau_d^k}(\mathbf{H})$ & $P_{\tau_a^k}(\mathbf{H})$ & $P_{\tau_b^k}(\mathbf{H})$\\
         \midrule
         UPP-SC & block-diagonalizable & $P_{\tau_e}(\bH)$ & $P_{\tau_e}(\bH)$\\
         \bottomrule
    \end{tabular}
    \caption{Difference between UPP-MC and UPP-SC}
    \label{tab:realizations}
\end{table}

Both UPP-MC and UPP-SC require $\bG^k\succ \mathbf{O}_{Nd}$ and their fundamental distinction is that UPP-MC employs $\bG^k$ as a polynomial of the graph Laplacian matrix and UPP-SC utilizes a block-diagonalizable structure that may enable explicit incorporation of second-order information.
\end{remark}

\section{Existing ALgorithms as Specializations}\label{section: specializations}
This section presents how UPP-MC and UPP-SC generalize a wide range of existing distributed optimization algorithms.
\subsection{Specializations of UPP-MC}
UPP-MC reduces to several distributed first-order algorithms for convex and nonconvex optimization.


\begin{itemize}
	\item[1)] \label{analysis extra} \textit{EXTRA}: EXTRA \cite{shi_extra_2015} is a renowned first-order algorithm developed from a decentralized gradient descent method. As is established in \cite{Wu2020AUA}, EXTRA can be equivalently written as
	\begin{align}
		  \label{EXTRA primal 1}\mathbf{x}^{k+1} &= (\tilde{\bW} \otimes \bI_d) \mathbf{x}^k - \alpha \nabla \tilde{f}(\mathbf{x}^k) - \alpha \mathbf{q}^k, \\
		  \label{EXTRA primal 2}\mathbf{q}^{k+1} &= \mathbf{q}^k + \frac{1}{\alpha} ((\bW - \tilde{\bW})\otimes \bI_d) \mathbf{x}^{k+1}.
	\end{align}

	This can be cast into UPP-MC with $\rho = \frac{1}{\alpha},\theta=1$, $\bD^k=(\bI_N-\tilde{\bW})\otimes \bI_d$, $\tilde{\bD}^k=(\tilde{\bW}-\bW)\otimes \bI_d$, and $\bG^k=\alpha\bI_{Nd}$. Given the assumptions in \cite{shi_extra_2015} that $\tilde{\bW} \succeq \bW$ and $\operatorname{Null}(\tilde{\bW} - \bW) = \operatorname{span}(\mathbf{1}_d)$, it is ensured that $\bD^k$, $\tilde{\bD}^k$ and $\bG^k$ share the same eigenvectors. Consequently, it is straightforward to verify that all the conditions in Assumptions~\ref{assumption: Dk}--\ref{assumption polynomial} are met.

	\item[2)]\textit{DIGing:} DIGing \cite{nedic_achieving_2017} is a gradient-tracking method for distributed convex optimization in the context of time-varying networks. Here, we focus specifically on DIGing whose weight matrices associated with the time-varying networks commute with each other. Let $\alpha>0$ and $\bW^k \in \mathbb{R}^{N\times N}\forall k\geq 0$ satisfy $\bW^k\mathbf{1}=(\bW^k)^{\mathsf{T}}\mathbf{1} = \mathbf{1}, [\bW]_{ij}=0 \,\forall i\in \mathcal{V}\,\forall j \notin \mathcal{N}_i \cup \{i\}$, and $\|\bW^k-\frac{1}{N}\mathbf{1}\mathbf{1}^{\mathsf{T}}\|<1$. Under these conditions, DIGing can be written as:
	\begin{align*}
		\mathbf{x}^{k+2}=&((\bW^{k+1}+\bW^k)\otimes I_d)\mathbf{x}^{k+1}-((\bW^k)^2\otimes I_d)\mathbf{x}^k-\alpha (\nabla \tilde{f} (\mathbf{x}^{k+1})-\nabla \tilde{f}(\mathbf{x}^k)), \quad \forall k\geq 0,
	\end{align*} 
	where $\mathbf{x}^0$ can be arbitrary and $\mathbf{x}^1 = (\bW^0\otimes I_d)\mathbf{x}^0 - \alpha \nabla \tilde{f}(\mathbf{x}^0)$. Adding the above equation from $k=0$ to $k=K-1$ yields
	\begin{align*}
		\mathbf{x}^{K+1}=&\big((\bW^K)^2\otimes \bI_d\big)\mathbf{x}^K+\bx^1 - (
        \bW^0\otimes \bI_d)\mathbf{x}^0 \\&-\sum_{k=0}^{K-1}\big((\bI_N-(\bW^{k+1}+\bW^{k})+(\bW^{k+1})^2)\otimes I_d\big)\mathbf{x}^{k+1}\notag\\
        &- \alpha \big(\nabla \tilde{f}(\mathbf{x}^K)-\nabla \tilde{f}(\mathbf{x}^0)\big), \quad \forall K>0.
	\end{align*}
	By letting $\mathbf{q}^0=\frac{1}{\alpha}\big((\bW^0)^2-\bW^0\big)\otimes \bI_d)\mathbf{x}^0$, the above update is the same as
	\begin{align*}
		\mathbf{x}^{K+1}  =& ((\bW^K)^2\otimes \bI_d)\mathbf{x}^K - \alpha \nabla \tilde{f}(\mathbf{x}^K) - \alpha \mathbf{q}^K,\\
		\mathbf{q}^{K+1} =&  \mathbf{q}^{K} + \frac{1}{\alpha}\Big(\big(\bI_N-(\bW^K+\bW^{K-1})+(\bW^K)^2\big)\otimes \bI_d\Big)\mathbf{x}^{K+1}.
	\end{align*}

	The algorithmic form of DIGing is equivalent to UPP-MC in \eqref{zk final}--\eqref{qk+1 final} with $\mathbf{q}^0 \in \mathcal{S}^{\bot}$, $\rho =1/\alpha,\theta=1$, $\bD^K=(\bI-(\bW^K)^2)\otimes \bI_d,$ $\tilde{\bD}^K=\big(\bI_N-(\bW^K+\bW^{K-1})+(\bW^K)^2\big)\otimes \bI_d$, and $\bG^K= \alpha \bI_{Nd}$. It can be verified that all the conditions in Assumptions~\ref{assumption: Dk}--\ref{assumption polynomial} are fulfilled.

	
	\item[3)] \textit{L-ADMM}: L-ADMM \cite{yi2022sublinear} is a distributed first-order primal-dual algorithm for solving problem \eqref{p2} with smooth, nonconvex $f_i$'s along with the P-{\L} condition. Our proposed UPP-MC \eqref{zk final}--\eqref{qk+1 final} reduces to L-ADMM when $\rho = \alpha,\theta=\beta,\mathbf{D}^k=\mathbf{L},\tilde{\mathbf{D}}^k=\frac{\beta}{\alpha\gamma}\mathbf{L}$, and $\mathbf{G}^k=\frac{1}{\gamma}\mathbf{I}_{Nd}$, so that \eqref{zk final}--\eqref{qk+1 final} become
	\begin{align*}
		 \mathbf{x}^{k+1} &= \mathbf{x}^k - \frac{1}{\gamma}(\nabla \tilde{f}(\mathbf{x}^k) + \alpha \mathbf{L} \mathbf{x}^k + \beta\mathbf{q}^k), \\
		 \mathbf{q}^{k+1} &= \mathbf{q}^{k} + \frac{\beta}{\gamma}\mathbf{L}\mathbf{x}^{k+1},
	\end{align*}
    with arbitrary $\mathbf{x}^0 \in \Re^{Nd}$ and $\mathbf{q}^0=\mathbf{0}$. As \cite{yi2022sublinear} assumes $\bL \succeq \mathbf{0}$ and $\operatorname{Null}(\mathbf{L}) = \{\mathbf{1}_{Nd}\}$, Assumptions~\ref{assumption: Dk}--\ref{assumption polynomial} hold.

	\item[4)] \textit{Prox-GPDA}: Prox-PDA \cite{hong_prox-pda_2017} is a proximal primal-dual algorithm specifically tailored for distributed nonconvex optimization. Here, we further consider its closed-form extension—Prox-GPDA—with the following form:
    \begin{align*}
        \bx^{k+1}=&\bx^k-\beta(\mathbf{A}^{\mathsf{T}}\mathbf{A}+\mathbf{B}^{\mathsf{T}}\mathbf{B})(\nabla \tilde{f}(\bx^k)+\bq^k+\beta\mathbf{A}^{\mathsf{T}}\mathbf{A}\bx^k),\\
        \bq^{k+1}=&\bq^k+\beta \mathbf{A}^{\mathsf{T}}\mathbf{A}\bx^{k+1},
    \end{align*}
    where $\mathbf{x}^0 \in \Re^{Nd}$ is arbitrary and $\mathbf{q}^0=\mathbf{0}$. This is exactly UPP with $\rho=\beta, \theta=1$, $\bD^k=\tilde{\bD}^k=\mathbf{A}^{\mathsf{T}}\mathbf{A}$ and $\bG^k=\beta(\mathbf{A}^{\mathsf{T}}\mathbf{A}+\mathbf{B}^{\mathsf{T}}\mathbf{B})$. Besides, since $\operatorname{Null}(\mathbf{A})=\{\mathbf{1}_{Nd}\}$ and $\mathbf{A}^{\mathsf{T}}\mathbf{A}+\mathbf{B}^{\mathsf{T}}\mathbf{B}\succ \mathbf{O}_{Nd}$, Assumptions~\ref{assumption: Dk}--\ref{assumption polynomial} hold.

    \item[5)] \textit{SUDA}: SUDA \cite{alghunaim_unified_2022} is a general stochastic unified decentralized algorithm for addressing distributed nonconvex optimization. Here, we consider its deterministic version, which can be regarded as a special case of our proposed UPP-MC by letting $\rho=\frac{1}{\alpha},\theta=1$, $\bD^k=\mathbf{A}^\dagger-\mathbf{C}$, $\tilde{\bD}^k=\mathbf{A}^\dagger\mathbf{B}^2$ and $\bG^k=\alpha \mathbf{A}$, where $\alpha>0$ and $\mathbf{A},\mathbf{B}^2,\mathbf{C}$ are chosen as polynomial functions of the weight matrix $\bW$ such that they are doubly stochastic and commute with each other. Again, Assumptions~\ref{assumption: Dk}--\ref{assumption polynomial} are met.

\end{itemize}

    \subsection{Specializations of UPP-SC}
    Recall that UPP-SC does not require the matrix $\bG^k$ to be a polynomial of the weight matrix (i.e., Assumption~\ref{assumption polynomial} does not necessarily hold), thereby enabling generalization of both first-order and second-order algorithms. 
    \begin{enumerate}
        \item[1)] \textit{ID-FBBS:} ID-FBBS \cite{xu_bregman_2018} adopts the form of \eqref{EXTRA primal 1}--\eqref{EXTRA primal 2}, with the modification that $\bW = 2 \tilde{\bW} - \bI_N$ and $\mathbf{q}^0$ can be any vector belonging to $\mathcal{S}^{\perp}$. Given the assumption in \cite{xu_bregman_2018} that $\tilde{\bW} \succ \mathbf{O}$, it can be shown that ID-FBBS is a specific instance of UPP-SC where $\rho = \frac{1}{\alpha}$, and $\bL=(\bI_N-\tilde{\bW})\otimes \bI_d$, $\bG^k=\alpha\bI_{Nd}$. Clearly, ID-FBBS satisfies Assumptions \ref{assumption: Dk}--\ref{assumption: distributed} and the block-diagonalization of $\bG^k$.
        \item[2)] \textit{DQM}: DQM \cite{mokhtari_dqm_2016} is a distributed second-order method for solving problem~\eqref{p1} where each local objective $f_i$ is strongly convex, smooth and twice continuously differentiable. The updates of DQM are given by: $x_i^{k+1}=x_i^k-\left(2c |\mathcal{N}_i| I_d + \nabla^2 f_i(x_i^k)\right)^{-1} ( c \sum_{j \in \mathcal{N}_i} (x_i^k - x_j^k)+ \nabla f_i(x_i^k) + q_i^k )$ and $q_i^{k+1}=q_i^k+c \sum_{j \in \mathcal{N}_i} (x_i^{k+1} - x_j^{k+1})$ with $c>0$, where $x_i^0\in\mathbb{R}^d\,\forall i\in\mathcal{V}$ are arbitrarily chosen, and $q_i^0\,\,\forall i \in \mathcal{V}$ are such that $\sum_{i\in\mathcal{V}}q_i^0=\mathbf{0}_d$. UPP-SC reduces to DQM by setting $\rho=c$, $\bG^k=\big((2c\operatorname{diag(|\mathcal{N}_1|,\dots,|\mathcal{N}_N)})\otimes \bI_d+\nabla \tilde{f}(\bx^k)\big)^{-1}$ and $\bL=\mathbf{H}=\mathbf{P}\otimes \bI_d$, where $p_{ij}=-1\,\,\forall \{i,j\}\in\mathcal{E}$ and $p_{ij}=|\mathcal{N}_i|\,\,\forall i\in\mathcal{V}$. It is clear that Assumptions~\ref{assumption: Dk}--\ref{assumption: distributed} are satisfied.
        \item[3)] \textit{SoPro}: SoPro \cite{wu_second-order_2021} is a distributed second-order proximal algorithm for addressing problem~\eqref{p1} where $\sum_{i\in\mathcal{V}}f_i$ is locally restricted strongly convex and twice continuously differentiable. The updates of SoPro is given by: $x_i^{k+1}=x_i^k-(\nabla f_i(x_i^k)+D_i)^{-1}(\nabla f_i(x_i^k)+\rho y_i^k+q_i^k)$ and $q_i^{k+1}=q_i^k+\rho \sum_{j\in\mathcal{N}_i}p_{ij}(x_i^{k+1}-x_j^{k+1})$. UPP-SC takes the same form of SoPro when $\bL=\bW$ and $\bG^k=(\nabla^2 \tilde{f}(\bx^k)+\bD)^{-1}$, where $\bD=\operatorname{diag}(\bD_1,\dots,\bD_N)$ is a symmetric block diagonal matrix. SoPro starts from arbitrary $\bx^0\in\mathbb{R}^{Nd}$ and $\bq^0\in\mathcal{S}^{\bot}$, which meets \eqref{q0}. Furthermore, $\mathbf{H}=\mathbf{P}\otimes \bI_d$ with $\mathbf{P}$ be a weighted Laplacian matrix satisfies Assumption~\ref{assumption: distributed}, which, together with $\nabla^2 \tilde{f}(\bx^k)+\bD\succ \mathbf{O}_{Nd}$, guarantees Assumption~\ref{assumption: Dk}.
    \end{enumerate}

\section{Convergence Analysis}\label{section: convergence analysis}
In this section, we provide the convergence analysis of the two realizations UPP-MC and UPP-SC, and the communication complexity bound for UPP-SC-OPT.
\subsection{Convergence analysis of UPP-MC}\label{general convergence}
In this section, we analyze the convergence rate of UPP-MC (described by \eqref{zk final}--\eqref{Gk}) under various nonconvex conditions.

First, we generate a sequence, consisting of the consensus error and the objective function error. For some $\bar{\epsilon}>0$, 
\begin{equation}
	\label{def V} V^k = \frac{1}{2} \|\mathbf{x}^k\|^2_\mathbf{K} + \frac{1}{2\bar{\varepsilon}}\|\mathbf{s}^k\|^2_{\bK}+ \langle \mathbf{x}^k, \mathbf{K}\mathbf{s}^k \rangle + f(\bar{x}^k) - f^*,
\end{equation}
where $\|\mathbf{x}^k\|^2_\mathbf{K}$ is the consensus error with $\mathbf{K} =(\mathbf{I}_N - \frac{1}{N} \mathbf{1}_N \mathbf{1}_N ^{\mathsf{T}}) \otimes \mathbf{I}_d$; $\|\mathbf{s}^k\|^2_{\bK}$ reveals the stationary error with $\mathbf{s}^k=\mathbf{q}^k+\frac{1}{\theta}\nabla \tilde{f}(\bar{\mathbf{x}}^k)$, and $\bar{\mathbf{x}}^k=\frac{1}{N}(\mathbf{1}_N\mathbf{1}^\mathsf{T}_N\otimes \mathbf{I}_d)\mathbf{x}^k$; and $f(\bar{x}^k) - f^*$ represents the objective function error with $\bar{x}^k = \frac{1}{N}(\mathbf{1}^\mathsf{T}_N\otimes \mathbf{I}_d)\mathbf{x}^k$.
The following proposition characterizes the dynamics of this sequence, verifying it is nonincreasing.

\begin{proposition}\label{proposition parameter range}
	Suppose Assumptions~\ref{assumption smooth}--\ref{assumption polynomial} hold. Let $\{\mathbf{x}^k\}$ be the sequence generated by UPP-MC with proper parameters\footnote{For better readability, the explict expressions of the parameters and the constants throughout Section~\ref{general convergence} are given in the corresponding proofs.}. Then, the sequence of $V^k$ is nonincreasing.
\end{proposition}
\begin{proof}
	See Appendix~\ref{proof proposition}.
\end{proof}

Below, we evaluate the convergence performance of UPP-MC via an optimality gap composed of the consensus error and the stationarity violation, which takes the form as
\begin{equation}\label{hatWk}
	\hat{W}^k = \|\mathbf{x}^{k} - \bar{\mathbf{x}}^k\|^2 + \|\mathbf{J}\nabla \tilde{f}(\mathbf{x}^k)\|^2,
\end{equation} 
where $\mathbf{J}=\frac{1}{N}\mathbf{1}_N \mathbf{1}_N ^{\mathsf{T}} \otimes \mathbf{I}_d$. 

\begin{theorem}\label{theorem nonconvex}
	Suppose Assumptions~\ref{assumption smooth}--\ref{assumption polynomial} hold. Let $\{\mathbf{x}^k\}$ be the sequence generated by UPP-MC with the same parameter selections in proposition~\ref{proposition parameter range}. Then, for some $C_1>0$,
	\begin{align}
		\label{theorem 1 sublinear}& \frac{1}{T}\sum _{k=0}^{T-1}\hat{W}^k \leq \frac{C_1}{T}.
	\end{align}

\end{theorem}
\begin{proof}
	See Appendix~\ref{proof theorem nonconvex}.
\end{proof}

From \eqref{theorem 1 sublinear}, we have the running average of the optimality gap dissipates at a rate of $\mathcal{O}(1/T)$. This indicates that UPP-MC converges to a stationary solution at a sublinear rate, which is of the same order as \cite{hong_prox-pda_2017,sun2018distributed,sun_distributed_2019,yi2022sublinear,yi2021linear} for smooth nonconvex problems. This extends the convergence results for the distributed convex optimization algorithms in \cite{shi_extra_2015,nedic_achieving_2017,xu_bregman_2018} to nonconvex settings.

\begin{remark}
	For the parameter selections described in Appendix~\ref{proof proposition}, we require the knowledge of the global information $\bar{M}$, $\lambda_1^{\mathbf{H}}$ and $\lambda_{N-1}^{\mathbf{H}}$, which can be collectively found via local communication between neighboring nodes \cite{Wu2020ASP}. Subsequently, we illustrate how to determine the matrices $\mathbf{G}^k,\mathbf{D}^k,\mathbf{\tilde{D}}^k$ in the iterations \eqref{zk final}--\eqref{qk+1 final} under the parameter selections in \eqref{range varepsilon theta}--\eqref{range underline varepsilon}. For example, for $\tau^k>1$, we can determine ${\mathbf{D}^k}=\sum_{t=1}^{\tau^k}b_t^k\mathbf{H}^t \succeq \mathbf{O}_{Nd}$, which can be simply satisfied by choosing $b_1^k,\dots,b_{\tau^k}^k>0$, and thus $\lambda_{1}^{\mathbf{D}^k}=\sum_{t=1}^{\tau^k}b_t^k({\lambda}_1^\mathbf{H})^t$. Subsequently, we can successively determine $\bar{\varepsilon}, \theta, \rho, \lambda_{2}^{\mathbf{G}^k}, \zeta^k$ and $\underline{\varepsilon}$ by \eqref{range varepsilon theta}--\eqref{range underline varepsilon} in Appendix~\ref{proof proposition}. To design $\mathbf{G}^k$, one way is to select appropriate $a_1^k,\dots, a_{\tau^k}^k>0$ for $k\geq 0$ such that $P_{\tau^k}(\mathbf{H})\succeq \mathbf{O}$, and thus each node is able to locally attain ${\lambda}_1^{P_{\tau^k}(\mathbf{H})} = \sum_{t=1}^{\tau^k}a_t^k({\lambda}_1^\mathbf{H})^t$ and ${\lambda}_{N-1}^{P_{\tau^k}(\mathbf{H})} = \sum_{t=1}^{\tau^k}a_t^k({\lambda}_{N-1}^\mathbf{H})^t$. Moreover, we select $\eta^k=(\zeta^k-\lambda_2^{\mathbf{G}^k})/\lambda_{N-1}^{P_{\tau^k}(\mathbf{H})}$. Since $\zeta^k < \kappa_p^k\lambda_{2}^{\mathbf{G}^k}/(\kappa_p^k-1)$, we have $\lambda_N^{\mathbf{G}^k} = \zeta^k - \eta^k\lambda_1^{P_{\tau^k}(\mathbf{H})}>0$, which satisfies Assumption~\ref{assumption polynomial}. Furthermore, we set $\tilde{\mathbf{D}}^k$ by $\underline{\varepsilon}\mathbf{D}^k\mathbf{G}^k\preceq\tilde{\mathbf{D}}^k \preceq \bar{\varepsilon}\mathbf{D}^k\mathbf{G}^k$ in \eqref{tildeD range}, and thus the matrices $\mathbf{D}^k$ and $\tilde{\mathbf{D}}^k$ satisfy Assumption~\ref{assumption: Dk}.
\end{remark}

Next, we analyze the convergence of UPP-MC under the P-{\L} condition.
\begin{assumption}\label{PL condition}
	The global cost function $f(x)$ satisfies the P-{\L} condition with constant $\nu>0$, i.e.,
	\begin{equation}
		\|\nabla f(x)\|^2 \geq 2\nu (f(x) - f^*), \quad \forall x \in \Re ^d. \label{pl}
	\end{equation}
\end{assumption}

Note that the P-{\L} condition is weaker than strong convexity, and can guarantee the global optimum without convexity. 

\begin{theorem}\label{theorem pl}
	Suppose Assumptions~\ref{assumption smooth}--\ref{PL condition} hold. Let $\{\mathbf{x}^k\}$ be the sequence generated by UPP-MC with proper parameters. Then, for some $C_2>0$ and $\delta\in(0,1)$,
	\begin{equation}
		\|\mathbf{x}^k\|^2_{\mathbf{K}} + f(\bar{x}^k) - f^* \leq C_2(1-\delta)^{k}, \quad \forall k \geq 1.
		\label{theorem 2 linear}
	\end{equation}
\end{theorem}
\begin{proof}
	See Appendix~\ref{proof theorem pl}.
\end{proof}

Theorem~\ref{theorem pl} shows that UPP enjoys a linear convergence to the global optimum under the P-{\L} condition. Similar linear rates have been achieved in \cite{yi2022sublinear,yi2021linear,alghunaim_unified_2022} under the P-{\L} condition. This result relaxes the strong convexity in \cite{shi_extra_2015,nedic_achieving_2017,xu_bregman_2018} and explores the global optimality for distributed nonconvex algorithms (e.g., \cite{alghunaim_unified_2022,hong_prox-pda_2017,sun2018distributed}).


\subsection{Convergence analysis of UPP-SC}\label{convergence of UPP-SC}

This subsection states the convergence results for UPP-SC (described by \eqref{xk+1 opt}--\eqref{qk+1 opt}).

	For better presentation of our analysis, we reinterpret UPP-SC \eqref{xk+1 opt}--\eqref{qk+1 opt} as follows: Since $\operatorname{Null}(\mathbf{L}) = \mathcal{S} := \{\mathbf{x} \in \Re ^{Nd}|x_1 = \cdots = x_N\}$, we reformulate the constraint $\bH^{\frac{1}{2}}\mathbf{x}=0$ in problem~\eqref{p2} as $\mathbf{L}^{\frac{1}{2}}\mathbf{x}=0$. Accordingly, we replace $\mathbf{H}$ in the primal-dual updates \eqref{argmin approximation primal}--\eqref{argmin approximation dual} with $\mathbf{L}$ and additionally require $\mathbf{B}^k\succ \mathbf{O}_{Nd}$. Similar to the discussion below \eqref{argmin approximation primal}--\eqref{argmin approximation dual}, the resulting updates attempt to minimize the following Augmented Lagrangian (AL) function
	\begin{equation}
		\operatorname{AL}(\mathbf{x},\mathbf{v})=\tilde{f}(\mathbf{x})+(\mathbf{v})^{\mathsf{T}} \mathbf{L}^{\frac{1}{2}}\mathbf{x}+\frac{\rho}{2}\|\mathbf{x}\|^2_\mathbf{L}. \label{AL}
	\end{equation}
	For convenience, we denote $\operatorname{AL}(\mathbf{x}^k,\mathbf{v}^k)$ by $\operatorname{AL}^k$. Moreover, the first-order optimality condition arising from such a primal update (with $\theta=1$) gives
	\begin{equation}
		\label{first-order opt modified} \nabla \tilde{f}(\mathbf{x}^k) + \mathbf{L}^{\frac{1}{2}}\mathbf{v}^k + \rho \mathbf{L} \mathbf{x}^{k+1}+ \mathbf{B}^k(\mathbf{x}^{k+1} - \mathbf{x}^k)=0,
	\end{equation}
	which is an important equation for our derivations. Then, the update equations \eqref{xk+1 opt}--\eqref{qk+1 opt} of UPP-SC can be obtained by the changes of variable $\mathbf{q}^k=\mathbf{L}^{\frac{1}{2}}\mathbf{v}^k$, $\mathbf{G}^k = (\mathbf{B}^k+\rho \mathbf{L})^{-1}$. Note that \eqref{qk+1 opt} is equivalent to 
	\begin{equation}
		\mathbf{v}^{k+1}=\mathbf{v}^k+\rho \mathbf{L}^{\frac{1}{2}}\mathbf{x}^{k+1}. \label{vk+1 opt}
	\end{equation}

With the above reinterpretation, we provide the convergence analysis for UPP-SC. The following lemma connects the primal and dual variables, in which we define 
\begin{gather}
	\label{kappa B} \lambda_1^{\mathbf{B}}:= \max_{k \geq 0} \lambda_1^{\mathbf{B}^k}, \quad \lambda_N^{\mathbf{B}}:= \min_{k \geq 0} \lambda_N^{\mathbf{B}^k}, \quad \kappa_{\mathbf{B}} = {\lambda_1^{\mathbf{B}}}/{\lambda_N^{\mathbf{B}}},\\
	\label{tilde kappa} \tilde{\kappa} := {\lambda_1^{\mathbf{B}}}/({\rho \lambda_{N-1}^{\mathbf{L}}}), \\ 
	\label{wk+1} \mathbf{w}^{k+1} := (\mathbf{x}^{k+1}-\mathbf{x}^{k})-(\mathbf{x}^{k}-\mathbf{x}^{k-1}).
\end{gather}

\begin{lemma} \label{lemma vk+1 - vk}
	Suppose Assumptions~\ref{assumption smooth}--\ref{assumption: distributed} hold. Let $\{\mathbf{x}^k\}$ be the sequence generated by UPP-SC with $\mathbf{O}\prec \lambda_N^{\mathbf{B}^k} \mathbf{I}\preceq \mathbf{B}^{k}\preceq \lambda_1^{\mathbf{B}^k}\mathbf{I} $. Then, for all $k\geq 0$,
	\begin{align}
		&\rho\|\mathbf{x}^{k+1}\|^2_{\mathbf{L}}=\frac{1}{\rho}\|\mathbf{v}^{k+1}-\mathbf{v}^k\|^2\leq 3\tilde{\kappa} \big((c_{\mathbf{B}}+\bar{M}^2)\|\mathbf{x}^k-\mathbf{x}^{k-1}\|^2_{(\mathbf{B}^{k})^{-1}}+\|\mathbf{w}^{k+1}\|^2_{\mathbf{B}^k}\big), \label{vk+1 - vk}
	\end{align}
	where $c_{\mathbf{B}}=4(\lambda_1^\mathbf{B})^2$ and $\{\mathbf{v}^k\}_{k=0}^\infty$ is given by \eqref{vk+1 opt}.
\end{lemma}
\begin{proof}
	See Appendix~\ref{proof lemma vk+1 - vk}.
\end{proof}

In the next lemma, we bound the decrease of the AL function \eqref{AL}.

\begin{lemma} \label{lemma ALk+1 - ALk}
	Suppose all the conditions in Lemma~\ref{lemma vk+1 - vk} hold. For all $k\geq 0$,
	\begin{align}
		&\operatorname{AL}^{k+1}-\operatorname{AL}^{k} \leq  -\|\mathbf{x}^{k+1}-\mathbf{x}^{k}\|^2_{\mathbf{B}^k+\frac{1}{2}\rho \mathbf{L}-\frac{\bar{M}}{2}\mathbf{I}_{Nd}} + \rho\|\mathbf{x}^{k+1}\|^2_{\mathbf{L}}. \label{ALk+1 - ALk lemma}
	\end{align}
\end{lemma}


Subsequently, we generate the following decreasing sequence (for some $\tilde{c},\tilde{\kappa}>0$):
\begin{align}
	\tilde{P}^{k+1}=&\operatorname{AL}^{k+1}+\|\mathbf{x}^{k+1}-\mathbf{x}^k\|^2_{3\tilde{\kappa} (c_{\mathbf{B}}+\bar{M}^2)(\mathbf{B}^{k+1})^{-1}}+\frac{\tilde{c}}{2}\big(\rho\|\mathbf{x}^{k+1}\|^2_\mathbf{L}\notag\\
    &+\|\mathbf{x}^{k+1}-\mathbf{x}^k\|^2_{\mathbf{B}^{k+1}+3\bar{M}\mathbf{I}_{Nd}+3\tilde{\kappa} (c_{\mathbf{B}}+\bar{M}^2)(\mathbf{B}^{k+1})^{-1}+\rho \mathbf{L}}\big). \label{tilde Pk+1}
\end{align}
In the next lemma, we show that, with proper parameters, the auxiliary functions will decrease along iterations. 

\begin{lemma}\label{lemma tildePk+1 nonincreasing}
	Suppose all the conditions in Lemma~\ref{lemma vk+1 - vk} hold.  Let $\{\mathbf{x}^k\}$ be the sequence generated by UPP-SC with the parameters selected as follows:
	\begin{align}
		\label{define tilde c}& 1 \geq \tilde{c}\geq 9\tilde{\kappa}=\max_{k\geq 0}\frac{9\lambda_1^{\mathbf{B}^k}}{\rho \lambda_{N-1}^{\mathbf{L}}}>0, \\
		\label{Bk+rhoL-c/2Bk-1}& \frac{1}{4}(\mathbf{B}^k+\rho \mathbf{L})-\frac{\tilde{c}}{2}(\mathbf{B}^{k+1}-\mathbf{B}^k+\mathbf{B}^{k-1})\!-\!(\frac{1}{2}+2\tilde{c})\bar{M}\mathbf{I}_{Nd}-\frac{\tilde{c}(2+\tilde{c})}{6}(c_{\mathbf{B}}+\bar{M}^2)(\mathbf{B}^{k+1})^{-1} \succeq \mathbf{O}_{Nd}.
	\end{align}
	Then, for all $k\geq 0$, we have
	\begin{align}
		\tilde{P}^{k+1}-\tilde{P}^{k}\leq& -\|\mathbf{x}^{k+1}-\mathbf{x}^k\|^2_{\frac{1}{4}(\mathbf{B}^k+\rho \mathbf{L})}-\frac{\tilde{c}}{2} \|\mathbf{w}^{k+1}\|^2_{\frac{(1-\tilde{c})}{3}\mathbf{B}^k+\mathbf{B}^{k-1}}\leq 0.  \label{tildePk+1 - tildePk lemma}
	\end{align}
\end{lemma}

\begin{proof}
	See Appendix~\ref{proof lemma tildePk+1 nonincreasing}.
\end{proof}

Next, we show the boundedness of $\tilde{P}^{k+1}$.
\begin{lemma}\label{lemma tildePk+1 geq f*}
	Suppose all the conditions in Lemma~\ref{lemma vk+1 - vk} hold. Let $\{\mathbf{x}^k\}$ be the sequence generated by UPP-SC with the parameters selected by \eqref{define tilde c} and \eqref{Bk+rhoL-c/2Bk-1}. Also let $\mathbf{x}^{-1}=\mathbf{v}^{-1}=0$, and $\mathbf{G}^{-1}=\mathbf{G}^0$. Then,
	\begin{equation*}
		\tilde{P}^{k+1}\geq f^*, \forall k >0, \quad \tilde{P}^{0}\leq \tilde{f}(\mathbf{x}^0)+\frac{6\tilde{c}+7}{8\bar{M}}\|\nabla \tilde{f}(0)\|^2,
	\end{equation*}
	where $f^*$ is defined in Assumption~\ref{assumption: optimality}.
\end{lemma}

\begin{proof}
	See Appendix~\ref{proof lemma tildePk+1 geq f*}.
\end{proof}

As the sequence $\{\tilde{P}^{k}\}$ is nonincreasing and bounded below, we deduce the convergence result of UPP-SC.
\begin{theorem}\label{theorem UPP-SC sublinear time-varying Bk}
	Suppose all the conditions in Lemma~\ref{lemma tildePk+1 geq f*} hold. Let $\{\mathbf{x}^k\}$ be the sequence generated by UPP-SC with the parameters selected as follows: $\xi_5=\frac{4\kappa_{\mathbf{B}}^2}{3}+\frac{2\kappa_{\mathbf{B}}-1}{2}$, $\xi_6= \frac{2\kappa_{\mathbf{B}}^2}{3}$, $0<\tilde{c}<\frac{-\xi_5+\sqrt{\xi_5^2+\xi_6}}{2\xi_6}$, $d_1= \frac{1}{4}- \xi_5 \tilde{c}-\xi_6 \tilde{c}^2$, $d_2 = (\frac{1}{2}+2\tilde{c})\bar{M}$, $d_3=\frac{(2+\tilde{c})\tilde{c}}{6}\bar{M}^2$, $\lambda_1^{\mathbf{B}}\geq\lambda_N^{\mathbf{B}}\geq \frac{d_2+\sqrt{d_2^2+4d_1d_3}}{2d_1}$ and $\rho \geq \frac{9\kappa_{\mathbf{B}}\lambda_N^{\mathbf{B}}}{\tilde{c}\lambda_{N-1}^{\mathbf{L}}}$.
	Let $T_k$ denote the iteration index in which UPP-SC satisfies
	\begin{equation}
		e(T_k):= \min_{k\in[T_k]}\frac{1}{N}\|\sum_{i=1}^N \nabla {f}_i(x_i^k)\|^2+\rho\|\mathbf{x}^{k}\|^2_\mathbf{L}\leq \epsilon, \label{eT time-varying Bk}
	\end{equation}
	where the error $\epsilon>0$. Then, the error above has the following bound:
	\begin{equation}
		e(T_k) \leq \tilde{C}_1\times {\tilde{C}_2}/{T_k}, \label{epsilon leq tildeC1C2/T}
	\end{equation}
	where $\tilde{C}_1:= \tilde{f}(\mathbf{x}^0)-f^*+\frac{6\tilde{c}+7}{8\bar{M}}\|\nabla \tilde{f}(0)\|^2$ and $\tilde{C}_2:= 4{\lambda}_1^{\mathbf{B}}+\frac{4}{3(1-\tilde{c})}(1+\frac{18d_1}{\tilde{c}(\tilde{c}+2)})+8(\frac{6d_1}{\tilde{c}(\tilde{c}+2)}+1).$
\end{theorem}
\begin{proof}
    See Appendix~\ref{proof theorem UPP-SC sublinear time-varying Bk}.
\end{proof}

Based on the parameter selections in Theorem~\ref{theorem UPP-SC sublinear time-varying Bk}, we set the parameters $\tilde{c}=\mathcal{O}(1)$. These imply that UPP-SC converges at a rate of $\mathcal{O}(\lambda_1^\bB/T)$ to stationarity and reaches an $\epsilon$-stationary solution (where $\epsilon>0$) within $\mathcal{O}(\lambda_1^\bB\tau/\epsilon)$ communication rounds.

\subsection{Optimal communication complexity bound}\label{convergence of UPP-SC-OPT}
In this subsection, we analyze the optimal communication complexity for UPP-SC-OPT, which is a special case of UPP-SC (with $\mathbf{G}^k=\mu \mathbf{I}_{Nd}$ for $k\geq 0$). Based on Theorem~\ref{theorem UPP-SC sublinear time-varying Bk}, we provide the following convergence result for UPP-SC-OPT.

\begin{theorem}\label{theorem UPP-SC-OPT}
	Suppose all the conditions in Lemma~\ref{lemma tildePk+1 geq f*} hold. Let $\{\bx^k\}$ be the sequence generated by UPP-SC-OPT with the parameters selected as follows: $0<\tilde{c}<\frac{1}{2}$, $d_4= \frac{1}{4}- \frac{\tilde{c}}{2}$, $\lambda_N^{\mathbf{B}}\geq \frac{d_2+\sqrt{d_2^2+4d_4d_3}}{2d_4}$, $\lambda_1^{\mathbf{B}}=\frac{1}{\mu}>\frac{9\lambda_N^\bB}{\tilde{c}}\Big((2\kappa_\bL-1)+\sqrt{(2\kappa_\bL-1)^2-1}\Big)$ and $\rho \geq \frac{9\kappa_{\mathbf{B}}\lambda_N^{\mathbf{B}}}{\tilde{c}\lambda_{N-1}^{\mathbf{L}}}$. 
	Then, the optimality error defined in \eqref{eT time-varying Bk} has the following bound:
	\vspace{-0.2cm}
	\begin{equation}
		e(T_k) \leq \hat{C}_1\times {\hat{C}_2}/{T_k}, \label{epsilon leq C1C2/T}
	\end{equation}
	where $\hat{C}_1:= \tilde{f}(\mathbf{x}^0)-f^*+\frac{6\tilde{c}+7}{8\bar{M}}\|\nabla \tilde{f}(0)\|^2$ and $\hat{C}_2:= 4{\lambda}_1^{\mathbf{B}}+\frac{4}{3(1-\tilde{c})}(1+\frac{18d_4}{\tilde{c}(\tilde{c}+2)})+8(\frac{6d_4}{\tilde{c}(\tilde{c}+2)}+1).$
\end{theorem}

\begin{proof}
	See Appendix~\ref{proof theorem UPP-SC-OPT}.
\end{proof}

Now we let $\tilde{c}=\frac{1}{4}$ and $\lambda_N^\bB=\mathcal{O}(\bar{M})$, and have $\lambda_1^\bB=\mathcal{O}(\bar{M}\kappa_\bL)$. Since $\mathbf{B}=\bG^{-1}-\rho\bL=(\frac{1}{\mu}\mathbf{I}_{Nd}-\rho \mathbf{L})$ implies that $\lambda_1^\bB=\frac{1}{\mu}$, UPP-SC-OPT reaches $\epsilon$-stationarity within $\mathcal{O}(\bar{M}\kappa_\bL/\epsilon)$ iterations.

Subsequently, we analyze the communication complexity bound of UPP-SC-OPT based on Chebyshev iterations.

\begin{theorem}\label{theorem optimal communication complexity bound}
	Suppose all the conditions in Theorem~\ref{theorem UPP-SC-OPT} hold. For $k\geq 0$, let $\tau^k=\tau=\lceil {\sqrt{\gamma}} \rceil$, $\mathbf{L}=P_{\tau}(\mathbf{H})/\lambda_1^{P_{\tau}(\mathbf{H})}$, employ the Chebyshev polynomial to $P_{\tau}(\mathbf{H})$ as is shown in Oracle~\ref{oracle cheb}, and choose the parameters as in Theorem~\ref{theorem UPP-SC-OPT}. Denote $T_c$ as the communication rounds needed to reach the $\epsilon$-stationary solution defined in \eqref{eT time-varying Bk}. Then, UPP-SC-OPT achieves the following communication complexity bound:
	\vspace{-0.2cm}
	\begin{equation}
		\label{Tc leq} T_c =\mathcal{O}({\bar{M}\sqrt{\gamma}}/{\epsilon}).
	\end{equation}
\end{theorem}

\begin{proof}
	See Appendix~\ref{proof theorem optimal communication complexity bound}.
\end{proof}

Note that our problem and network settings described by Assumptions~\ref{assumption smooth}--\ref{assumption polynomial} are encompassed within the specified class described in \cite{sun_distributed_2019}. From the perspective of the algorithm class, \cite{sun_distributed_2019} restricts that each node $i$ only broadcasts its local variable $x_i$ to its neighbors per iteration—a condition fulfilled by UPP-SC-OPT. Consequently, UPP-SC-OPT emerges as a communication-wise optimal method, aligning with the lower bound of communication complexity, as is stated in \cite{sun_distributed_2019,mancino2023decentralized}.
  
\section{Numerical results}\label{section: numerical results}
In this section, we evaluate the convergence performance of various versions of UPP via numerical examples. 

\begin{figure*}[ht]
	\centering
    \begin{minipage}{0.25\linewidth}
        \begin{subfigure}[]{
            \centering
            \includegraphics[height=4.3cm]{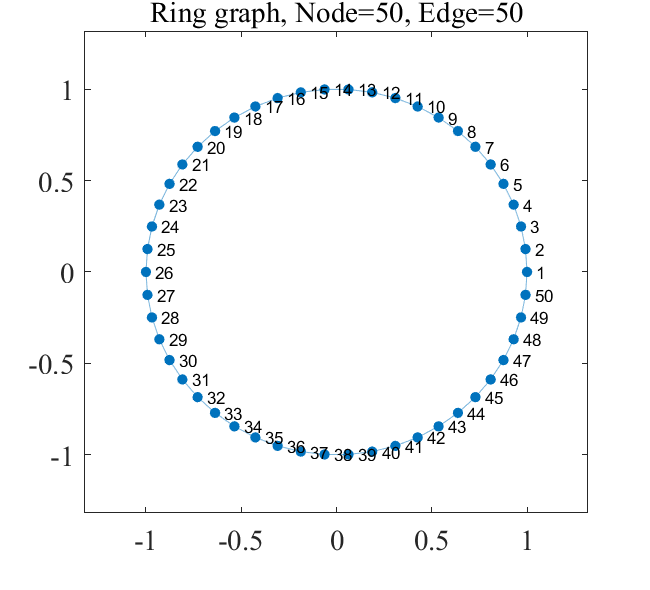}
            \label{ring graph}}
        \end{subfigure}
	\end{minipage}
	\begin{minipage}{0.3\linewidth}
	\begin{subfigure}[]{
		\centering
		\includegraphics[height=4.3cm]{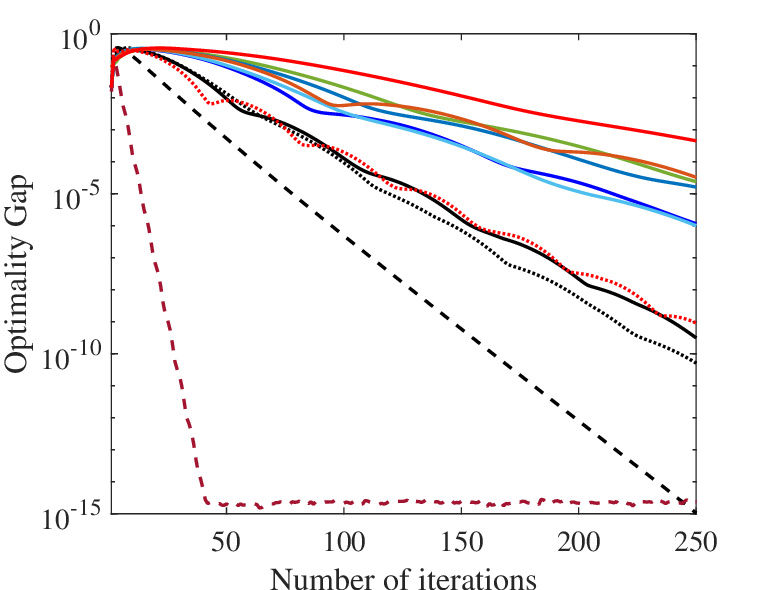}
		\label{ring graph iteration}
		}\end{subfigure}
	\end{minipage}
	\begin{minipage}{0.43\linewidth}
		\begin{subfigure}[]{
			\centering
			\includegraphics[height=4.3cm]{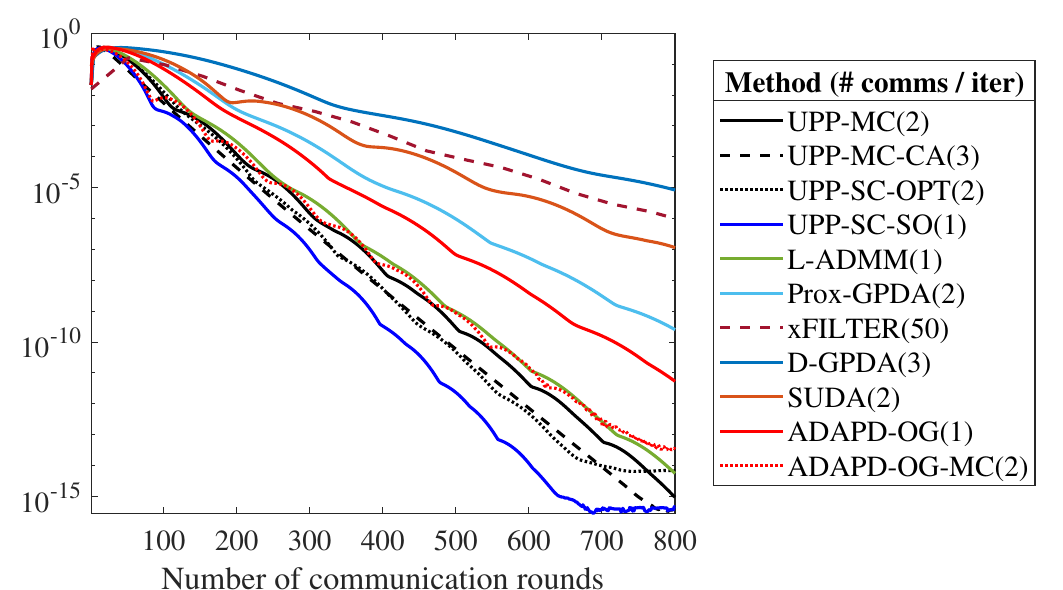}
			\label{ring graph communication}
			}\end{subfigure}
		\end{minipage}
	\caption{Convergence performance of related works on ring graph (${\gamma}=253.64$). In the legend, the number in the parentheses represents the number of communication rounds per iteration, and this notation applies equally to Fig. ~\ref{simulation for grid graph}, Fig. ~\ref{simulation for geometric graph} and Fig. ~\ref{simulation for regular graph}. \label{simulation for ring graph}}
\end{figure*}

\begin{figure*}[ht]
	\centering
    \begin{minipage}{0.25\linewidth}
        \begin{subfigure}[]{
            \centering
            \includegraphics[height=4.3cm]{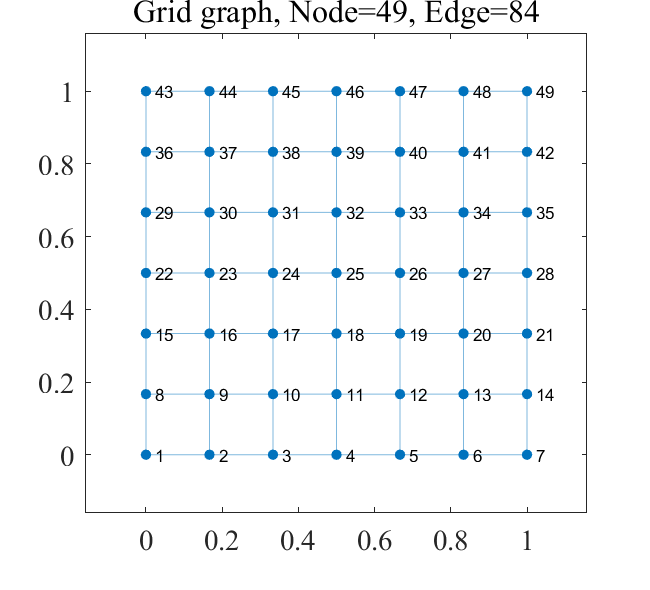}
            \label{grid graph}}
        \end{subfigure}
	\end{minipage}
	\begin{minipage}{0.3\linewidth}
	\begin{subfigure}[]{
		\centering
		\includegraphics[height=4.3cm]{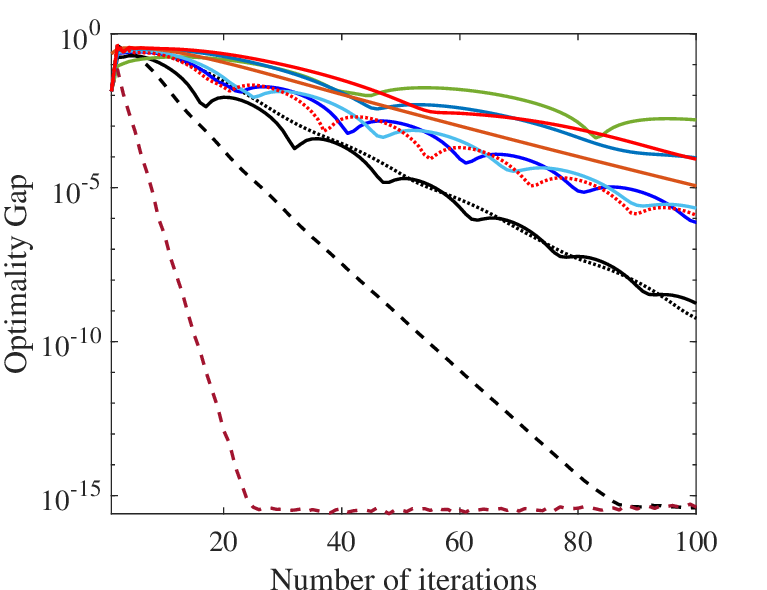}
		\label{grid graph iteration}
		}\end{subfigure}
	\end{minipage}
	\begin{minipage}{0.43\linewidth}
		\begin{subfigure}[]{
			\centering
			\includegraphics[height=4.3cm]{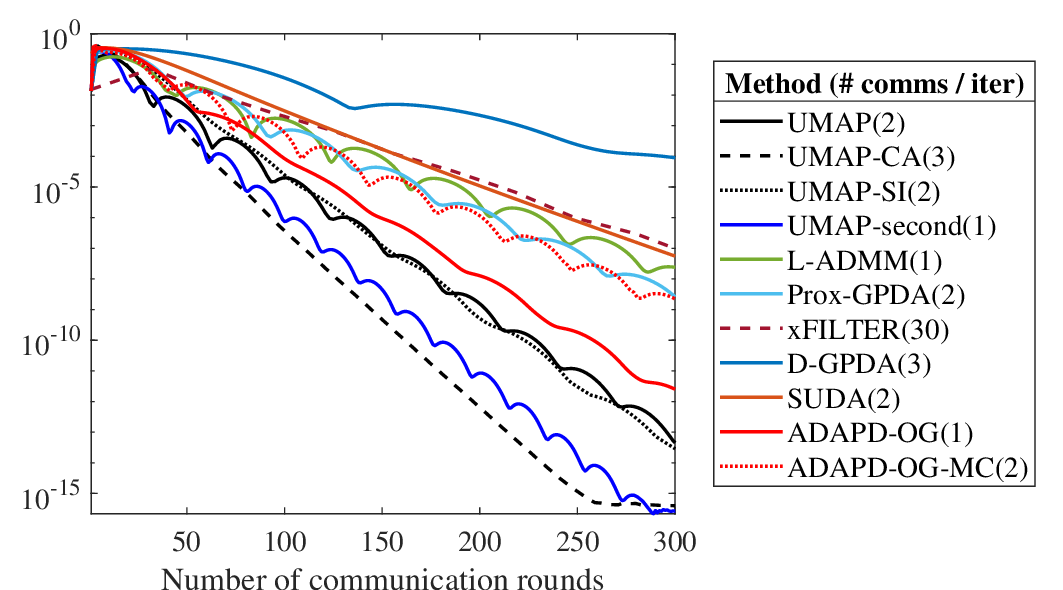}
			\label{grid graph communication}
			}\end{subfigure}
		\end{minipage}
	\caption{Convergence performance of related works on grid graph (${\gamma}=36.3$).\label{simulation for grid graph}}
\end{figure*}

\begin{figure*}[ht]
	\centering
    \begin{minipage}{0.25\linewidth}
        \begin{subfigure}[]{
            \centering
            \includegraphics[height=4.3cm]{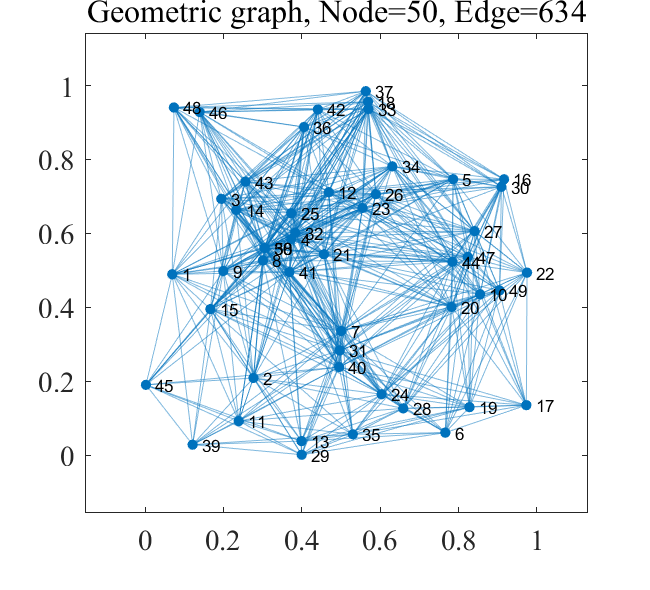}
            \label{geometric graph}}
        \end{subfigure}
	\end{minipage}
	\begin{minipage}{0.3\linewidth}
	\begin{subfigure}[]{
		\centering
		\includegraphics[height=4.3cm]{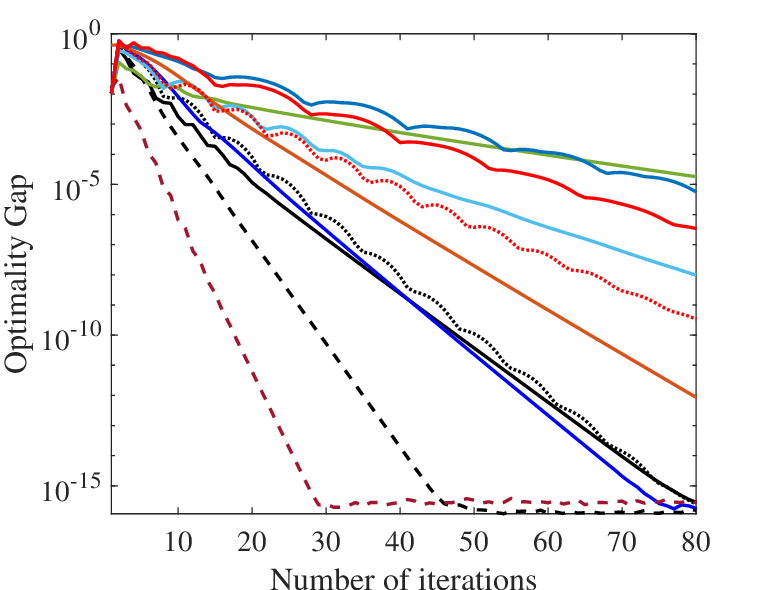}
		\label{geometric graph iteration}
		}\end{subfigure}
	\end{minipage}
	\begin{minipage}{0.43\linewidth}
		\begin{subfigure}[]{
			\centering
			\includegraphics[height=4.3cm]{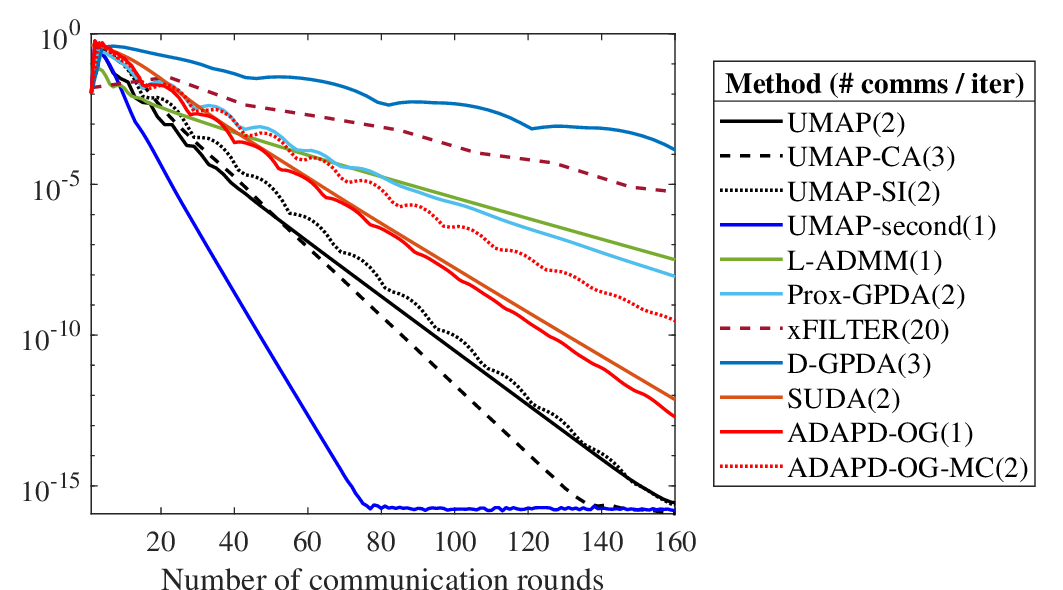}
			\label{geometric graph communication}
			}\end{subfigure}
		\end{minipage}
	\caption{Convergence performance of related works on geometric graph (${\gamma}=3.98$) with $r=0.5$.\label{simulation for geometric graph}}
\end{figure*}

\begin{figure*}[ht]
	\centering
    \begin{minipage}{0.25\linewidth}
        \begin{subfigure}[]{
            \centering
            \includegraphics[height=4.3cm]{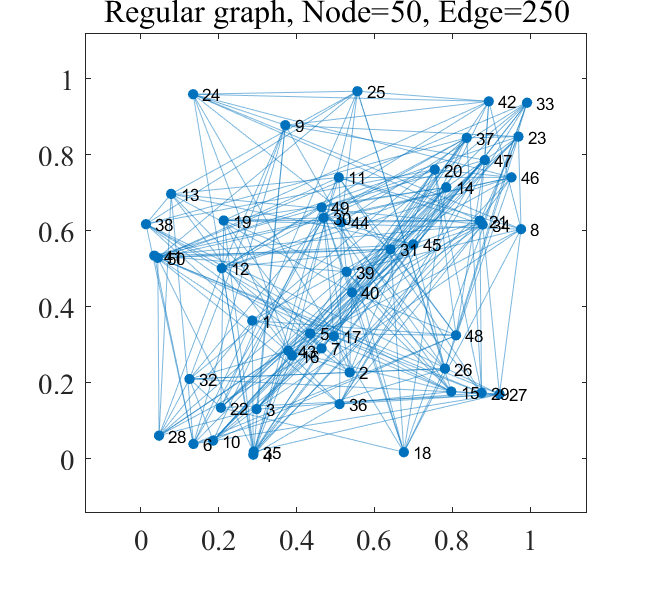}
            \label{regular graph}}
        \end{subfigure}
	\end{minipage}
	\begin{minipage}{0.3\linewidth}
	\begin{subfigure}[]{
		\centering
		\includegraphics[height=4.3cm]{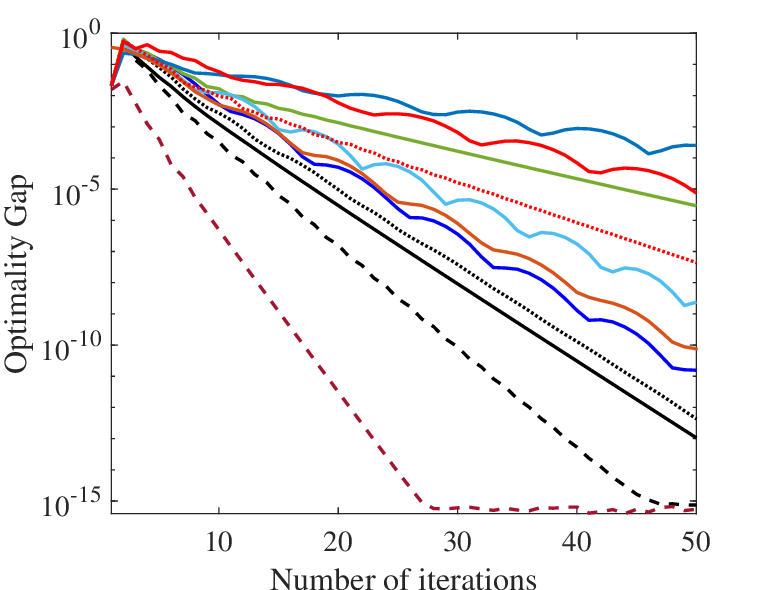}
		\label{regular graph iteration}
		}\end{subfigure}
	\end{minipage}
	\begin{minipage}{0.43\linewidth}
		\begin{subfigure}[]{
			\centering
			\includegraphics[height=4.3cm]{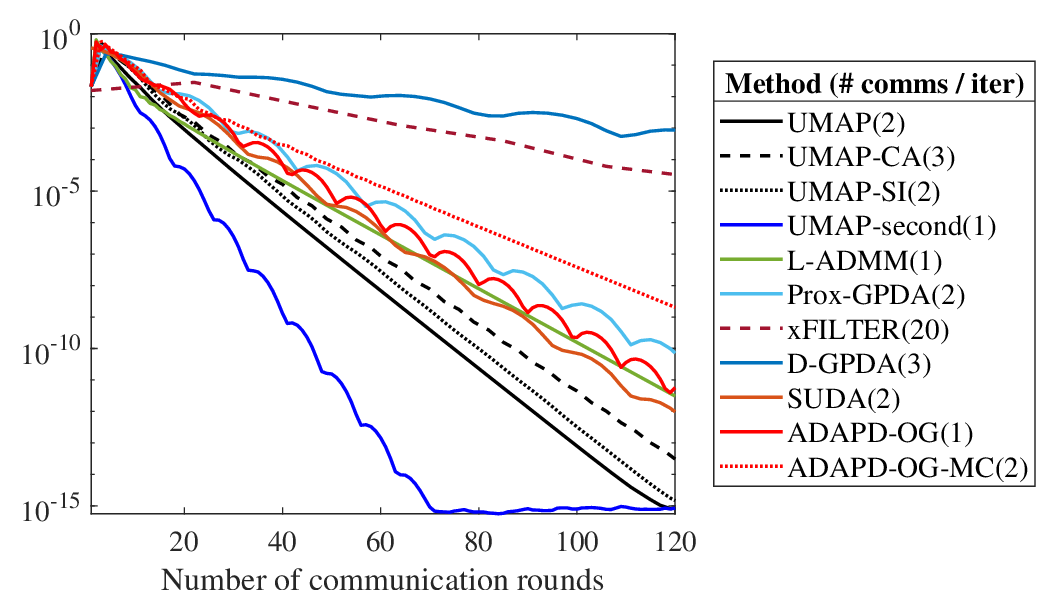}
			\label{regular graph communication}
			}\end{subfigure}
		\end{minipage}
	\caption{Convergence performance of related works on regular graph (${\gamma}=3.29$) with degree $=10$.\label{simulation for regular graph}}
\end{figure*}

We consider the distributed binary classification problem with nonconvex regularizers \cite{antoniadis2011penalized}, which satisfies Assumptions~\ref{assumption smooth}--\ref{assumption polynomial} and takes the form of \eqref{p1} with
\begin{equation*}
	f_i(x)=\frac{1}{m}\sum_{s}^{m}\log(1+\exp(-y_{is}x^{\mathsf{T}}z_{is}))+\sum_{t=1}^{d}\frac{\lambda\mu([x]_t)^2}{1+\mu([x]_t)^2},
\end{equation*}
where $m$ is the number of data samples of each node, $\lambda,\mu$ denote the regularization parameters, $y_{is}\in \{-1,1\}$ and $z_{is}\in \mathbb{R}^d$ denote the label and the feature for the $s$-th data sample of node $i$, respectively. In the simulation, we set $N=50$, $d=10$, $m=200$ with $\lambda=0.001$ and $\mu=1$. Besides, we randomly generate $y_{is}$ and $z_{is}$ for each node $i$, and construct ring graph, grid graph, regular graph with degree $=10$ and geometric graph with the radius threshold $r=0.5$.

We compare our proposed UPP-MC, UPP-MC-CA, UPP-SC-OPT and UPP-SC-SO with L-ADMM \cite{yi2022sublinear}, SUDA \cite{alghunaim_unified_2022}, Prox-GPDA \cite{hong_prox-pda_2017}, D-GPDA \cite{sun2018distributed}, ADAPD-OG \cite{mancino2023decentralized} and its multi-communication variant ADAPD-OG-MC which utilizes Chebyshev acceleration to facilitate convergence. We also include the xFILTER algorithm \cite{sun_distributed_2019}, which employs Chebyshev iterations to approximate the minimization step. In SUDA, we let $\mathbf{A}=\mathbf{I}-\mathbf{H},\mathbf{B}=\mathbf{H}^{1/2},\mathbf{C}=\mathbf{I}$; In ADAPD-OG-MC, we set its inner loop $R=2$; In UPP-MC, we let $\bD^k=\tilde{\bD}^k=\mathbf{H}$ and $\bG^k=\bI-\alpha\bH$; In UPP-MC-CA and UPP-SC-OPT, we fix $\tau=2$ for Chebyshev acceleration. We hand-optimize the parameters of all these algorithms, and measure their performance by the optimality gap given by the optimality gap defined as: $\|\nabla \tilde{f}(\mathbf{x})\|^2+\|\bH^{\frac{1}{2}}\mathbf{x}\|^2.$

Figure~\ref{ring graph iteration}, \ref{grid graph iteration}, \ref{geometric graph iteration} and \ref{regular graph iteration} illustrate the optimality gap versus the number of iterations. As UPP-MC-CA, UPP-SC-OPT, ADAPD-OG-MC and xFILTER require inner loops in each iteration, which increase the communication cost, we also compare their optimality gaps with respect to the communication rounds in Figure~\ref{ring graph communication}, \ref{grid graph communication}, \ref{geometric graph communication}, \ref{regular graph communication}. Note that each node transmits a local vector to its neighbors in a communication round and we mark the number of communication rounds within each iteration for the algorithms in the legends. 

In iteration-wise comparative experiments (as is illustrated in Figures~\ref{ring graph iteration}, \ref{grid graph iteration}, \ref{geometric graph iteration}, \ref{regular graph iteration}), our proposed UPP-MC, UPP-MC-CA, UPP-SC-OPT and UPP-SC-SO exhibit faster convergence than all existing baselines—except xFILTER. However, xFILTER demands excessive communication rounds per iteration, leading to significantly inferior communication-round-wise convergence performance, as is shown in Figures~\ref{ring graph communication}, \ref{grid graph communication}, \ref{geometric graph communication}, \ref{regular graph communication}. 

In terms of communication rounds, our first-order algorithms UPP-MC, UPP-SC-OPT and UPP-MC-CA converge faster than competing first-order baselines (L-ADMM, Prox-GPDA, xFILTER, D-GPDA, SUDA, ADAPD-OG, ADAPD-OG-MC). Moreover, our second-order method UPP-SC-SO delivers the fastest convergence among all compared first-order algorithms, at the cost of extra Hessian matrix computations.

We further analyze the impact of network topology on convergence performance. As shown in Figure~\ref{ring graph communication}, \ref{grid graph communication}, \ref{geometric graph communication} and \ref{regular graph communication}, UPP-MC-CA and UPP-SC-OPT—equipped with Chebyshev acceleration—outperform UPP-MC in sparse networks (characterized by a large condition number $\gamma$). In dense networks, however, UPP-MC exhibits superior performance to the two accelerated variants, as evidenced by Figure~\ref{regular graph communication}.
Notably, although UPP-SC-OPT theoretically achieves optimal communication complexity \cite{sun_distributed_2019}, it does not always yield the best communication-round-wise convergence in practice. This discrepancy arises because the optimal bound in \cite{sun_distributed_2019} applies specifically to algorithms that only exchange local decisions over the network, whereas UPP-MC and UPP-MC-CA additionally transmit gradients of the augmented Lagrangian. This also validates the effectiveness of designing the matrix $\bG^k$ as a polynomial of the weight matrix in UPP-MC.

\section{Conclusion}\label{section: conclusion}
We have developed a unifying primal-dual framework, referred to as UPP, to tackle distributed nonconvex optimization problems. We develop two distributed realizations of UPP, i.e., UPP-MC and UPP-SC, which can be specialized into various existing methods by adjusting the parameter configurations, encompassing both first-order and second-order algorithms. Theoretical analysis demonstrates that UPP-MC and UPP-SC converge to stationarity at a sublinear rate and UPP-MC further achieves linear convergence to the global optimum under the P-{\L} condition. These results enrich the applicability and the convergence rate guarantees of multiple prior distributed optimization algorithms. Moreover, we leverage Chebyshev polynomials to facilitate the communication process of our algorithms, enabling UPP-SC-OPT to attain the optimal communication complexity bound. Through numerical experiments under different network topologies, we compare our algorithms with multiple state-of-the-art approaches and validate the superior performance of our proposed methods in both convergence speed and communication efficiency.

\section{APPENDIX} \label{appendix}
\subsection{Proof of Proposition~\ref{proposition parameter range}}\label{proof proposition}
(\romannumeral1) In order to present our convergence results, we first introduce the following notations:
\begin{align}
	\mathbf{K} =& (\mathbf{I}_N - \frac{1}{N} \mathbf{1}_N \mathbf{1}_N ^{\mathsf{T}}) \otimes \mathbf{I}_d = (\tilde{\mathbf{R}}\operatorname{diag}(0,1,\dots,1) \tilde{\mathbf{R}}^\mathsf{T}) \otimes \mathbf{I}_d, \notag\\
	\mathbf{J} =& \frac{1}{N}\mathbf{1}_N \mathbf{1}_N ^{\mathsf{T}} \otimes \mathbf{I}_d = (\tilde{\mathbf{R}}\operatorname{diag}(1,0,\dots,0) \tilde{\mathbf{R}}^\mathsf{T}) \otimes \mathbf{I}_d, \notag\\
	\bar{x}^k =& \frac{1}{N}(\mathbf{1}_N^{\mathsf{T}} \otimes \mathbf{I}_d) \mathbf{x}^k, \quad \bar{\mathbf{x}}^k = \mathbf{J} \mathbf{x}^k,\quad \mathbf{s}^k=\mathbf{q}^k+\frac{1}{\theta}\nabla \tilde{f}(\bar{\mathbf{x}}^k), \notag\\
	\mathbf{g}^k =& \nabla \tilde{f}(\mathbf{x}^k), \quad \bar{\mathbf{g}}^k = \mathbf{J} \mathbf{g}^k, \quad \mathbf{g}_a^k = \nabla \tilde{f}(\bar{\mathbf{x}}^k), \quad \bar{\mathbf{g}}_a^k = \mathbf{J} \mathbf{g}_a^k, \notag
\end{align}
where $\tilde{\mathbf{R}}\in \Re ^{N \times N}$ with its first column $\mathbf{r} = \frac{1}{\sqrt{N}} \mathbf{1}_N$. Hence, we have $\mathbf{K}^2=\mathbf{K}$, $\mathbf{J}^2=\mathbf{J}$, $\mathbf{J}\mathbf{q}^k = \mathbf{J}\bH^{\frac{1}{2}}\mathbf{v}^k= \mathbf{0}_{Nd}$.

We first establish some results for the mixing matrices. From Assumption~\ref{assumption: Dk}, we have $ \operatorname{Null} (\mathbf{D}^k) = \operatorname{Null} (\tilde{\mathbf{D}}^k) =\operatorname{Null} (\mathbf{K}) = \mathcal{S}$, and thus 
\begin{align}
	\label{KDk=DkK}\mathbf{K}\mathbf{D}^k = \mathbf{D}^k\mathbf{K} = \mathbf{D}^k, \\
	\label{KtildeDk=tildeDkK}\mathbf{K}\mathbf{\tilde{D}}^k = \mathbf{\tilde{D}}^k\mathbf{K} = \mathbf{\tilde{D}}^k.
\end{align}
Also, we have $\mathbf{J}\mathbf{D}^k = \mathbf{J}\tilde{\mathbf{D}}^k=\mathbf{O}_{Nd}$. Here, we define ${\kappa_{\mathbf{D}}^k} = \frac{\lambda_{1}^{\mathbf{D}^k}}{\lambda_{N-1}^{\mathbf{D}^k}}\geq 1$. Then, Assumption~\ref{assumption: Dk} implies
	\begin{align}
		\mathbf{O}_{Nd} \preceq&(\lambda_{1}^{\mathbf{D}^k}/\kappa_{\mathbf{D}}^k)\mathbf{K} = \lambda_{N-1}^{\mathbf{D}^k}\mathbf{K} \preceq \mathbf{D}^k \preceq \lambda_{1}^{\mathbf{D}^k}\mathbf{K}. \label{Dk range}
	\end{align}

Moreover, $\bD^k$ and $\tilde{\bD}^k$ possess the eigen-decompositions as $\mathbf{D}^k = (\tilde{\mathbf{R}} \Lambda_{\mathbf{D}^k} \tilde{\mathbf{R}}^{\mathsf{T}})\otimes \mathbf{I}_d$ and $\tilde{\bD}^k = (\tilde{\mathbf{R}} \Lambda_{\tilde{\mathbf{D}}^k} \tilde{\mathbf{R}}^{\mathsf{T}})\otimes \mathbf{I}_d$, where $\Lambda_{\mathbf{D}^k}=\operatorname{diag}(0,\lambda_{N-1}^{\bD^k},\dots,\lambda_{1}^{\bD^k})$ and $\Lambda_{\tilde{\mathbf{D}}^k}=\operatorname{diag}(0,\lambda_{N-1}^{\tilde{\bD}^k},\dots,\lambda_{1}^{\tilde{\bD}^k})$. In particular, since $\mathbf{D}^k$ is a polynomial of $\mathbf{H}=\mathbf{P}\otimes \mathbf{I}_{d}$, each diagonal entry of $\Lambda_{\mathbf{D}^k}$ is the polynomial of each eigenvalue of $\mathbf{P}$ with the same coefficients. Next, we define $\mathbf{Q}^k := (\tilde{\mathbf{R}} \Lambda _{\tilde{\mathbf{D}}^k} ^{\dagger} \tilde{\mathbf{R}}^{\mathsf{T}}) \otimes \mathbf{I}_d$ such that 
\begin{equation}
	\mathbf{Q}^k \tilde{\mathbf{D}}^k = \tilde{\mathbf{D}}^k \mathbf{Q}^k= \mathbf{K}. \label{QktildeDk=tildeDkQk}
\end{equation}

From $\mathbf{G}^k = (\tilde{\mathbf{R}}\Lambda_{\mathbf{G}^k}\tilde{\mathbf{R}}^{\mathsf{T}}) \otimes \mathbf{I}_{d}$, where $\Lambda_{\mathbf{G}^k} = \operatorname{diag}({\zeta^k}, \lambda_2^{\mathbf{G}^k}, \dots, \lambda_N^{\mathbf{G}^k})$ with $\zeta^k=\lambda_1^{\mathbf{G}^k}>\lambda_2^{\mathbf{G}^k} \geq \dots\geq\lambda_N^{\mathbf{G}^k}$, together with ${\kappa_{\mathbf{G}}^k} = \frac{\lambda_{2}^{\mathbf{G}^k}}{\lambda_{N}^{\mathbf{G}^k}}\geq 1$, we have
\begin{align}
	\mathbf{O}_{Nd} \preceq (\lambda_{2}^{\mathbf{G}^k}/\kappa_{\mathbf{G}}^k)\mathbf{K} = \lambda_{N}^{\mathbf{G}^k}\mathbf{K} \preceq \mathbf{G}^k \mathbf{K} \preceq \lambda_{2}^{\mathbf{G}^k}\mathbf{K}. \label{Gk range}
\end{align}
For some $0<\underline{\varepsilon}<\bar{\varepsilon}<1$, we set
\begin{equation}
	\underline{\varepsilon}\mathbf{D}^k\mathbf{G}^k\preceq\tilde{\mathbf{D}}^k \preceq \bar{\varepsilon}\mathbf{D}^k\mathbf{G}^k, \quad \forall k \geq 0. \label{tildeD range}
\end{equation}
Thus,
\begin{equation}
	\frac{1}{\bar{\varepsilon}}\mathbf{K}\preceq\mathbf{G}^k\mathbf{D}^k\mathbf{Q}^k \preceq \frac{1}{\underline{\varepsilon}}\mathbf{K}, \quad \forall k\geq 0. \label{GkDkQk}
\end{equation}

(\romannumeral2) For $k\geq 0$, we define the following parameters:
\begin{align}
		{\epsilon_1^k} =&  \frac{\rho \lambda_{1}^{\mathbf{D}^k}(1-\bar{\varepsilon})}{\kappa_{\mathbf{D}}^k\kappa_{\mathbf{G}}^k} -\big(\frac{4+4\theta+\theta^2}{4\kappa_{\mathbf{G}}^k}+(\frac{1}{2}+\frac{1}{\theta})\bar{M}^2\big),\notag\\
		{\epsilon_2^k} =& (\rho\lambda_{1}^{\mathbf{D}^k})^2(2+\frac{3\bar{\varepsilon}}{2}+\frac{11\bar{\varepsilon}^2}{8})\notag\\
		&+ \big(2+\frac{5\bar{\varepsilon}}{2}+\frac{7\bar{\varepsilon}^2}{8}+\bar{\varepsilon}\rho\lambda_{1}^{\mathbf{D}^k}+\frac{1}{2}(\rho\lambda_{1}^{\mathbf{D}^k})^2\big)\bar{M}^2, \notag\\
		{\epsilon_3^k} =&\frac{\theta}{4\kappa_\bG^k},\quad {\epsilon_4^k} = \frac{1}{2}+(\frac{7}{2}+\bar{\varepsilon})\theta^2+\frac{1}{2}(\rho\lambda_1^{\bD^k}),\notag\\
		{\epsilon_{5}^k} =& {\xi_1^k}+\frac{{\xi_2^k}}{\lambda_{2}^{\mathbf{G}^k}} \text{, with } {\xi_1^k} = \frac{\bar{M}}{2}+\frac{\bar{M}^2}{\theta^2}(\frac{3}{2\bar{\varepsilon}}+\frac{7}{8}), \notag\\
		{\xi_2^k} =& \frac{\kappa_{\mathbf{G}}^k\bar{M}^2}{\theta^2}\big(\frac{1}{\theta}(\frac{1}{4}+\frac{1}{\bar{\varepsilon}})^2+\frac{1}{2}\big). \label{epsilon1-epsilon5}
	\end{align}
Moreover, We choose the algorithm parameters of \eqref{zk final}--\eqref{qk+1 final} in the following way:
\begin{gather}
	0< \bar{\varepsilon}<1,\quad \theta>\max_{k\geq 0}\{2\bar{M}\sqrt{\kappa_\bG^k},2(\frac{1}{4}+\frac{1}{\bar{\varepsilon}})^2\},\label{range varepsilon theta}\\
	\rho > \max_{k\geq 0}{\left(\frac{4+4\theta+\theta^2}{4\kappa_{\mathbf{G}}^k}+(\frac{1}{2}+\frac{1}{\theta})\bar{M}^2\right)\kappa_{\mathbf{D}}^k\kappa_{\mathbf{G}}^k}/\big({\lambda_{1}^{\mathbf{D}^k}(1-\bar{\varepsilon})}\big),\label{range rho}\\
	0<  \lambda_{2}^{\mathbf{G}^k} <\min\, \{ \frac{{\epsilon_1^k}}{{\epsilon_2^k}}, \frac{{\epsilon_3^k}}{\epsilon_4^k}, \frac{\frac{1}{4}-{\xi_2^k}}{\xi_1^k},\frac{1}{\bar{\varepsilon}\theta\rho \lambda_{1}^{\mathbf{D}^k}}\}, \label{range hat lambda Gk}\\
	\lambda_{2}^{\mathbf{G}^k} \leq \zeta^k < \min \{\frac{1}{4{\epsilon_{5}^k}}, \frac{\kappa_p^k\lambda_{2}^{\mathbf{G}^k}}{\kappa_p^k-1}\}, \label{range zetak}\\
	\frac{1}{\bar{\varepsilon}}<\frac{1}{\underline{\varepsilon}}\leq\frac{\theta\lambda_2^{\bG^k}}{4\kappa_\bG^k}+\frac{1}{\bar{\varepsilon}}<\frac{1}{4}+\frac{1}{\bar{\varepsilon}}.\label{range underline varepsilon}
\end{gather}

Subsequently, we show that the parameters in \eqref{epsilon1-epsilon5} and the parameter selections in \eqref{range varepsilon theta}--\eqref{range underline varepsilon} are well-defined.

	Considering $0<\bar{\varepsilon}<1$ from \eqref{range varepsilon theta}, we obtain $\rho>0$, making \eqref{range rho} well-defined. Additionally, from $0<\bar{\varepsilon}<1$ and \eqref{range rho}, we have $\epsilon_1^k>0$. Since $\theta>\max_{k\geq 0}\{2\bar{M}\sqrt{\kappa_\bG^k},2(\frac{1}{4}+\frac{1}{\bar{\varepsilon}})^2\}$ in \eqref{range varepsilon theta}, we obtain ${\xi_2^k}<\frac{1}{4}$, ensuring the well-posedness of \eqref{range hat lambda Gk}. Then, from $\lambda_{2}^{\mathbf{G}^k}<(\frac{1}{4}-{\xi_2^k})/{\xi_1^k}$, we have $\lambda_{2}^{\mathbf{G}^k}<\frac{1}{4\epsilon_{5}^k}$. Together with $\kappa_p^k > 1$, we show that \eqref{range zetak} is well-posed. From $\lambda_{2}^{\mathbf{G}^k}<\frac{\epsilon_3^k}{\epsilon_4^k}$ in \eqref{range hat lambda Gk}, we obtain $\frac{\theta\lambda_2^{\bG^k}}{\kappa_\bG^k}<1$, and thus \eqref{range underline varepsilon} holds.

(\romannumeral3) Next, we establish some results based on the parameter selections in  \eqref{range varepsilon theta}--\eqref{range underline varepsilon}. From $\lambda_{2}^{\mathbf{G}^k}<\frac{\epsilon_1^k}{\epsilon_2^k}$ in \eqref{range hat lambda Gk}, we obtain $ \rho \lambda_1^{\mathbf{D}^k}\lambda_2^{\mathbf{G}^k}<\frac{\epsilon_1^k\rho \lambda_1^{\mathbf{D}^k}}{\epsilon_2^k} < \frac{(\rho\lambda_1^{\mathbf{D}^k})^2}{2\kappa_{\mathbf{D}}^k\kappa_{\mathbf{G}}^k(\rho\lambda_1^{\mathbf{D}^k})^2}<1$. Hence,
	\begin{equation}
		\label{rhoDkGk} \rho \mathbf{D}^k \mathbf{G}^k \preceq \mathbf{K}.
	\end{equation}
    Together with \eqref{tildeD range}, we have
    \begin{equation}\label{rho tildeDk Gk preceq}
        \rho\tilde{\bD}^k\bG^k\preceq \bar{\varepsilon} \bG^k\mathbf{K}.
    \end{equation}
	From $\lambda_{2}^{\mathbf{G}^k}<\sqrt{\frac{1}{\bar{\varepsilon}\theta\rho\lambda_{1}^{\mathbf{D}^k}}}$ in \eqref{range hat lambda Gk}, 
	\begin{equation}
		\label{theta rho tildeDk Gk} \theta\rho\tilde{\mathbf{D}}^k\bG^k \preceq \mathbf{K}.
	\end{equation}
	Besides, $\zeta^k<\frac{1}{4\epsilon_5^k}$ in \eqref{range zetak} implies that
	\begin{equation}
		\label{zetak leq theta2} {\zeta^k}<\frac{\theta^2}{2\kappa_{\mathbf{G}}^k\bar{M}^2}\lambda_{2}^{\mathbf{G}^k}.
	\end{equation}
Because of \eqref{Gk} and because $\mathbf{J}P_{\tau_d^k}(\mathbf{H})=0$, we have
	\begin{align}
		\mathbf{J}\mathbf{G}^k = & \mathbf{J}(\zeta^k \mathbf{I}_{Nd}-\eta^k P_{\tau_d^k}(\mathbf{H}))=\zeta^k \mathbf{J}. \label{JGk}
	\end{align}
From \eqref{qk=tildeH1/2vk}, \eqref{zk final}, \eqref{xk+1 final}, \eqref{JGk} and $\mathbf{J}\mathbf{D}^k=\mathbf{J}\tilde{\mathbf{D}}^k=0$ that 
\begin{align}
	\bar{\mathbf{x}}^{k+1} - \bar{\mathbf{x}}^k = -\mathbf{J} \mathbf{G}^k(\mathbf{g}^k + \rho \mathbf{D}^k \mathbf{x}^k + \mathbf{q}^k) = - {\zeta^k}\bar{\mathbf{g}}^k. \label{bar xk+1 - bar xk}
\end{align}
Due to Assumption~\ref{assumption smooth}, we have
\begin{align}
	\|\mathbf{g}_a^k - \mathbf{g}^k\|^2 \leq \bar{M}^2\|\bar{\mathbf{x}}^k - \mathbf{x}^k\|^2 = \bar{M}^2 \|\mathbf{x}^k\|^2_\mathbf{K}. \label{g0k-gk}
\end{align}
Then, since $\lambda_1^{\mathbf{J}} = 1$, we have
\begin{align}
	\|\bar{\mathbf{g}}_a^k - \bar{\mathbf{g}}^k\|^2 = & \|\mathbf{J}(\mathbf{g}_a^k - \mathbf{g}^k)\|^2 \leq \bar{M}^2 \|\mathbf{x}^k\|^2_\mathbf{K}. \label{bar g0k- bar gk}
\end{align}
From Assumption~\ref{assumption smooth} and \eqref{bar xk+1 - bar xk}, we have
\begin{align}
	\|\mathbf{g}_a^{k+1} - \mathbf{g}_a^k\|^2 \leq \bar{M}^2 \|\bar{\mathbf{x}}^{k+1} - \bar{\mathbf{x}}^{k}\|^2 = (\zeta^k)^2 \bar{M}^2\|\bar{\mathbf{g}}^k\|^2. \label{g0k+1-g0k}
\end{align}

(\romannumeral4) To establish Proposition~\ref{proposition parameter range}, we then provide the dynamics of the sequence \eqref{def V} in the following lemma.

\begin{lemma}
	Suppose Assumptions~\ref{assumption smooth}--\ref{assumption polynomial} and \eqref{range varepsilon theta}--\eqref{range underline varepsilon} hold. Let $\{\mathbf{x}^k\}$ be the sequence generated by \eqref{zk final}--\eqref{qk+1 final}. Then, for any $k \geq 0$,
	\begin{align}
		V^{k+1} - V^k\leq & - \|\mathbf{x}^k\|^2_{\lambda_{2}^{\mathbf{G}^k} ({\epsilon_1^k} - {\epsilon_2^k} \lambda_{2}^{\mathbf{G}^k})\mathbf{K} }- \|\mathbf{s}^k \|^2_{ \lambda_{2}^{\mathbf{G}^k}({\epsilon_3^k} - \epsilon_4^k \lambda_{2}^{\mathbf{G}^k}) \mathbf{K}}- {\zeta^k}(\frac{1}{4} - {\epsilon_{5}^k} {\zeta^k}) \|\bar{\mathbf{g}}^k\|^2 - \frac{{\zeta^k}}{4}\|\bar{\mathbf{g}}_a^k\|^2, \label{tilde Vk+1-tilde Vk}
	\end{align}
	where the parameters $\epsilon_1^k$, $\epsilon_2^k$, $\epsilon_3^k$, $\epsilon_4^k$ and $\epsilon_{5}^k$ is presented in \eqref{epsilon1-epsilon5}.
\end{lemma}
\begin{proof}
	We first bound each term of the definition of $V^{k+1}$ (given by \eqref{def V}).
	
	For the first term of the definition of $V^{k+1}$, since $\lambda_1^\mathbf{K} =1$, $\mathbf{K}^2=\mathbf{K}$, \eqref{zk final}, \eqref{xk+1 final}, \eqref{KDk=DkK}, \eqref{Gk range} and \eqref{g0k-gk}, we have
\begin{align}
	&\frac{1}{2}\|\mathbf{x}^{k+1}\|_{\mathbf{K}}^2 
	\underset{\eqref{xk+1 final}}{\overset{\eqref{zk final}}{=}}  \frac{1}{2}\|\mathbf{x}^k - \mathbf{G}^k(\rho \mathbf{D}^k \mathbf{x}^k + \theta\mathbf{q}^k +\mathbf{g}_a^k + \mathbf{g}^k-\mathbf{g}_a^k)\|_{\mathbf{K}}^2 \notag\\
	= & \frac{1}{2}\|\mathbf{x}^k\|_{\mathbf{K}}^2-\|\mathbf{x}^k\|_{\rho \mathbf{D}^k\mathbf{G}^k-\frac{1}{2}(\rho {\mathbf{D}^k}{\mathbf{G}^k})^2} - \langle (\mathbf{I}-\rho \mathbf{D}^k\mathbf{G}^k)\mathbf{x}^k,\mathbf{G}^k\mathbf{K}(\theta \mathbf{s}^k+\mathbf{g}^k-\mathbf{g}_a^k) \rangle \notag\\
	& +\frac{1}{2}\|\mathbf{G}^k(\theta\mathbf{s}^k+\mathbf{g}^k-\mathbf{g}_a^k)\|^2_{\mathbf{K}} \notag\\
	\overset{\eqref{KDk=DkK}}{\leq} & \frac{1}{2}\|\mathbf{x}^k\|_{\mathbf{K}}^2-\|\mathbf{x}^k\|_{\rho \mathbf{D}^k\mathbf{G}^k-\frac{1}{2}(\rho {\mathbf{D}^k}{\mathbf{G}^k})^2} +\|\mathbf{x}^k\|_{\theta\mathbf{G}^k\mathbf{K}}^2\notag\\
	&+\frac{1}{4}\|\mathbf{s}^k\|_{\theta\mathbf{G}^k\mathbf{K}}^2+\frac{1}{2}\|\mathbf{x}^k\|^2_{\mathbf{G}^k\mathbf{K}}+\frac{1}{2}\|\mathbf{g}^k-\mathbf{g}^k_a\|^2_{\mathbf{G}^k\mathbf{K}}\notag\\
	& +\frac{1}{2}\|\mathbf{x}^k\|^2_{(\rho{\mathbf{D}^k}{\mathbf{G}^k})^2}+\frac{1}{2}\|\mathbf{s}^k\|^2_{(\theta{\mathbf{G}^k})^2\mathbf{K}}+\frac{1}{2}\|\mathbf{x}^k\|^2_{(\rho{\mathbf{D}^k}{\mathbf{G}^k})^2}\notag\\
	&+\frac{1}{2}\|\mathbf{g}^k-\mathbf{g}^k_a\|^2_{({\mathbf{G}^k})^2\mathbf{K}} + \|\mathbf{s}^k\|^2_{(\theta{\mathbf{G}^k})^2\mathbf{K}}+\|\mathbf{g}^k-\mathbf{g}^k_a\|^2_{({\mathbf{G}^k})^2\mathbf{K}} \notag\\
	\underset{\eqref{Gk range}}{\overset{\eqref{g0k-gk}}{\leq}} & \frac{1}{2}\|\mathbf{x}^k\|_{\mathbf{K}}^2-\|\mathbf{x}^k\|_{\rho \mathbf{D}^k\mathbf{G}^k-\frac{1+2\theta}{2}\mathbf{G}^k\mathbf{K}-\frac{3}{2}(\rho\lambda_1^{\mathbf{D}^k}\lambda_2^{\mathbf{G}^k})^2\mathbf{K}}\notag\\
	& +\frac{\bar{M}^2}{2}\|\mathbf{x}^k\|^2_{\big(\lambda_2^{\mathbf{G}^k}+3(\lambda_2^{\mathbf{G}^k})^2\big)\mathbf{K}}+\|\mathbf{s}^k\|^2_{(\frac{\theta}{4}{\mathbf{G}^k}+\frac{3}{2}(\theta\lambda_2^{\mathbf{G}^k})^2)\mathbf{K}}.
	\label{norm xk+1}
\end{align}

For the second term of the definition of $V^{k+1}$, since \eqref{zk final}--\eqref{qk+1 final}, \eqref{QktildeDk=tildeDkQk}, \eqref{tildeD range}, \eqref{GkDkQk} and \eqref{range underline varepsilon}--\eqref{theta rho tildeDk Gk}, we obtain
\begin{align}
	&\frac{1}{2\bar{\varepsilon}} \|\mathbf{s}^{k+1}\|^2_{\bK} \overset{\eqref{GkDkQk}}{\leq} \frac{1}{2} \|\mathbf{s}^{k+1}\|^2_{ \mathbf{G}^k\mathbf{D}^k\mathbf{Q}^k} \notag\\
	=& \frac{1}{2}\|\mathbf{q}^k+\frac{1}{\theta}\mathbf{g}_a^k+\frac{1}{\theta}(\mathbf{g}_a^{k+1}-\mathbf{g}_a^{k})+\rho \tilde{\mathbf{D}}^k\Big(( \mathbf{I}-\rho \mathbf{D}^k\bG^k)\mathbf{x}^k-\bG^k\big(\theta \mathbf{s}^k+(\mathbf{g}^k-\mathbf{g}_a^k)\big)\Big)\|^2_{ \mathbf{G}^k\mathbf{D}^k\mathbf{Q}^k}\notag\\
	\overset{\eqref{QktildeDk=tildeDkQk}}{=}& \frac{1}{2}\|(\mathbf{I}-\theta\rho\tilde{\mathbf{D}}^k \bG^k)\mathbf{s}^{k} + \frac{1}{\theta}(\mathbf{g}_a^{k+1}-\mathbf{g}_a^k)+\rho\tilde{\mathbf{D}}^k(\mathbf{I}-\rho \mathbf{D}^k\bG^k)\mathbf{x}^k - \rho\tilde{\mathbf{D}}^k \bG^k(\mathbf{g}^k-\mathbf{g}_a^k)\|^2_{ \mathbf{G}^k\mathbf{D}^k\mathbf{Q}^k} \notag\\
	\underset{\eqref{tildeD range}}{\overset{\eqref{GkDkQk}}{=}} & \frac{1}{2}\|(\mathbf{I}-\theta\rho\tilde{\mathbf{D}}^k \bG^k)\mathbf{s}^{k}\|^2_{ \mathbf{G}^k\mathbf{D}^k\mathbf{Q}^k} \notag\\
    &+\langle (\mathbf{I}\!-\!\theta\rho \tilde{\mathbf{D}}^k\bG^k)\mathbf{s}^k,\rho \mathbf{D}^k \mathbf{G}^k( \mathbf{I}\!-\!\rho \mathbf{D}^k\bG^k)\mathbf{x}^k \!+\!\frac{1}{\theta} \mathbf{G}^k\mathbf{D}^k\mathbf{Q}^k(\mathbf{g}_a^{k+1}\!-\!\mathbf{g}_a^k)\!-\! \rho \mathbf{D}^k(\mathbf{G}^k)^2(\mathbf{g}^k\!-\!\mathbf{g}_a^k)\rangle\notag\\
    &+\frac{1}{2}\|\bar{\varepsilon}\rho{\mathbf{D}}^k\bG^k(\mathbf{I}-\rho \mathbf{D}^k\bG^k)\mathbf{x}^k + \frac{1}{\theta}(\mathbf{g}_a^{k+1}-\mathbf{g}_a^k) - \rho\tilde{\mathbf{D}}^k \bG^k(\mathbf{g}^k-\mathbf{g}_a^k)\|^2_{\mathbf{G}^k\mathbf{D}^k\mathbf{Q}^k}\notag\\
	\underset{\eqref{rho tildeDk Gk preceq}}{\overset{\eqref{rhoDkGk}}{\leq}} & \frac{1}{2}\|\mathbf{s}^{k}\|^2_{ \mathbf{G}^k\mathbf{D}^k\mathbf{Q}^k}+ \langle (\mathbf{I}-\theta\rho \tilde{\mathbf{D}}^k\bG^k)\mathbf{s}^k, \rho \mathbf{D}^k\mathbf{G}^k \mathbf{x}^k \rangle \notag\\
	&+ \frac{1}{2}\|(\mathbf{I}-\theta\rho \tilde{\mathbf{D}}^k\bG^k)\mathbf{s}^{k}\|^2_{ (\rho \mathbf{D}^k\bG^k)^2 } + \frac{1}{2}\|\mathbf{x}^k\|^2_{ (\rho \mathbf{D}^k\mathbf{G}^k)^2 } \notag\\
	&+ \frac{\theta}{4}\|(\mathbf{I}-\theta\rho \tilde{\mathbf{D}}^k\bG^k)\mathbf{s}^k\|^2_{\mathbf{G}^k\mathbf{K}} + \frac{1}{\theta^3}\|\mathbf{g}_a^{k+1}-\mathbf{g}_a^k\|^2_{\mathbf{G}^k(\mathbf{D}^k\mathbf{Q}^k)^2}\notag\\
	&+ \frac{1}{2}\|(\mathbf{I}-\theta\rho \tilde{\mathbf{D}}^k\bG^k)\mathbf{s}^k\|^2_{(\mathbf{G}^k)^2\mathbf{K}} + \frac{1}{2}\|\mathbf{g}^k-\mathbf{g}_a^k\|^2_{(\rho\mathbf{D}^k\bG^k )^2} \notag\\
	& +\frac{3\bar{\varepsilon}^2}{2}\|\mathbf{x}^k\|^2_{(\rho\mathbf{D}^k\mathbf{G}^k)^2\mathbf{G}^k\mathbf{D}^k\mathbf{Q}^k} + \frac{3}{2\theta^2}\|\mathbf{g}_a^{k+1}-\mathbf{g}_a^k\|^2_{\mathbf{G}^k\mathbf{D}^k\mathbf{Q}^k}+ \frac{3\bar{\varepsilon}^2}{2}\|\mathbf{g}^k-\mathbf{g}_a^k\|^2_{(\mathbf{G}^k)^3\mathbf{D}^k\mathbf{Q}^k}\notag\\
	\underset{\eqref{range underline varepsilon}}{\overset{\eqref{theta rho tildeDk Gk}}{\leq}}& \frac{1}{2\bar{\varepsilon}}\|\mathbf{s}^{k}\|^2_{\bK}+ \langle (\mathbf{I}-\theta\rho \tilde{\mathbf{D}}^k\bG^k)\mathbf{s}^k, \rho \mathbf{D}^k\mathbf{G}^k \mathbf{x}^k \rangle \notag\\
	& + \|\mathbf{s}^k\|^2_{\frac{\theta\lambda_2^{\bG^k}}{8\kappa_\bG^k}\bK+ (\frac{1}{4}\mathbf{G}^k+\frac{1}{2}(\mathbf{G}^k)^2)\mathbf{K}+\frac{1}{2}(\rho\mathbf{D}^k\mathbf{G}^k)^2}\notag\\
	&+\|\mathbf{x}^k\|^2_{\frac{1}{2}(1+3\bar{\varepsilon}+\frac{3\bar{\varepsilon}^2}{4})(\rho \mathbf{D}^k\mathbf{G}^k)^2} \notag\\
	&+ \frac{1}{\theta^2}(\frac{1}{4}+\frac{1}{\bar{\varepsilon}})\big((\frac{1}{4}+\frac{1}{\bar{\varepsilon}})\frac{\kappa_{\bG}^k}{\theta\lambda_2^{\bG^k}}+\frac{3}{2}\big) \|\mathbf{g}_a^{k+1}-\mathbf{g}_a^k\|^2\notag\\
	&+\frac{1}{2}\|\mathbf{g}^k-\mathbf{g}_a^k\|^2_{(\rho\mathbf{D}^k\mathbf{G}^k)^2+3\bar{\varepsilon}(\frac{\bar{\varepsilon}}{4}+1)(\bG^k)^2\mathbf{K}}.
	\label{norm sk+1 original}
\end{align}

For the third term of the definition of $V^{k+1}$, since \eqref{zk final}--\eqref{qk+1 final}, together with \eqref{rhoDkGk}--\eqref{theta rho tildeDk Gk} and \eqref{tildeD range}, we have
\begin{align}
	& \langle \mathbf{x}^{k+1}, \mathbf{K}\mathbf{s}^{k+1} \rangle \notag\\
	=& \langle (\mathbf{I}-\rho\mathbf{D}^k\mathbf{G}^k)\mathbf{x}^{k} - \theta\mathbf{G}^k\mathbf{s}^k- \mathbf{G}^k(\mathbf{g}^k-\mathbf{g}_a^k),  \notag\\
	&\mathbf{K}\big((\mathbf{I}-\theta\rho\tilde{\mathbf{D}}^k\mathbf{G}^k)\mathbf{s}^k+ \rho \tilde{\mathbf{D}}^k(\mathbf{I}-\rho \mathbf{D}^k\mathbf{G}^k)\mathbf{x}^k+\frac{1}{\theta}(\mathbf{g}_a^{k+1}-\mathbf{g}_a^k)-\rho\tilde{\mathbf{D}}^k\mathbf{G}^k(\mathbf{g}^k-\mathbf{g}_a^k) \big) \rangle\notag\\
	=& \langle \mathbf{x}^k, \mathbf{K}\mathbf{s}^{k} \rangle - \langle \rho\mathbf{D}^k\mathbf{G}^k\mathbf{x}^k,(\mathbf{I}-\theta\rho\tilde{\mathbf{D}}^k\mathbf{G}^k)\mathbf{s}^k \rangle\notag\\
	&-\langle \mathbf{x}^k,\theta\rho\tilde{\mathbf{D}}^k\mathbf{G}^k\mathbf{s}^k\rangle + \|\mathbf{x}^k\|^2_{\rho \tilde{\mathbf{D}}^k(\mathbf{I}-\rho\mathbf{D}^k\mathbf{G}^k)^2}\notag\\
	&+\langle \mathbf{K}(\mathbf{I}-\rho\mathbf{D}^k\mathbf{G}^k)\mathbf{x}^k,\frac{1}{\theta}(\mathbf{g}_a^{k+1}-\mathbf{g}_a^k)-\rho\tilde{\mathbf{D}}^k\mathbf{G}^k(\mathbf{g}^k-\mathbf{g}_a^k)\rangle\notag\\
	& -\|\mathbf{s}^k\|^2_{\theta \mathbf{G}^k\mathbf{K}-\theta^2\rho\tilde{\mathbf{D}}^k(\mathbf{G}^k)^2}+\|\mathbf{g}^k-\mathbf{g}_a^k\|^2_{\rho \tilde{\mathbf{D}}^k(\mathbf{G}^k)^2} \notag\\
    &-\langle \theta\mathbf{G}^k\mathbf{K}\mathbf{s}^k,\frac{1}{\theta}(\mathbf{g}_a^{k+1}-\mathbf{g}_a^k) +\rho\tilde{\mathbf{D}}^k(\mathbf{I}-\rho\mathbf{D}^k\mathbf{G}^k)\mathbf{x}^k- \rho\tilde{\mathbf{D}}^k\mathbf{G}^k(\mathbf{g}^k-\mathbf{g}_a^k)\rangle \notag\\
	& -\langle \mathbf{G}^k\mathbf{K}(\mathbf{g}^k-\mathbf{g}_a^k), (\mathbf{I}-\theta\rho\tilde{\mathbf{D}}^k\mathbf{G}^k)\mathbf{s}^k +\frac{1}{\theta}(\mathbf{g}_a^{k+1}-\mathbf{g}_a^k) + \rho \tilde{\mathbf{D}}^k(\mathbf{I}-\rho \mathbf{D}^k\mathbf{G}^k)\mathbf{x}^k \rangle \notag\\
	\underset{\eqref{theta rho tildeDk Gk}}{\overset{\eqref{rhoDkGk}}{\leq}} & \langle \mathbf{x}^k, \mathbf{K}\mathbf{s}^{k} \rangle - \langle \rho\mathbf{D}^k\mathbf{G}^k\mathbf{x}^k,(\mathbf{I}-\theta\rho\tilde{\mathbf{D}}^k\mathbf{G}^k)\mathbf{s}^k \rangle\notag\\
	&+\frac{1}{2}\|\mathbf{x}^k\|^2_{(\rho \tilde{\mathbf{D}}^k)^2}+\frac{1}{2}\|\mathbf{s}^k\|^2_{(\theta\bG^k)^2\bK}+ \|\mathbf{x}^k\|^2_{\rho \tilde{\mathbf{D}}^k}\notag\\
	& +\frac{1}{2}\|\mathbf{x}^k\|^2_{\mathbf{G}^k\mathbf{K}}+\frac{1}{2\theta^2}\|\mathbf{g}_a^{k+1}-\mathbf{g}_a^k\|^2_{(\mathbf{G}^k)^{-1}} +\frac{1}{2}\|\mathbf{x}^k\|^2_{\rho\tilde{\mathbf{D}}^k\bG^k}\notag\\
	& +\frac{1}{2}\|\mathbf{g}^k-\mathbf{g}^k_a\|^2_{\rho\tilde{\mathbf{D}}^k\bG^k}-\|\mathbf{s}^k\|^2_{\theta \mathbf{G}^k\mathbf{K}-\theta^2\rho\tilde{\mathbf{D}}^k(\mathbf{G}^k)^2} +\|\mathbf{g}^k-\mathbf{g}_a^k\|^2_{\rho \tilde{\mathbf{D}}^k(\mathbf{G}^k)^2} \notag\\
	&+\frac{1}{2}\|\mathbf{s}^k\|^2_{(\theta \mathbf{G}^k)^2\mathbf{K}} +\frac{1}{2\theta^2}\|\mathbf{g}_a^{k+1}-\mathbf{g}_a^k\|^2 + \frac{1}{2}\|\mathbf{s}^k\|^2_{(\theta\mathbf{G}^k)^2\bK}\notag\\
	& +\frac{1}{2}\|\mathbf{x}^k\|^2_{(\rho \tilde{\mathbf{D}}^k)^2} +\frac{1}{2}\|\mathbf{s}^k\|^2_{(\theta \mathbf{G}^k)^2\mathbf{K}} + \frac{1}{2}\|\mathbf{g}^k-\mathbf{g}_a^k\|^2_{(\rho \tilde{\mathbf{D}}^k\bG^k)^2}\notag\\
	&+\frac{1}{\theta}\|\mathbf{g}^k-\mathbf{g}_a^k\|^2_{\mathbf{G}^k\mathbf{K}}+\frac{\theta}{4}\|\mathbf{s}^k\|^2_{\mathbf{G}^k\mathbf{K}}+\frac{1}{2}\|\mathbf{g}^k-\mathbf{g}_a^k\|^2_{\rho \tilde{\mathbf{D}}^k\mathbf{G}^k} \notag\\
	&  + \frac{1}{2}\|\mathbf{x}^k\|^2_{\rho \tilde{\mathbf{D}}^k\mathbf{G}^k} +\frac{1}{2}\|\mathbf{g}^k-\mathbf{g}_a^k\|^2_{(\mathbf{G}^k)^2\mathbf{K}} +\frac{1}{2\theta^2}\|\mathbf{g}_a^{k+1}-\mathbf{g}_a^k\|^2\notag\\
	\underset{\eqref{tildeD range}}{\overset{\eqref{rho tildeDk Gk preceq}}{\leq}}& \langle \mathbf{x}^k, \mathbf{K}\mathbf{s}^{k} \rangle - \langle \rho\mathbf{D}^k\mathbf{G}^k\mathbf{x}^k,(\mathbf{I}-\theta\rho\tilde{\mathbf{D}}^k\mathbf{G}^k)\mathbf{s}^k \rangle\notag\\
	& + \|\bx^k\|^2_{\bar{\varepsilon}\rho\bD^k\bG^k+\frac{1}{2}\bG^k\bK+(\bar{\varepsilon}\rho\bD^k\bG^k)^2}\notag\\
	&-\|\bs^k\|^2_{\mathbf{G}^k(\frac{3\theta}{4}\mathbf{I}-\theta^2(2+\bar{\varepsilon})\bG^k)\bK}+\frac{1}{2\theta^2}(\frac{\kappa_\bG^k}{\lambda_2^{\bG^k}}+2)\|\bg_a^{k+1}-\bg_a^k\|^2\notag\\
	&+\|\bg^k-\bg_a^k\|^2_{\frac{1}{\theta}\bG^k\bK+\bar{\varepsilon}\rho\bD^k(\bG^k)^2+\frac{1}{2}(1+2\bar{\varepsilon}+\bar{\varepsilon}^2)(\bG^k)^2\bK}.
	\label{xk+1, sk+1 original}
\end{align}

For the last term of the definition of $V^{k+1}$, since $f^* = \tilde{f}^*$, $\mathbf{J}=\mathbf{J}^2$, the descent lemma, the Cauchy-Schwarz inequality, \eqref{zetak leq theta2} and \eqref{bar g0k- bar gk}, we obtain
\vspace{-0.1cm}
\begin{align}
	&f(\bar{x}^{k+1}) - f^*=  \tilde{f}(\bar{\mathbf{x}}^{k}) - \tilde{f}^* + \tilde{f}(\bar{\mathbf{x}}^{k+1}) - \tilde{f}(\bar{\mathbf{x}}^{k}) \notag\\
	\leq & \tilde{f}(\bar{\mathbf{x}}^{k}) - \tilde{f}^* - \langle {\zeta^k}\bar{\mathbf{g}}^k, \bar{\mathbf{g}}_a^k \rangle + \frac{1}{2}(\zeta^k)^2\bar{M}\|\bar{\mathbf{g}}^k\|^2\notag\\
	= & f(\bar{x}^k) -f^* - \frac{1}{2}{\zeta^k}\langle \bar{\mathbf{g}}^k, \bar{\mathbf{g}}^k +\bar{\mathbf{g}}_a^k -\bar{\mathbf{g}}^k \rangle \notag\\
	&- \frac{1}{2}{\zeta^k}\langle \bar{\mathbf{g}}^k -\bar{\mathbf{g}}_a^k+\bar{\mathbf{g}}_a^k, \bar{\mathbf{g}}^k_a  \rangle + \frac{1}{2}(\zeta^k)^2\bar{M}\|\bar{\mathbf{g}}^k\|^2\notag\\
	\leq & f(\bar{x}^k) -f^* - \frac{{\zeta^k}}{4}\|\bar{\mathbf{g}}^k\|^2 + \frac{{\zeta^k}}{2} \|\bar{\mathbf{g}}^k -\bar{\mathbf{g}}_a^k\|^2\notag\\
	&  - \frac{{\zeta^k}}{4}\|\bar{\mathbf{g}}_a^k\|^2+ \frac{1}{2}(\zeta^k)^2\bar{M}\|\bar{\mathbf{g}}^k\|^2 \notag\\
	\underset{\eqref{zetak leq theta2}}{\overset{\eqref{bar g0k- bar gk}}{\leq}} & f(\bar{x}^k) -f^* - \frac{{\zeta^k}}{4}(1-2{\zeta^k}\bar{M})\|\bar{\mathbf{g}}^k\|^2 \notag\\
	&- \frac{{\zeta^k}}{4}\|\bar{\mathbf{g}}_a^k\|^2 + \frac{\theta^2}{4\kappa_{\mathbf{G}}^k}\lambda_{2}^{\mathbf{G}^k}\|\mathbf{x}^k\|^2_{\mathbf{K}}.
	\label{fk+1-f*}
\end{align}

By incorporating \eqref{norm xk+1}--\eqref{fk+1-f*} into the expression of $V^{k+1}-V^k$ resulting from the definition in \eqref{def V} and using \eqref{g0k-gk} and \eqref{g0k+1-g0k}, we obtain \eqref{tilde Vk+1-tilde Vk}.
\end{proof}

(\romannumeral5) Ultimately, we illustrate how the sequence in \eqref{tilde Vk+1-tilde Vk} descends along iterations based on the well-defined parameters in \eqref{epsilon1-epsilon5}--\eqref{range underline varepsilon}.

From $0 < \lambda_{2}^{\mathbf{G}^k} < \min\{\frac{{\epsilon_1^k}}{{\epsilon_2^k}}, \frac{{\epsilon_3^k}}{\epsilon_4^k}\}$ in \eqref{range hat lambda Gk},
\begin{align}
	\label{epsilon3-epsilon4} & \lambda_{2}^{\mathbf{G}^k} ({\epsilon_1^k} - {\epsilon_2^k} \lambda_{2}^{\mathbf{G}^k})> 0 ,\\
    &\lambda_{2}^{\mathbf{G}^k}({\epsilon_3^k} - \epsilon_4^k \lambda_{2}^{\mathbf{G}^k})> 0.
\end{align}
From $\lambda_{2}^{\mathbf{G}^k}<{\zeta^k}<\frac{1}{4\epsilon_{5}^k}$ in \eqref{range zetak}, we obtain
\begin{align}
	{\zeta^k}(\frac{1}{4} - {\epsilon_{5}^k} {\zeta^k}) >0 .\label{1/4-epsilon5}
\end{align}
By combining \eqref{epsilon3-epsilon4}--\eqref{1/4-epsilon5} with \eqref{tilde Vk+1-tilde Vk}, we show that $V^{k+1}\leq V^k$, and thus the sequence of $V^k$ is nonincreasing.

\subsection{Proof of convergence results of UPP-MC}
\subsubsection{Proof of Theorem~\ref{theorem nonconvex}}\label{proof theorem nonconvex}
First, for each $k\geq 0$, we define
\begin{align}
	\hat{V}^k =& \|\mathbf{x}^k\|_\mathbf{K}^2 + \|\mathbf{s}^k\|^2_\mathbf{K} + f(\bar{x}^k) - f^*. \label{hatVk}\\
	\label{Wk}W^k =& \|\mathbf{x}^{k}\|^2_{\mathbf{K}} +  \|\mathbf{s}^k\|^2_\mathbf{K} + \|\bar{\mathbf{g}}^k\|^2 + \|\bar{\mathbf{g}}^k_a\|^2.
\end{align}
We also define $\xi_3=\frac{1}{2}\big(\frac{1}{\bar{\varepsilon}}-1+\sqrt{(\frac{1}{\bar{\varepsilon}}-1)^2+4}\big)$, which guarantees ${\frac{1}{2}}+\frac{\xi_3}{2}={\frac{1}{2\bar{\varepsilon}}}+\frac{1}{2{\xi_3}}$. It then follows from \eqref{GkDkQk} and \eqref{def V} that
\begin{align}
	V^k \leq & {\frac{1}{2}} \|\mathbf{x}^k\|^2_\mathbf{K} + {\frac{1}{2\bar{\varepsilon}}} \|\mathbf{s}^k\|^2_\mathbf{K} + \frac{{\xi_3}}{2}\|\mathbf{x}^k\|^2_\mathbf{K}+ \frac{1}{2{\xi_3}}\|\mathbf{s}^k\|^2_\mathbf{K} + f(\bar{x}^k) - f^* \notag\\
	= & {\frac{1+{\xi_3}}{2}}(\|\mathbf{x}^k\|^2_\mathbf{K} + \|\mathbf{s}^k\|^2_\mathbf{K}) +f(\bar{x}^k) - f^*\leq  {\delta_1} \hat{V}^k, \label{delta2}
\end{align}
where ${\delta_1} = \max\{\frac{1+\xi_3}{2},1\}$. Here, we define $\xi_4:=\frac{1}{2}(1-\frac{1}{\bar{\varepsilon}}+\sqrt{(1-\frac{1}{\bar{\varepsilon}})^2+4})$, and have ${\frac{1}{2}}-\frac{\xi_4}{2}={\frac{1}{2\bar{\varepsilon}}}-\frac{1}{2{\xi_4}}$. Because $0<\bar{\varepsilon}<1$ leads to $0<\xi_4<\frac{1}{2}\big(1-\frac{1}{\bar{\varepsilon}}+\sqrt{(1-\frac{1}{\bar{\varepsilon}})^2+\frac{4}{\bar{\varepsilon}}}\big)=1$, together with \eqref{GkDkQk} and \eqref{def V}, we obtain
\begin{align}
	V^k \geq & {\frac{1}{2}} \|\mathbf{x}^k\|^2_\mathbf{K} + {\frac{1}{2\bar{\varepsilon}}} \|\mathbf{s}^k\|^2_\mathbf{K} - \frac{{\xi_4}}{2}\|\mathbf{x}^k\|^2_\mathbf{K} - \frac{1}{2{\xi_4}}\|\mathbf{s}^k\|^2_\mathbf{K} + f(\bar{x}^k) - f^* \notag\\
	= & {\frac{1-{\xi_4}}{2}}(\|\mathbf{x}^k\|^2_\mathbf{K} + \|\mathbf{s}^k\|^2_\mathbf{K})+f(\bar{x}^k) - f^*\geq  {\delta_2} \hat{V}^k, \label{deltassumption: optimality}
\end{align}
where $\delta_2:=\frac{1-{\xi_4}}{2}>0$.

Adding \eqref{tilde Vk+1-tilde Vk} from $k=0$ to $T-1$ and using \eqref{Wk} yields $V^{T} \leq V^0 - \delta_3 \sum_{k=0}^{T-1} W^k$. 
Combining $\hat{W}^k\leq W^k$, we have $V^{T} + \delta_3 \sum_{k=0}^{T-1} \hat{W}^k\leq V^{T} + \delta_3 \sum_{k=0}^{T-1} W^k \leq V^0$, where $\delta_3  = \min_{k\in(0,T)} \{\lambda_{2}^{\mathbf{G}^k}({\epsilon_1^k} - {\epsilon_2^k} \lambda_{2}^{\mathbf{G}^k}),\lambda_{2}^{\mathbf{G}^k}({\epsilon_3^k} - \epsilon_4^k \lambda_{2}^{\mathbf{G}^k}),{\zeta^k}(\frac{1}{4} - {\epsilon_{5}^k} {\zeta^k}),{\zeta^k}/4 \}.$
This, along with $V^{T}\geq 0$ due to \eqref{deltassumption: optimality}, gives $\sum_{k=0}^{T-1} \hat{W}^{k}\leq \frac{V^0}{\delta_3}$. It then follows from \eqref{delta2} that \eqref{theorem 1 sublinear} holds with $C_1 = \frac{{\delta_1} \hat{V}^0}{\delta_3 }$.


\subsubsection{Proof of Theorem~\ref{theorem pl}}\label{proof theorem pl}
From \eqref{hatVk} and \eqref{deltassumption: optimality},
\begin{equation}
	\|\mathbf{x}^k\|^2_{\mathbf{K}} + f(\bar{x}^k) - f^* \leq \hat{V}^k \leq {V^k}/{{\delta_2}}.\label{norm consensus+function}
\end{equation}
Due to Assumption~\ref{PL condition},
\begin{equation}
	\|\bar{\mathbf{g}}_a^k\|^2 = N\|\nabla f(\bar{x}^k)\|^2 \geq 2\nu N(f(\bar{x}^k) - f^*).\label{barg0k geq 2nu}
\end{equation}
Based on \eqref{epsilon3-epsilon4}, we have
\begin{equation}
	\delta_4 = \min_{k\geq 0} \{\lambda_{2}^{\mathbf{G}^k}({\epsilon_1^k} -  {\epsilon_2^k}\lambda_{2}^{\mathbf{G}^k}),\lambda_{2}^{\mathbf{G}^k}({\epsilon_3^k} - \epsilon_4^k\lambda_{2}^{\mathbf{G}^k}), \frac{\nu N{\zeta^k}}{2}\} > 0. \label{delta4}
\end{equation}
Then, from \eqref{tilde Vk+1-tilde Vk}, \eqref{1/4-epsilon5}, \eqref{barg0k geq 2nu} and \eqref{delta2}, we have
\begin{align}
	V^{k+1} \leq&  V^k- \|\mathbf{x}^k\|^2_{\lambda_{2}^{\mathbf{G}^k} ({\epsilon_1^k} - {\epsilon_2^k} \lambda_{2}^{\mathbf{G}^k})\mathbf{K} } - \|\mathbf{s}^k\|^2_{  \lambda_{2}^{\mathbf{G}^k}({\epsilon_3^k} - \epsilon_4^k \lambda_{2}^{\mathbf{G}^k})\mathbf{K} }\notag\\
	& - \frac{\nu N {\zeta^k}}{2} (f(\bar{x}^k) - f^*) \leq V^{k} - \delta_4 \hat{V}^k\leq V^k - \frac{\delta_4}{{\delta_1}}V^k.
	\label{tildeVk+1 leq -delta4}
\end{align}
Because $\delta_3 \leq \lambda_{2}^{\mathbf{G}^k}({\epsilon_3^k} - \epsilon_4^k \lambda_{2}^{\mathbf{G}^k})$ and $\delta_1\geq 1$, we have $\frac{\delta_4}{{\delta_1}}\leq \frac{(\epsilon_3^k)^2}{4\epsilon_4^k}<\frac{1}{96\kappa_\bG^k}<1$. Thus, combining \eqref{delta4} and \eqref{tildeVk+1 leq -delta4} yields
\vspace{-0.2cm}
\begin{equation}
	V^{k} \leq (1-\delta) V^{k-1} \leq (1-\delta)^{k} V^0 \leq (1-\delta)^{k} {\delta_1} \hat{V}^0, \notag
\end{equation}
where $\delta = \delta_4/\delta_1\in(0,1)$. This, along with \eqref{norm consensus+function}, implies \eqref{theorem 2 linear} with $C_2=\delta_1\hat{V}^0/\delta_2$.

\subsection{Proof of convergence results of UPP-SC}

\subsubsection{Proof of Lemma~\ref{lemma vk+1 - vk}}\label{proof lemma vk+1 - vk}
From \eqref{vk+1 opt}, \eqref{wk+1} and the first-order optimality condition \eqref{first-order opt modified}, we have
\vspace{-0.2cm}
\begin{align}
	&\mathbf{L}^{\frac{1}{2}}(\mathbf{v}^{k+1}-\mathbf{v}^k)=-(\nabla \tilde{f}(\mathbf{x}^k)-\nabla \tilde{f}(\mathbf{x}^{k-1}))-\mathbf{B}^k\mathbf{w}^{k+1}-(\mathbf{B}^{k}-\mathbf{B}^{k-1})(\mathbf{x}^k-\mathbf{x}^{k-1}), \quad \forall k \geq 0.\label{qk+1 - qk}
\end{align}
Due to \eqref{vk+1 opt}, \eqref{qk+1 - qk}, Assumption~\ref{assumption smooth} and $\mathbf{v}^{k+1}-\mathbf{v}^k \in \mathcal{S}^{{\perp}} \,\forall k\geq 0$,
\vspace{-0.2cm}
\begin{align}
	&\rho\|\mathbf{x}^{k+1}\|^2_{\mathbf{L}}=\frac{1}{\rho}\|\mathbf{v}^{k+1}-\mathbf{v}^k\|^2\notag \\ 
    \leq& \frac{1}{\rho\lambda_{N-1}^{\mathbf{L}}}\|\nabla \tilde{f}(\mathbf{x}^k)-\nabla \tilde{f}(\mathbf{x}^{k-1})+\mathbf{B}^k\mathbf{w}^{k+1}+(\mathbf{B}^{k}-\mathbf{B}^{k-1})(\mathbf{x}^k-\mathbf{x}^{k-1})\|^2  \notag \\
	\leq &\frac{3}{\rho\lambda_{N-1}^{\mathbf{L}}}\big(\|\nabla \tilde{f}(\mathbf{x}^k)-\nabla \tilde{f}(\mathbf{x}^{k-1})\|^2+\|\mathbf{w}^{k+1}\|^2_{(\mathbf{B}^k)^2}+ \|\mathbf{x}^k-\mathbf{x}^{k-1}\|^2_{(\mathbf{B}^{k}-\mathbf{B}^{k-1})^2}\big) \notag\\
	\leq & \frac{3}{\rho\lambda_{N-1}^{\mathbf{L}}}\big(\bar{M}^2\|\mathbf{x}^k-\mathbf{x}^{k-1}\|^2+\|\mathbf{w}^{k+1}\|^2_{(\mathbf{B}^k)^2}+ 4(\lambda_1^\mathbf{B})^2\|\mathbf{x}^k-\mathbf{x}^{k-1}\|^2\big) \notag\\
	\leq &3\tilde{\kappa} \big((c_{\mathbf{B}}+\bar{M}^2)\|\mathbf{x}^k-\mathbf{x}^{k-1}\|^2_{(\mathbf{B}^{k})^{-1}}+\|\mathbf{w}^{k+1}\|^2_{\mathbf{B}^k}\big), \label{rho xk+1 leq 3 tildekappa}
\end{align}
where $c_{\mathbf{B}}=4(\lambda_1^\mathbf{B})^2$ and $\tilde{\kappa} = {\lambda_1^{\mathbf{B}}}/({\rho \lambda_{N-1}^{\mathbf{L}}})$. 

\subsubsection{Proof of Lemma~\ref{lemma ALk+1 - ALk}}\label{proof lemma ALk+1 - ALk}
Due to Assumption~\ref{assumption smooth}, \eqref{xk+1 opt}, \eqref{AL} and $\mathbf{G}^k=(\mathbf{B}^k+\rho \mathbf{L})^{-1}$, we have
\begin{align}
	& \operatorname{AL}(\mathbf{x}^{k+1},\mathbf{v}^{k})-\operatorname{AL}(\mathbf{x}^{k},\mathbf{v}^{k})\notag\\
	{\leq}& \langle \nabla \tilde{f}(\mathbf{x}^k)+\mathbf{L}^{\frac{1}{2}}\mathbf{v}^k,\mathbf{x}^{k+1}-\mathbf{x}^k \rangle + \frac{\bar{M}}{2}\|\mathbf{x}^{k+1}-\mathbf{x}^k\|^2  + \frac{\rho}{2}\|\mathbf{x}^{k+1}\|^2_\mathbf{L} - \frac{\rho}{2}\|\mathbf{x}^{k}\|^2_\mathbf{L} \notag\\
	=& -\|\mathbf{x}^{k+1}-\mathbf{x}^k\|^2_{(\mathbf{G}^k)^{-1}-\frac{\bar{M}}{2}\mathbf{I}_{Nd}-\frac{1}{2}\rho \mathbf{L}}\notag\\
	=& -\|\mathbf{x}^{k+1}-\mathbf{x}^k\|^2_{\mathbf{B}^k+\frac{1}{2}\rho \mathbf{L}-\frac{\bar{M}}{2}\mathbf{I}_{Nd}}. \label{ALxk+1 vk}
\end{align}

From \eqref{vk+1 opt} and \eqref{ALxk+1 vk},
\vspace{-0.2cm}
\begin{align*}
	\operatorname{AL}^{k+1} - \operatorname{AL}^k =
    & \operatorname{AL}(\mathbf{x}^{k+1},\mathbf{v}^{k})-\operatorname{AL}(\mathbf{x}^{k},\mathbf{v}^{k}) + \langle \mathbf{L}^{\frac{1}{2}}(\mathbf{v}^{k+1}-\mathbf{v}^k),\mathbf{x}^{k+1}\rangle \notag\\
	{\leq} & -\|\mathbf{x}^{k+1}-\mathbf{x}^{k}\|^2_{\mathbf{B}^k+\frac{1}{2}\rho \mathbf{L} -\frac{\bar{M}}{2}\mathbf{I}_{Nd}} + \rho\|\mathbf{x}^{k+1}\|^2_\mathbf{L} .
\end{align*}
Hence, we obtain \eqref{ALk+1 - ALk lemma}.

\subsubsection{Proof of Lemma~\ref{lemma tildePk+1 nonincreasing}}\label{proof lemma tildePk+1 nonincreasing}
From \eqref{first-order opt modified}, \eqref{vk+1 opt} and Assumption~\ref{smooth}, we derive
\vspace{-0.2cm}
\begin{align}
	0=& -\langle \mathbf{B}^k (\mathbf{x}^{k+1}-\mathbf{x}^k)-\mathbf{B}^{k-1}(\mathbf{x}^{k}-\mathbf{x}^{k-1})+\nabla \tilde{f}(\mathbf{x}^k) \notag\\
    &-\nabla \tilde{f}(\mathbf{x}^{k-1}) + \mathbf{L}^{\frac{1}{2}}(\mathbf{v}^{k+1}-\mathbf{v}^k), \mathbf{x}^{k+1}-\mathbf{x}^k-(\mathbf{x}^{k}-\mathbf{x}^{k-1})\rangle\notag\\
	\leq & \frac{1}{2}\|\mathbf{x}^{k+1} - \mathbf{x}^{k}\|^2_{\mathbf{B}^k}+\frac{1}{2}\|\mathbf{x}^{k} - \mathbf{x}^{k-1}\|^2_{\mathbf{B}^k}-\frac{1}{2}\|\mathbf{w}^{k+1}\|^2_{\mathbf{B}^k}\notag\\
    & - \frac{1}{2}\|\mathbf{x}^{k} - \mathbf{x}^{k-1}\|^2_{\mathbf{B}^{k-1}}+\frac{1}{2}\|\mathbf{x}^{k+1} - \mathbf{x}^{k}\|^2_{\mathbf{B}^{k-1}}-\frac{1}{2}\|\mathbf{w}^{k+1}\|^2_{\mathbf{B}^{k-1}}\notag\\
	& + \frac{\bar{M}}{2}\|\mathbf{x}^{k} - \mathbf{x}^{k-1}\|^2+\frac{\bar{M}}{2}\|\mathbf{x}^{k+1} - \mathbf{x}^{k}\|^2+\bar{M}\|\mathbf{x}^{k} - \mathbf{x}^{k-1}\|^2\notag\\
	& -\frac{\rho}{2}\|\mathbf{x}^{k+1}\|^2_\mathbf{L} + \frac{\rho}{2}\|\mathbf{x}^k\|^2_\mathbf{L}- \frac{\rho}{2}\|\mathbf{x}^{k+1}- \mathbf{x}^{k}\|^2_\mathbf{L}+\frac{\rho}{2}\|\mathbf{x}^{k+1}\|^2_\mathbf{L} + \frac{\rho}{2}\|\mathbf{x}^{k}- \mathbf{x}^{k-1}\|^2_\mathbf{L}. \label{0=-<>}
\end{align}
By using $\mathbf{B}^{k-1}\succ \mathbf{O}$ and \eqref{vk+1 - vk}, we obtain
\begin{align}
	&\frac{1}{2}\|\mathbf{x}^{k+1} - \mathbf{x}^{k}\|^2_{\mathbf{B}^k-\mathbf{B}^{k-1}+\rho \mathbf{L}-\bar{M}\mathbf{I}}+\frac{\rho}{2}\|\mathbf{x}^{k+1}\|^2_{\mathbf{L}}\notag\\
	\leq& \frac{1}{2}\|\mathbf{x}^{k} - \mathbf{x}^{k-1}\|^2_{\mathbf{B}^k+\rho \mathbf{L}+3\tilde{\kappa}(c_{\mathbf{B}}+\bar{M}^2)(\mathbf{B}^k)^{-1}+3\bar{M}\mathbf{I}_{Nd}} + \frac{\rho}{2}\|\mathbf{x}^k\|^2_\mathbf{L}-\frac{1}{2}\|\mathbf{w}^{k+1}\|^2_{\mathbf{B}^k+\mathbf{B}^{k-1}-3\tilde{\kappa}\mathbf{B}^k}.\label{0 geq -Bk xk+1 - xk - Bk-1 xk - xk-1}
\end{align}
Next, from \eqref{tilde Pk+1}, \eqref{ALk+1 - ALk lemma}, \eqref{vk+1 - vk} and \eqref{0 geq -Bk xk+1 - xk - Bk-1 xk - xk-1},
\begin{align}
	&\tilde{P}^{k+1}-\tilde{P}^{k}\notag\\
	&{\leq} -\|\mathbf{x}^{k+1}-\mathbf{x}^{k}\|^2_{\mathbf{B}^{k}+\frac{1}{2}\rho \mathbf{L} -\frac{\bar{M}}{2}\mathbf{I}_{Nd}}\notag\\
    &+3\tilde{\kappa} \big((c_{\mathbf{B}}+\bar{M}^2)\|\mathbf{x}^{k}-\mathbf{x}^{k-1}\|^2_{(\mathbf{B}^k)^{-1}}+\|\mathbf{w}^{k+1}\|^2_{\mathbf{B}^k}\big)\notag\\
	&+3\tilde{\kappa} (c_{\mathbf{B}}+\bar{M}^2)\big(\|\mathbf{x}^{k+1}-\mathbf{x}^{k}\|^2_{(\mathbf{B}^{k+1})^{-1}}-\|\mathbf{x}^{k}-\mathbf{x}^{k-1}\|^2_{(\mathbf{B}^k)^{-1}}\big)\notag\\
	&+\frac{\tilde{c}}{2}\|\mathbf{x}^{k+1}-\mathbf{x}^{k}\|^2_{\mathbf{B}^{k+1}-\mathbf{B}^k+\mathbf{B}^{k-1}+4\bar{M}\mathbf{I}_{Nd}+3\tilde{\kappa} (c_{\mathbf{B}}+\bar{M}^2)(\mathbf{B}^{k+1})^{-1}}\notag\\
	&+\frac{\tilde{c}}{2}\big(\rho\|\mathbf{x}^{k+1}\|^2_\mathbf{L}+\|\mathbf{x}^{k+1}-\mathbf{x}^k\|^2_{\mathbf{B}^k-\mathbf{B}^{k-1}+\rho \mathbf{L}-\bar{M}\mathbf{I}_{Nd}}\notag\\
    &-(\rho \|\mathbf{x}^k\|^2_\mathbf{L}+\|\mathbf{x}^k-\mathbf{x}^{k-1}\|^2_{\mathbf{B}^k+3\bar{M}\mathbf{I}_{Nd}+3\tilde{\kappa} (c_{\mathbf{B}}+\bar{M}^2)(\mathbf{B}^k)^{-1}+\rho \mathbf{L}})\big)\notag\\
	& \overset{\eqref{0 geq -Bk xk+1 - xk - Bk-1 xk - xk-1}}{\leq}-\|\mathbf{x}^{k+1}-\mathbf{x}^{k}\|^2_{\mathbf{B}^k+\frac{1}{2}\rho \mathbf{L} -\frac{\tilde{c}}{2}(\mathbf{B}^{k+1}-\mathbf{B}^k+\mathbf{B}^{k-1})}\notag\\
    &-\|\mathbf{x}^{k+1}-\mathbf{x}^{k}\|^2_{-(\frac{1}{2}+2\tilde{c})\bar{M}\mathbf{I}_{Nd}-\frac{3(2+\tilde{c})\tilde{\kappa}}{2} \bar{M}^2(\mathbf{B}^{k+1})^{-1}}\notag\\
	&-\frac{\tilde{c}}{2}\|\mathbf{w}^{k+1}\|^2_{\mathbf{B}^k+\mathbf{B}^{k-1}-3\tilde{\kappa}\mathbf{B}^k}+3\tilde{\kappa}\|\mathbf{w}^{k+1}\|^2_{\mathbf{B}^k}.\notag
\end{align}
This, combining with \eqref{define tilde c} and \eqref{Bk+rhoL-c/2Bk-1}, yields \eqref{tildePk+1 - tildePk lemma}.

\subsubsection{Proof of Lemma~\ref{lemma tildePk+1 geq f*}}\label{proof lemma tildePk+1 geq f*}
According \eqref{AL} and \eqref{vk+1 opt}, 
\begin{align}
	&\operatorname{AL}^{k+1}-\tilde{f}(\mathbf{x}^{k+1})-\frac{\rho}{2}\|\mathbf{x}^{k+1}\|^2_\mathbf{L}= \langle \mathbf{v}^{k+1},\mathbf{L}^{\frac{1}{2}}\mathbf{x}^{k+1} \rangle\notag\\
	=& \frac{1}{\rho} \langle \mathbf{v}^{k+1},\mathbf{v}^{k+1}- \mathbf{v}^{k}\rangle  \notag\\
	=&\frac{1}{2\rho}(\|\mathbf{v}^{k+1}\|^2-\|\mathbf{v}^{k}\|^2+\|\mathbf{v}^{k+1}-\mathbf{v}^{k}\|^2 ). \label{ALk+1 - fxk+1}
\end{align}
 
Define $\widehat{\operatorname{AL}}^{k+1}:=\operatorname{AL}^{k+1}-f^*$, $\hat{f}(\mathbf{x}):=\tilde{f}(\mathbf{x})-f^*\geq 0$, $\hat{P}^{k+1}:= \tilde{P}^{k+1}-f^*$. Summing over $k=-1,\dots,T$, we obtain $\sum_{k=-1}^T \widehat{\operatorname{AL}}^{k+1} = \frac{1}{2\rho}(\|\mathbf{v}^{T+1}\|^2-\|\mathbf{v}^{-1}\|^2)
    + \sum_{k=-1}^T (\hat{f}(\mathbf{x}^{k+1})+\frac{\rho}{2}\|\mathbf{x}^{k+1}\|^2_\mathbf{L}+\frac{1}{2\rho}\|\mathbf{v}^{k+1}-\mathbf{v}^k\|^2).$
By the initialization $\mathbf{v}^{-1}=0$, the above sum is lower bounded by zero. Therefore, the sum of $\hat{P}^{k+1}$ is also lower bounded by zero, i.e., 
	$\sum_{k=0}^T \hat{P}^{k+1}\geq 0, \quad \forall T>0$. 
Note that Lemma~\ref{lemma tildePk+1 nonincreasing} shows that $\hat{P}^{k+1}$ is nonincreasing under the parameter selections \eqref{define tilde c}--\eqref{Bk+rhoL-c/2Bk-1}. Hence, we can conclude that
\begin{equation}
	\hat{P}^{k+1}\geq 0, \quad \tilde{P}^{k+1}\geq f^*, \quad \forall k \geq 0. \label{hat Pk+1 geq 0 tildePk+1 geq f*}
\end{equation}

Next, since $\mathbf{x}^{-1}=0$ and $\mathbf{v}^{-1}=0$, we have
\begin{align}
	\tilde{P}^0 =& \operatorname{AL}^0+3\tilde{\kappa}(c_{\mathbf{B}}+\bar{M}^2)\|\mathbf{x}^0-\mathbf{x}^{-1}\|^2_{(\mathbf{B}^{0})^{-1}}+\frac{\tilde{c}}{2}(\rho \|\mathbf{x}^0\|^2_\mathbf{L}\notag\\
    &+\|\mathbf{x}^0-\mathbf{x}^{-1}\|^2_{\mathbf{B}^{0}+3\bar{M}\mathbf{I}_{Nd}+3\tilde{\kappa}(c_{\mathbf{B}}+\bar{M}^2)(\mathbf{B}^{0})^{-1}+\rho \mathbf{L}})\notag\\
	=& \operatorname{AL}^0+\|\mathbf{x}^0\|^2_{\frac{3(2+\tilde{c})\tilde{\kappa}}{2}(c_{\mathbf{B}}+\bar{M}^2)(\mathbf{B}^{0})^{-1}+\frac{\tilde{c}}{2}(2\rho \mathbf{L} +\mathbf{B}^0+3\bar{M}\mathbf{I}_{Nd})}, \label{tildeP0} \\
	\operatorname{AL}^0=&\tilde{f}(\mathbf{x}^0)+\langle \mathbf{L}^{\frac{1}{2}}\mathbf{v}^0,\mathbf{x}^0\rangle+\frac{\rho}{2}\|\mathbf{x}^0\|^2_\mathbf{L} \notag\\
	\overset{\eqref{vk+1 opt}}{=}& \tilde{f}(\mathbf{x}^0)+\langle \mathbf{L}^{\frac{1}{2}}\mathbf{v}^{-1}+\rho \mathbf{L} \mathbf{x}^0,\mathbf{x}^0\rangle +\frac{\rho}{2}\|\mathbf{x}^0\|^2_\mathbf{L}\notag\\
	=& \tilde{f}(\mathbf{x}^0)+\frac{3\rho}{2}\|\mathbf{x}^0\|^2_\mathbf{L}, \label{AL0 initialization}\\
	\mathbf{x}^0 =& -\mathbf{G}^0\nabla \tilde{f}(0) . \label{x0 initialization}
\end{align}
By incorporating \eqref{AL0 initialization} into \eqref{tildeP0} and using \eqref{define tilde c}, we have
\begin{equation}
	\tilde{P}^0\leq \tilde{f}(\mathbf{x}^0)+\|\mathbf{x}^0\|^2_{\frac{\tilde{c}(2+\tilde{c})(c_{\mathbf{B}}+\bar{M}^2)}{6}(\mathbf{B}^0)^{-1}+\frac{\tilde{c}}{2}\mathbf{B}^0+\frac{3\tilde{c}}{2}\bar{M}\mathbf{I}_{Nd}+(\tilde{c}+\frac{3}{2})\rho \mathbf{L}}. \label{tildeP0 first}
\end{equation}
From \eqref{Bk+rhoL-c/2Bk-1} and the initialization $\mathbf{G}^{-1}=\mathbf{G}^0$, 
\begin{equation}
	\frac{1+2\tilde{c}}{4}(\mathbf{G}^0)^{-1} \succeq \frac{\tilde{c}}{2}\mathbf{B}^0+(\frac{1}{2}+2\tilde{c})\bar{M}\mathbf{I}+\frac{\tilde{c}(2+\tilde{c})(c_{\mathbf{B}}+\bar{M}^2)}{6}(\mathbf{B}^0)^{-1}, \label{1/4+tildec/2 succeq}
\end{equation}
which, with $\mathbf{B}^0\succ \mathbf{O}$, further implies that $(\frac{1}{4}+\frac{\tilde{c}}{2})(\mathbf{G}^0)^{-1}\succ (\frac{1}{2}+2\tilde{c})\bar{M}\mathbf{I}_{Nd}$. Thus, combing \eqref{x0 initialization}--\eqref{1/4+tildec/2 succeq} gives
\begin{align}
	\tilde{P}^0\leq& \tilde{f}(\mathbf{x}^0)+ (\frac{3}{2}\tilde{c}+\frac{7}{4})\|\nabla \tilde{f}(0)\|^2_{\mathbf{G}^0}\notag\\
	\overset{\eqref{Bk+rhoL-c/2Bk-1}}{\leq} &\tilde{f}(\mathbf{x}^0) +\frac{(\frac{3}{2}\tilde{c}+\frac{7}{4})(\frac{1}{4}+\frac{\tilde{c}}{2})}{(\frac{1}{2}+2\tilde{c})\bar{M}}\|\nabla \tilde{f}(0)\|^2 \notag\\
	< & \tilde{f}(\mathbf{x}^0) +\frac{6\tilde{c}+7}{8\bar{M}}\|\nabla \tilde{f}(0)\|^2. \label{tildeP0 leq fx0+nabla f0}
\end{align}
Hence, from \eqref{hat Pk+1 geq 0 tildePk+1 geq f*} and \eqref{tildeP0 leq fx0+nabla f0}, we obtain Lemma~\ref{lemma tildePk+1 geq f*}.
\subsubsection{Proof of Theorem~\ref{theorem UPP-SC sublinear time-varying Bk}}\label{proof theorem UPP-SC sublinear time-varying Bk}

\textbf{(i)} First we illustrate the feasibility of the parameters in \eqref{define tilde c} and \eqref{Bk+rhoL-c/2Bk-1}. 



To see \eqref{Bk+rhoL-c/2Bk-1}, it is sufficient to show that $\frac{1}{4}\lambda_N^{\mathbf{B}}-\frac{\tilde{c}}{2}(\lambda_1^{\mathbf{B}}-\lambda_N^{\mathbf{B}}+\lambda_1^{\mathbf{B}})-(\frac{1}{2}+2\tilde{c})\bar{M}
-\frac{\tilde{c}(2+\tilde{c})}{6}(c_\mathbf{B}+\bar{M}^2)/\lambda_N^{\mathbf{B}}> 0,$
which can be transformed into the following inequality:
\begin{equation}\label{d1lambdaNb-d2-d3/lambdaNB}
	d_1\lambda_N^{\mathbf{B}}-d_2-d_3/\lambda_N^{\mathbf{B}} >0,
\end{equation}
where $d_1= \frac{1}{4}- \xi_5 \tilde{c}-\xi_6 \tilde{c}^2$, $d_2 = (\frac{1}{2}+2\tilde{c})\bar{M}$, and $d_3={(2+\tilde{c})\tilde{c}}\bar{M}^2/6$.
Since $0<\tilde{c}<({-\xi_5+\sqrt{\xi_5^2+\xi_6}})/({2\xi_6})$ implies $d_1>0$, together with \eqref{d1lambdaNb-d2-d3/lambdaNB}, we obtain $\lambda_N^{\mathbf{B}}\geq ({d_2+\sqrt{d_2^2+4d_1d_3}})/({2d_1})$. Hence, \eqref{Bk+rhoL-c/2Bk-1} is satisfied. Besides, from $\rho \geq ({9\kappa_{\mathbf{B}}\lambda_N^{\mathbf{B}}})/({\tilde{c}\lambda_{N-1}^{\mathbf{L}}})$, we have \eqref{define tilde c}.

\textbf{(ii)}
Next, we analyze the convergence result of UPP-SC. 

By multiplying both sides of \eqref{first-order opt modified} by matrix $\mathbf{J}$, since $\mathbf{J}\mathbf{L}=0$ we obtain $\mathbf{J}\nabla \tilde{f}(\mathbf{x}^{k}) + \mathbf{J}\mathbf{B}^k(\mathbf{x}^{k+1}-\mathbf{x}^{k})=0$, for $k \geq -1.$
From $\mathbf{J}\nabla \tilde{f}(\mathbf{x})=\mathbf{1}_N \otimes\big(\frac{1}{N} \sum_{i=1}^{N}\nabla\tilde{f}_i(x_i)\big)$, $\mathbf{J}\mathbf{J}=\mathbf{J}$, $\|\mathbf{J}\|=1$,
\begin{align}
	&\frac{1}{N}\|\sum_{i=1}^N \nabla \tilde{f}_i(x_i^k)\|^2 \leq (\mathbf{x}^{k+1}-\mathbf{x}^{k})^{\mathsf{T}}\mathbf{B}^k\mathbf{J}\mathbf{J}\mathbf{B}^k(\mathbf{x}^{k+1}-\mathbf{x}^{k}) \notag\\
	\leq & \|\mathbf{x}^{k+1}-\mathbf{x}^{k}\|^2_{\mathbf{B}^k}\|\mathbf{B}^k\| 
	\overset{\eqref{tildePk+1 - tildePk lemma}}{\leq} 4\|\mathbf{B}^k\|(\tilde{P}^k-\tilde{P}^{k+1}). \label{bounded gradient norm Bk}
\end{align}
Since \eqref{d1lambdaNb-d2-d3/lambdaNB} implies that $\lambda_N^{\mathbf{B}}> \frac{d_3}{d_1\lambda_N^{\mathbf{B}}}>\frac{(c_{\mathbf{B}}+\bar{M}^2)\tilde{c}(\tilde{c}+2)}{6d_1\lambda_N^{\mathbf{B}}}$, 
\begin{align}
	\mathbf{B}^k \succeq \lambda_N^{\mathbf{B}}\mathbf{I}_{Nd}&\succeq \frac{(c_{\mathbf{B}}+\bar{M}^2)\tilde{c}(\tilde{c}+2)}{6d_1}(\mathbf{B}^k)^{-1}. \label{Bk geq d3/d1}
\end{align}
It then follows from \eqref{vk+1 - vk} that
\begin{align}
	&\rho\|\mathbf{x}^{k+1}\|^2_\mathbf{L} \notag\\
	\leq& 3\tilde{\kappa}\big((c_{\mathbf{B}}+\bar{M}^2)\|\mathbf{x}^k-\mathbf{x}^{k-1}\|^2_{(\mathbf{B}^k)^{-1}}+\|\mathbf{w}^{k+1}\|^2_{\mathbf{B}^k}\big) \notag\\
	\leq& 3\tilde{\kappa}\big(2(c_{\mathbf{B}}+\bar{M}^2)\|\mathbf{x}^{k+1}-\mathbf{x}^k\|^2_{(\mathbf{B}^k)^{-1}}+\|\mathbf{w}^{k+1}\|^2_{\mathbf{B}^k}\notag\\
	&+ 3(c_{\mathbf{B}}+\bar{M}^2)\|\mathbf{w}^{k+1}\|^2_{(\mathbf{B}^k)^{-1}}\big)\notag\\
	\overset{\eqref{Bk geq d3/d1}}{\leq} & 3\tilde{\kappa}\big(\|\mathbf{x}^{k+1}-\mathbf{x}^k\|^2_{2(c_{\mathbf{B}}+\bar{M}^2)(\mathbf{B}^k)^{-1}}+ \|\mathbf{w}^{k+1}\|^2_{(1+\frac{18d_1}{\tilde{c}(\tilde{c}+2)})\mathbf{B}^k}\big). \notag
\end{align}
This, together with \eqref{Bk geq d3/d1} and \eqref{tildePk+1 - tildePk lemma}, gives
\begin{align}
	&\rho\|\mathbf{x}^{k}\|^2_\mathbf{L}\leq 2\rho\|\mathbf{x}^{k+1}\|^2_\mathbf{L}+2\rho\|\mathbf{x}^{k+1}-\mathbf{x}^k\|^2_\mathbf{L}\notag\\
	\leq & 6\tilde{\kappa}\big((1+\frac{18d_1}{\tilde{c}(\tilde{c}+2)})\|\mathbf{w}^{k+1}\|^2_{\mathbf{B}^k}\notag\\
    &+2(c_{\mathbf{B}}+\bar{M}^2)\|\mathbf{x}^{k+1}-\mathbf{x}^k\|^2_{(\mathbf{B}^k)^{-1}}\big)+2\rho\|\mathbf{x}^{k+1}-\mathbf{x}^k\|^2_\mathbf{L}\notag\\
	\overset{\eqref{Bk geq d3/d1}}{\leq} & \frac{2\tilde{c}}{3}(1+\frac{18d_1}{\tilde{c}(\tilde{c}+2)})\|\mathbf{w}^{k+1}\|^2_{\mathbf{B}^k}+\|\mathbf{x}^{k+1}-\mathbf{x}^k\|^2_{\frac{12d_1}{\tilde{c}(\tilde{c}+2)}\mathbf{B}^k+2\rho \mathbf{L}}\notag\\
	\overset{\eqref{tildePk+1 - tildePk lemma}}{\leq}& \Big(\frac{4}{3(1-\tilde{c})}(1+\frac{18d_1}{\tilde{c}(\tilde{c}+2)})+8(\frac{6d_1}{\tilde{c}(\tilde{c}+2)}+1)\Big)(\tilde{P}^k-\tilde{P}^{k+1}).\notag
\end{align}
Substituting this and \eqref{bounded gradient norm Bk} into the definition of $e(T)$ in \eqref{eT time-varying Bk}, and using \eqref{tildePk+1 - tildePk lemma}, \eqref{hat Pk+1 geq 0 tildePk+1 geq f*} and \eqref{tildeP0 leq fx0+nabla f0} yields
\begin{align}
	&T\times e(T)\overset{\eqref{eT time-varying Bk}}{\leq} \sum_{k=0}^{T-1}(\frac{1}{N}\|\sum_{i=1}^N \nabla \tilde{f}_i(x_i^k)\|^2 + \rho\|\mathbf{x}^{k}\|^2_\mathbf{L}) \notag\\
	&{\leq}  \sum_{k=0}^{T-1}\Big(\frac{4}{3(1-\tilde{c})}(1+\frac{18d_1}{\tilde{c}(\tilde{c}+2)})+8(\frac{6d_1}{\tilde{c}(\tilde{c}+2)}+1)\Big)(\tilde{P}^k-\tilde{P}^{k+1})\notag\\
	&\overset{\eqref{tildePk+1 - tildePk lemma}}{\leq} \tilde{C}_2(\tilde{P}^0-\tilde{P}^{T})\notag\\
    &\overset{\eqref{hat Pk+1 geq 0 tildePk+1 geq f*}}{\leq} \tilde{C}_2(\tilde{P}^0-f^*)\notag\\
	&\overset{\eqref{tildeP0 leq fx0+nabla f0}}{\leq} \tilde{C}_2(\tilde{f}(\mathbf{x}^0)-f^*+\frac{6\tilde{c}+7}{8\bar{M}}\|\nabla \tilde{f}(0)\|^2).  \label{T times eT Bk} 
\end{align}
Therefore, we obtain \eqref{epsilon leq tildeC1C2/T}.

\subsection{Proof of optimal communication complexity bound}\label{optimal convergence}
\subsubsection{Proof of Theorem~\ref{theorem UPP-SC-OPT}}\label{proof theorem UPP-SC-OPT}
Since $\bB^k=\bB=\frac{1}{\mu}\bI-\rho\bL$, the deductions in \eqref{rho xk+1 leq 3 tildekappa} holds for $c_{\bB}=0$. Based on this, the condition in \eqref{Bk+rhoL-c/2Bk-1} becomes $\frac{1-2\tilde{c}}{4}\mathbf{B}+\frac{1}{4}\rho \mathbf{L}-(\frac{1}{2}+2\tilde{c})\bar{M}\mathbf{I}_{Nd}-\frac{\tilde{c}(2+\tilde{c})\bar{M}^2}{6}(\mathbf{B})^{-1} \succeq \mathbf{O}_{Nd}.$
This can be satisfied by $\frac{1-2\tilde{c}}{4}\lambda_N^\bB-(\frac{1}{2}+2\tilde{c})\bar{M}-\frac{\tilde{c}(2+\tilde{c})\bar{M}^2}{6}/\lambda_N^\bB \geq 0.$
From $\tilde{c}\in(0,\frac{1}{2})$, we have $1-2\tilde{c}>0$, together with $\lambda_N^{\mathbf{B}}\geq \frac{d_2+\sqrt{d_2^2+4d_4d_3}}{2d_4} , \text{ with } d_4= \frac{1}{4}- \frac{\tilde{c}}{2}$, we verify the feasibility of \eqref{rho xk+1 leq 3 tildekappa}. The rest proof is the same as the deductions in Appendix~\ref{proof theorem UPP-SC sublinear time-varying Bk}, and thus we omit here.

\subsubsection{Proof of Theorem~\ref{theorem optimal communication complexity bound}}\label{proof theorem optimal communication complexity bound}
First, we state the properties of the Chebyshev polynomial that is similar with Theorem 7 in \cite{xu2020accelerated}. Consider the normalized Laplacian $\mathbf{H}$ with a spectrum in $[1-c_1^{-1},1+c_1^{-1}]$, the Chebyshev polynomial $P_{K}(x)=1-\frac{T_{\tau}(c_1(1-x))}{T_{K}(c_1)}$ introduced in \cite{auzinger2011iterative,arioli2014chebyshev} is the solution of the problem: 
\begin{equation*}
	\min_{p\in \mathbb{P}^{K},p(0)=0} \max_{x\in[1-c_1^{-1},1+c_1^{-1}]} |p(x)-1|. 
\end{equation*}
Accordingly, we have 
\begin{equation*}
	\max_{x\in[1-c_1^{-1},1+c_1^{-1}]} |P_K(x)-1| \leq 2\frac{c_0^K}{1+c_0^{2K}},
\end{equation*}
where $c_0=\frac{1-\sqrt{\gamma}}{1+\sqrt{\gamma}}$. By setting $K=\tau=\lceil \frac{1}{\sqrt{\gamma}} \rceil$, we obtain 
$c_0^K=(\frac{1-\sqrt{\gamma}}{1+\sqrt{\gamma}})^{\lceil \frac{1}{\sqrt{\gamma}} \rceil} \leq (\frac{1-\sqrt{\gamma}}{1+\sqrt{\gamma}})^{\frac{1}{\sqrt{\gamma}}} \leq e^{-1}$. 
It follows that 
\begin{equation}
	2\frac{c_0^K}{1+c_0^{2K}}=\frac{2}{c_0^K+(c_0^{K})^{-1}}\leq \frac{2}{e+e^{-1}}<1.
\end{equation}
Subsequently, we obtain the following property for Chebyshev polynomial: 
$1-\frac{2}{e+e^{-1}}\leq 1-2\frac{c_0^K}{1+c_0^{2K}}\leq \lambda_{N-1}^{P_{\tau}(\mathbf{H})}\leq \lambda_{1}^{P_{\tau}(\mathbf{H})}
    \leq 1+2\frac{c_0^K}{1+c_0^{2K}} \leq 1+ \frac{2}{e+e^{-1}}.$
Then, we have 
\begin{equation}
	\kappa_\mathbf{L}=\frac{{\lambda}_1^{P_{\tau}(\mathbf{H})}}{\lambda_{N-1}^{P_{\tau}(\mathbf{H})}}\leq (\frac{e^{\frac{1}{2}}+e^{-\frac{1}{2}}}{e^{\frac{1}{2}}-e^{-\frac{1}{2}}})^2. \label{underline lambdaL}
\end{equation}
Combining \eqref{theorem UPP-SC-OPT} and $\frac{1}{\mu}>\frac{9\lambda_N^\bB}{\tilde{c}}\Big((2\kappa_\bL-1)+\sqrt{(2\kappa_\bL-1)^2-1}\Big)$, we have $\frac{1}{\mu}= \mathcal{O}(\kappa_\mathbf{L}\bar{M})=\mathcal{O}(\bar{M})$. 
Since UPP-SC-OPT conducts $\lceil \frac{1}{\sqrt{\gamma}}\rceil$ communication rounds at each iteration, we obtain the communication complexity bound as \eqref{Tc leq}.

			\bibliographystyle{IEEEtran}
			\bibliography{reference} 
			
		\begin{IEEEbiography}
			[{\includegraphics[width=1in,height=1.25in,clip,keepaspectratio]{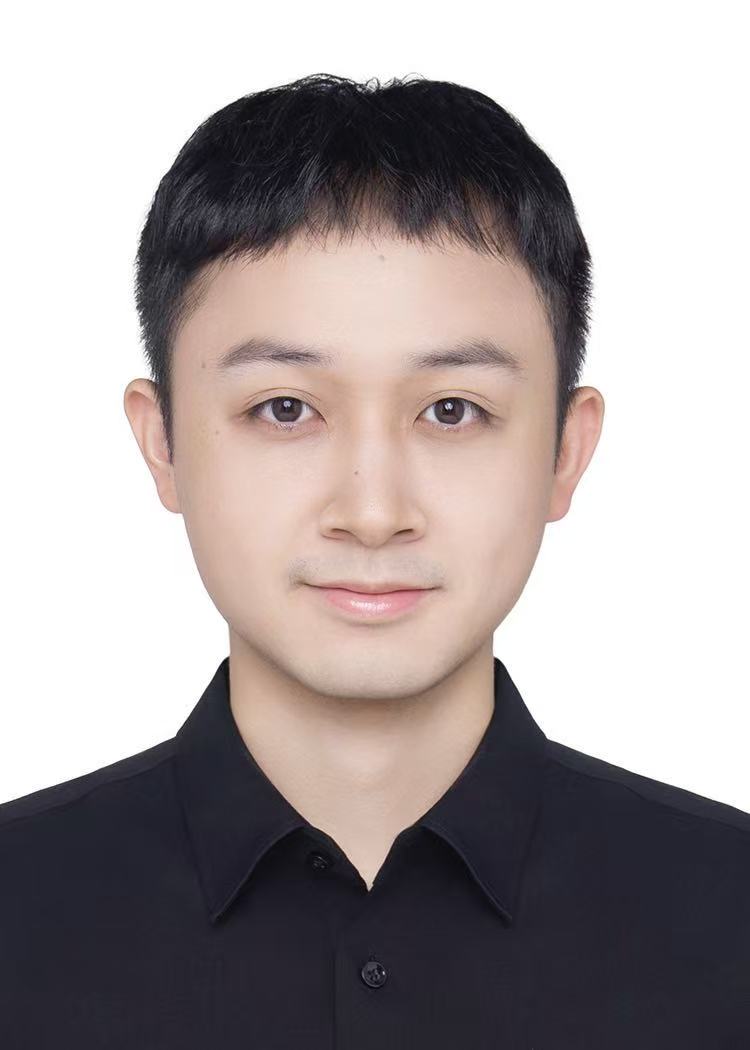}}]{Zichong Ou} received the B.S. degree in Measurement and Control Technology and Instrument from Northwestern Polytechnical University, Xi'an, China, in 2020. He is now pursuing the Ph.D degree from the School of Information Science and Technology at ShanghaiTech University, Shanghai, China. His research interests include distributed optimization, large-scale optimization, and their applications in IoT and machine learning.
		\end{IEEEbiography}
		
		\begin{IEEEbiography}
			[{\includegraphics[width=1in,height=1.25in,clip,keepaspectratio]{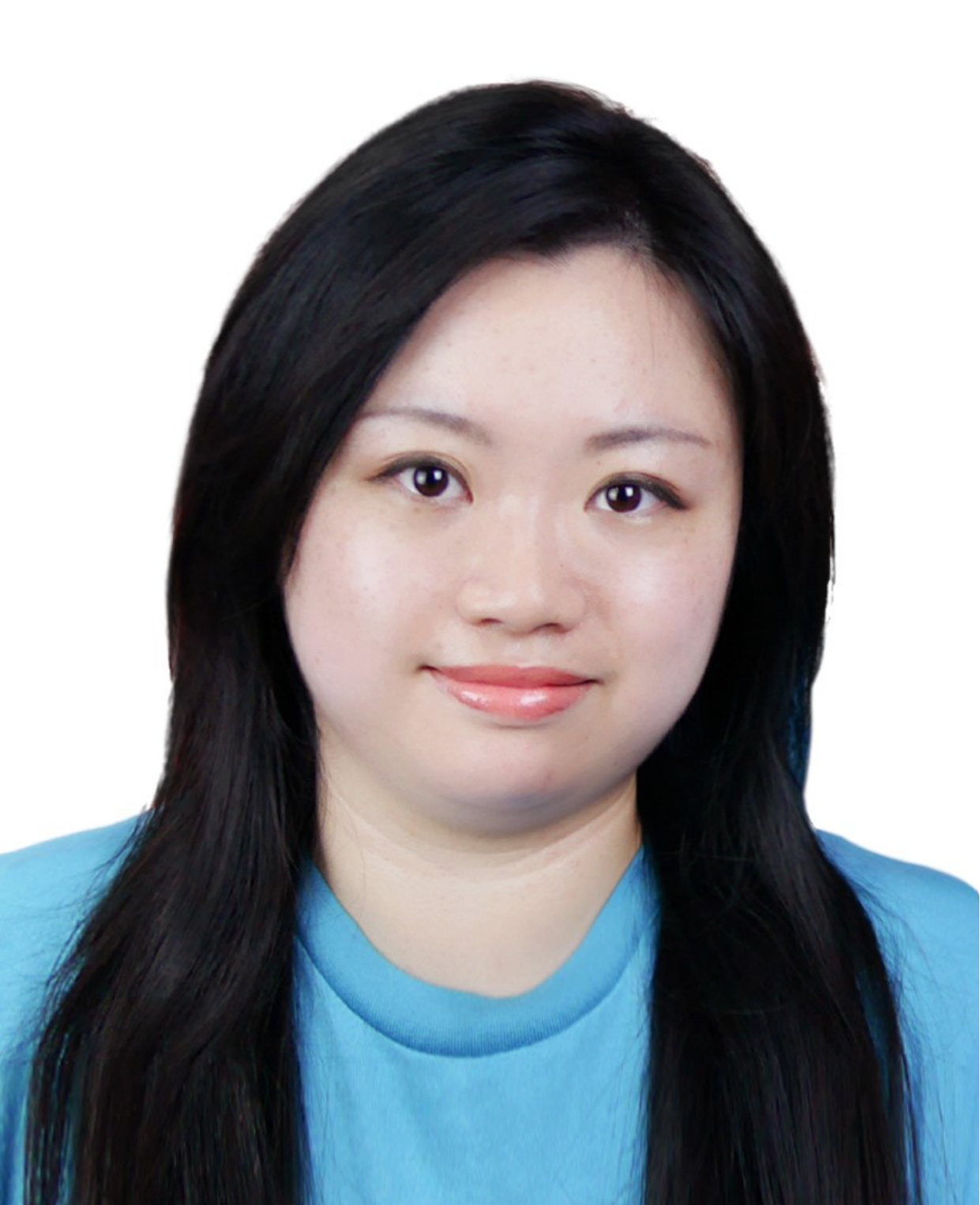}}]
			{Jie Lu} (Member, IEEE) received the B.S. degree in information engineering from Shanghai Jiao Tong University, Shanghai, China, in 2007, and the Ph.D. degree in electrical and computer engineering from the University of Oklahoma, Norman, OK, USA, in 2011. 
			She is currently an Associate Professor with the School of Information Science and Technology, ShanghaiTech University, Shanghai, China. Before she joined ShanghaiTech University in 2015, she was a Postdoctoral Researcher with the KTH Royal Institute of Technology, Stockholm, Sweden, and with the Chalmers University of Technology, Gothenburg, Sweden from 2012 to 2015. Her research interests include distributed optimization, optimization theory and algorithms, learning-assisted optimization, and networked dynamical systems.
		\end{IEEEbiography}

		\end{document}